\documentclass[10pt,reqno]{amsart}
\usepackage[top=.8in, left=.8in, right=.8in, bottom=.8in]{geometry}

\usepackage{savesym}
\savesymbol{checkmark}
\usepackage{nccmath, soul, dingbat}
\allowdisplaybreaks
\numberwithin{equation}{section}
\usepackage{graphicx}
\graphicspath{ {./images/} }

\usepackage[normalem]{ulem}
\makeatletter 
\def\@tocline#1#2#3#4#5#6#7{\relax
	\ifnum #1>\c@tocdepth 
	\else
	\par \addpenalty\@secpenalty\addvspace{#2}%
	\begingroup \hyphenpenalty\@M
	\@ifempty{#4}{%
		\@tempdima\csname r@tocindent\number#1\endcsname\relax
	}{%
		\@tempdima#4\relax
	}%
	\parindent\z@ \leftskip#3\relax \advance\leftskip\@tempdima\relax
	\rightskip\@pnumwidth plus4em \parfillskip-\@pnumwidth
	#5\leavevmode\hskip-\@tempdima
	\ifcase #1
	\or\or \hskip 1em \or \hskip 2em \else \hskip 3em \fi%
	#6\nobreak\relax
	\dotfill\hbox to\@pnumwidth{\@tocpagenum{#7}}\par
	\nobreak
	\endgroup
	\fi}
\makeatother
\usepackage{transparent}
\usepackage{amsmath,amsthm,amssymb}
\usepackage[initials,msc-links]{amsrefs}
\usepackage{url}
\usepackage[dvipsnames]{xcolor}
\usepackage[english=american]{csquotes}
\usepackage{enumerate}
\usepackage{textcomp}
\usepackage{stmaryrd} 
\SetSymbolFont{stmry}{bold}{U}{stmry}{m}{n}
\usepackage{mathrsfs}
\usepackage{mathtools}
\usepackage{color, colortbl}
\definecolor{Gray}{gray}{0.9}
\usepackage{bbm}
\usepackage{hyperref}
\usepackage{stmaryrd}
\usepackage{xfrac}
\usepackage{wrapfig}
\hypersetup{
	unicode=false,          
	pdftoolbar=true,        
	pdfmenubar=true,        
	pdffitwindow=false,     
	pdfstartview={FitH},    
	pdftitle={My title},    
	pdfauthor={Author},     
	pdfsubject={Subject},   
	pdfcreator={Creator},   
	pdfproducer={Producer}, 
	pdfkeywords={keyword1, key2, key3}, 
	pdfnewwindow=true,      
	colorlinks=true,       
	linkcolor=blue ,          
	citecolor=blue ,        
	filecolor=magenta,      
	urlcolor=Aquamarine           
}
\usepackage{xcolor}
\usepackage{epsfig,color,graphicx,graphics}
\usepackage{caption}
\usepackage{amsfonts}
\usepackage{tikz}
\usepackage{siunitx}
\usepackage{booktabs,colortbl, array}
\usepackage{pgfplotstable}
\pgfplotsset{compat=1.8}

\definecolor{rulecolor}{RGB}{0,71,171}
\definecolor{tableheadcolor}{gray}{0.92}

\def\a#1{\llbracket #1 \rrbracket}

\newcommand\restr[2]{{
		\left.\kern-\nulldelimiterspace 
		#1 
		\right|_{#2} 
}}
\newcommand\res{\mathop{\hbox{\vrule height 7pt width .3pt depth 0pt
\vrule height .3pt width 5pt depth 0pt}}\nolimits}

\newtheorem{theorem}{Theorem}[section]
\newtheorem{lemma}[theorem]{Lemma}
\newtheorem{proposition}[theorem]{Proposition}

\theoremstyle{definition}

\newtheorem{remark}[theorem]{Remark}

\newtheorem{question}[theorem]{Question}

\newcommand{\cc}{\subset\!\subset}

\newcommand{\Ecal}{\mathcal{E}}

\newcommand{\Gcal}{\mathcal{G}}
\newcommand{\Hcal}{\mathcal{H}}

\newcommand{\Pcal}{\mathcal{P}}

\newcommand{\Rcal}{\mathcal{R}}

\newcommand{\Xcal}{\mathcal{X}}

\newcommand{\Escr}{\mathscr{E}}

\newcommand{\Gscr}{\mathscr{G}}

\newcommand{\Jbf}{\mathbf{J}}

\newcommand{\Mbb}{\mathbb{M}}

\newcommand{\Sbb}{\mathbb{S}}

\DeclareMathOperator{\co}{co}

\DeclareMathOperator{\graph}{graph}

\DeclareMathOperator{\dist}{dist}

\newcommand{\R}{\mathbb{R}}

\newcommand{\spt}{\mathrm{spt}}

\newcommand{\sing}{\mathrm{sing}}

\newcommand{\eps}{\epsilon}
\newcommand{\cG}{\mathcal{G}}

\newcommand{\LT}{\tilde{\Lambda}}

\newcommand{\proofstep}[1]{\textit{#1}}

\newcommand{\om}{\omega}
\renewcommand{\eps}{\varepsilon}
\newcommand{\e}{\varepsilon}

\DeclareMathOperator{\tr}{tr}

	\DeclareMathOperator{\lens}{lens}
	\DeclareMathOperator{\plane}{plane}
	\newcommand{\X}{\mathcal{X}}

	\title[]{Minimizing clusters with prescribed asymptotic geometry}

	\author[R. Neumayer]{Robin Neumayer}
	\address{Department of Mathematical Scienes, Carnegie Mellon University, 5000 Forbes Avenue, Pittsburgh, PA 15213, United States of America }
	\email{neumayer@cmu.edu}
	
	\author[M. Novack]{Michael Novack}
	\address{Department of Applied Mathematics, Illinois Institute of Technology, Chicago, IL 60616, United States of America}
	\email{mnovack@illinoistech.edu}
	
	\author[A. Skorobogatova]{Anna Skorobogatova}
	\address{Institute for Theoretical Studies, ETH Z\"{u}rich, Scheuchzerstrasse 70
8092 Z\"{urich}, Switzerland}
\email{anna.skorobogatova@eth-its.ethz.ch}

\begin{document}

\begin{abstract}
We construct locally minimizing $(1,2)$-clusters 
whose exterior interfaces are asymptotic to various prescribed singular area-minimizing cones. For  $n+1\leq 7$, Bronsard \& Novack characterized all minimizing $(1,2)$-clusters as standard lenses, whose exterior interface is planar. For $n+1 \in [8,2700]$, the authors together with Bronsard showed the existence of a locally minimizing $(1,2)$-cluster whose exterior interface blows down to {\it some} (unknown, possibly non-unique) singular area-minimizing hypercone. For $n+1=8$, this was shown independently by Novaga, Paolini \& Tortorelli.

Here we develop a refined construction using the Hardt-Simon foliation that realizes prescribed cones. For a singular area-minimizing hypercone $C$ that has an isolated singularity or is cylindrical, we show that if $C$ satisfies an explicit energy bound, then there is a locally minimizing $(1,2)$-cluster whose exterior interface is asymptotic to $C$ with quantitative rates. {In fact, {if $C$ is an area minimizing Lawson cone  satisfying this energy bound,} we produce a countably infinite family of distinct locally minimizing clusters asymptotic to $C$, distinguished by their prescribed asymptotic decay to leading order.}

We verify this energy bound for {the generalized Simons cones $C_{k,k}$ in every even ambient dimension $n+1 = 2k+2\geq 8$, and for the cylindrical cone $C_{3,3}\times\mathbb{R}$ in $\mathbb{R}^9$, where $C_{3,3}$ is the Simons cone}, {therefore answering the cone realization problem in these cases}. {This in particular removes the upper bound of 2700 on the ambient dimension when $n+1$ is even in our preceding work}. 
\end{abstract}

\maketitle

\vspace{-3mm}
\section{Introduction}
\subsection{Overview} 
The classical multiple bubble problem asks for least-area partitions of $\R^{n+1}$ into finitely many chambers of prescribed finite volume together with one infinite volume chamber. It has been the subject of extensive work in recent decades \cite{Alm76, JTaylor76, ColEdeSpo22,FoiAlfBroHodZim93, HutMorRitRos02,ReiHeiLaiSpi03,Rei08, Wichiramala,PaTo20, DRTi23,MilXu25, MilNee23}, including the breakthrough work of Milman \& Neeman \cites{MilNee22}. In recent years, this problem has been extended to clusters with multiple infinite-volume chambers, introducing a new feature absent from the classical multiple bubble problem: the interface between the two infinite chambers may have nontrivial asymptotic geometry.

The simplest instance is the $(1,2)$-cluster problem. A $(1,2)$-cluster is a partition $\mathcal{X}=(\mathcal{X}(1),\mathcal{X}(2),\mathcal{X}(3))$ of $\mathbb{R}^{n+1}$ into three sets of locally finite perimeter, called chambers, with $0<|\mathcal{X}(1)|<\infty$ and $|\mathcal{X}(2)|=|\mathcal{X}(3)|=\infty$. A $(1,2)$-cluster $\X$ is  (locally) minimizing if, for every $\rho>0$,
\[
\Pcal(\Xcal;B_\rho(0))\leq \Pcal (\Xcal'; B_\rho(0))
\]
whenever  $\Xcal'= (\mathcal{X}'(1),\mathcal{X}'(2),\mathcal{X}'(3))$ has $\Xcal(i) \Delta \Xcal'(i) \cc B_\rho(0)$ and $|\Xcal(i)| = |\Xcal'(i)|$ for $i=1,2,3$. Here,
\begin{equation}
	\label{e:perim-cluster}
	\Pcal(\X;B_\rho(x)) := \frac{1}{2}\sum_{i=1}^3 P(\X(i); B_\rho(x))\,,
\end{equation}
where $P(\X(i); B_\rho(x))$ is the relative perimeter of $\X(i)$ in $B_\rho(x)$. The $(1,2)$-cluster problem, and the more general $(M,N)$-cluster problem, was first introduced by Alama, Bronsard \& Vriend \cite{AlaBroVri23} in relation to the structure tri-block copolymers in the regime when one phase is infinitesimal in volume relative to the other two. There have since been numerous works concerning the existence, classification and stability of such clusters \cite{AlaBroLuWan22, AlaBroLuWan25, BN,BonCriTop25, MilXu25, BNNS, NovPaoTor25,NovPaoTor23}. 

When $n+1\leq 7$, Bronsard and the second author \cite{BN} showed that  there is a unique locally minimizing $(1,2)$-cluster in $\R^{n+1}$ up to homotheties: the \emph{standard lens cluster}.  See also \cite{AlaBroVri23} when $n=1$. This cluster $\X_{\lens}$ is characterized by the properties $|\X_{\lens}(1)|=1$, $\partial\X_{{\lens}}(2)\cap \partial\X_{{\lens}}(3) \subset \{x_n=0\}$ and $\partial \X_{ {\lens}}(1)$ consists of a pair of spherical caps of equal radii meeting on $\{x_n=0\}$ with equal angles of $\frac{2\pi}{3}$ between the three interfaces.

This uniqueness for $n+1 \leq 7$ is closely related to the fact that planes are the only minimal hypercones in these dimensions. Given the existence of non-planar minimizing hypercones for $n+1\geq 8$ \cite{BDGG}, two natural questions arise. The first raised in was raised in \cite{BN} and \cite{NovPaoTor23}:
\begin{question}\label{q:nonunique}
   For $n+1\geq 8$, do there exist locally minimizing $(1,2)$-clusters besides the standard lens? 
\end{question}
Since the exterior interface $\partial\X(2)\cap\partial\X(3)$ of a locally minimizing $(1,2)$-cluster necessarily blows down to a  (possibly non-unique) area-minimizing hypercone, a closely related question is the {\it cone realization problem}:
\begin{question}\label{q:prescribed-bd}
	Given an area-minimizing hypercone $C$ in $\R^{n+1}$, is there a locally minimizing $(1,2)$-cluster whose exterior interface blows down to $C$?
\end{question}

Our previous work with Bronsard \cite{BNNS} addressed Question~\ref{q:nonunique}. We showed that for $8 \leq n+1 \leq 2700$, uniqueness of locally minimizing $(1,2)$-clusters fails by demonstrating the existence of local minimizers whose exterior interfaces blow down to \emph{singular} area-minimizing hypercones and are thus not lenses. An analogous result was independently and simultaneously proven by Novaga, Paolini \& Tortorelli \cite{NovPaoTor25} for the case $n+1=8$ with a similar proof. 

These constructions, however, do not identify the blow-down cones and thus yield no information about Question~\ref{q:prescribed-bd} for any given area-minimizing cone. Indeed, in these results, nothing could be said about the singular blow-down cones apart from an upper bound on their density. Thus, even after the affirmative answer to Question~\ref{q:nonunique} for $8 \leq n+1 \leq 2700$,  the only area-minimizing cone known to arise as a blow-down of the exterior interface of a minimizing $(1,2)$-cluster was the plane, realized by the standard lens cluster.

Moreover, the blow-down cones corresponding to the  minimizing $(1,2)$-clusters constructed in these works were not shown to be unique, preventing any possibility of establishing rates of convergence or providing any further description of the shape of the minimizing $(1,2)$-clusters. And, until the present work, Question~\ref{q:nonunique} remained open in dimensions $n+1>2700.$

A main goal of this work is to address Question~\ref{q:prescribed-bd} for a large family of  area-minimizing cones, including both cylindrical cones and cones with isolated singularity, and to give a more precise description of the minimizing $(1,2)$-clusters produced by the construction. In doing so, we  also answer Question~\ref{q:nonunique} in the affirmative in every even ambient dimension $n+1\geq 8$.

The clusters we construct have exterior interfaces that converge quantitatively to the prescribed cone $C$ at infinity. In the isolated-singularity case, this asymptotic control also yields smoothness of the exterior interface outside a sufficiently large ball. In fact, for those cones we can treat that have an isolated singularity, we produce countably many distinct minimizing $(1,2)$-clusters, {distinguished by their leading order of decay to $C$}, standing in contrast to the unique minimizing $(1,2)$-cluster that is asymptotic to the plane (namely, the lens).

\subsection{Main results}
For an area-minimizing cone $C$ in $\R^{n+1}$ that either has an isolated singularity or is cylindrical, we reduce the cone realization problem to the validity of a strict inequality for the normalized energy
\begin{equation}\label{e: LambdaC}
\Lambda_C := \inf\left\{P(E) - \mathcal{H}^n(C \cap E^{(1)}) \ : \ \text{$E$ is a bounded set of finite perimeter with $|E|=1$}\right\}\,.
\end{equation}
We write $\Lambda_{\plane}(n+1)$ to denote $\Lambda_{C}$ when $C$ is a hyperplane in $\R^{n+1}$; this is explicitly given by
\[
\Lambda_{\plane}(n+1) = P(\X_{\lens}(1)) -\omega_{n}\rho_{n+1}^{n}
\]
where $\rho_{n+1}$ is the radius of the disc $\X_{\lens}(1) \cap \{x_n=0\}$.
Our first main result provides a positive answer to Question \ref{q:prescribed-bd}, yielding prescribed conical blowdowns with a rate, for cones with isolated singularities and cylindrical cones satisfying $\Lambda_C < \Lambda_{\rm plane}(n+1)$. Recall that a (multiplicity one) cone $C$ is \emph{cylindrical} if, up to rotation, it is of the form $C= C' \times \R^{{m}}$ for a cone $C'$ with an isolated singularity. We refer the reader to Section \ref{s:prelim} for preliminaries about positive Jacobi fields on $C$.
\begin{theorem}\label{thm:main thm}
	If $C\subset \mathbb{R}^{n+1}$ is an area-minimizing hypercone that either has an isolated singularity or is cylindrical, and
	\begin{align}\label{eqn: strict inequ thm 1 statement}
		\Lambda_C < \Lambda_{\rm plane}(n+1)\,,
	\end{align}
	then there exists a locally minimizing $(1,2)$-cluster $\mathcal{X}$ such that the blowdown of $\partial \X(2) \cap \partial \X(3)$ is $C$. Moreover, letting $\gamma_1^\pm$ be the homogeneities of the positive, radially homogeneous, decaying Jacobi fields on $C$, there exist $R_0>0$,  $C_0{>0},$ and $C_1 \geq 0$ 
such that for $r\geq R_0$,
    \begin{equation}\label{e:HS-slow-decay}
        \dist_\Hcal(\partial \X(2) \cap \partial \X(3)\cap B_r^c, C\cap B_r^c) \leq r^{-\gamma_1^-}(C_0 + C_1\log r)\,.
    \end{equation}
    or
    \begin{equation}\label{e:HS-fast-decay}
        \dist_\Hcal(\partial \X(2) \cap \partial \X(3)\cap B_r^c, C\cap B_r^c) \leq C_0 r^{-\gamma_1^+}\,,
    \end{equation}
    where $\dist_\Hcal$ denotes Hausdorff distance.
\end{theorem}
\begin{remark}\label{rmk: decay}
    The decay rates \eqref{e:HS-slow-decay}, \eqref{e:HS-fast-decay} arise from the decay of positive Jacobi fields on a cone $C$ with isolated singularity, which determine the asymptotic behavior at infinity of {the leaves of the Hardt-Simon foliation {(and thus $\partial\X(2)\cap \partial\X(3)$)}; see Section \ref{subsec:iso sing prelim} below}. Note that the decay \eqref{e:HS-slow-decay} occurs when $C$ is strictly minimizing, with $C_1=0$ if and only if $C$ is strictly stable, namely $\lambda_1 > -(\tfrac{n-2}{2})^2$.
\end{remark}
\begin{remark}
When $C$ is a cone with isolated singularity satisfying the assumptions of Theorem~\ref{thm:main thm}, then the decay estimates \eqref{e:HS-slow-decay},\eqref{e:HS-fast-decay} together with De Giorgi's $\e$-regularity guarantee that for some $R_1>0$, the exterior interface $\partial\X(2)\cap \partial\X(3)$ is smooth outside of $B_{R_1}(0)$.
\end{remark}

For any area-minimizing hypercone $C$ that has any isolated singularity or is cylindrical,
Theorem~\ref{thm:main thm} in particular gives an affirmative answer to Question~\ref{q:prescribed-bd} provided $C$ satisfies the  strict inequality 
\begin{equation}\label{e:Lambda-cone-vs-Lambda-lens}
	\Lambda_C < \Lambda_{\plane}(n+1)\,.
\end{equation}
Thus the next major question becomes whether \eqref{e:Lambda-cone-vs-Lambda-lens} holds for given cones.
In our previous work \cite{BNNS}, we verified \eqref{e:Lambda-cone-vs-Lambda-lens} for certain  \emph{quadratic} (or Lawson) cones
\begin{equation}\label{eqn: quadratic cones}
    C_{k,l}=\{(x,y) \in \R^{k+1}\times \R^{l+1} : n{+1}=k+l+2, \ |x|^2 = \tfrac{k}{l} |y|^2\}\,,
\end{equation}
in dimensions $8 \leq n+1 \leq 2700$.
There, the proof in $n+1=8$ is done through explicit by-hand computation, while in higher dimension the proof is computer-assisted via fully rigorous interval arithmetic.
 Our next main result verifies \eqref{e:Lambda-cone-vs-Lambda-lens} for generalized Simons cones in every even dimension $n+1 \geq 8$ and for the cylindrical cone $C_{3,3}\times \R$.

\begin{theorem}\label{t:even-dim-nonuniqueness}
    The {strict inequality \eqref{e:Lambda-cone-vs-Lambda-lens} and thus the} conclusions of Theorem \ref{thm:main thm} hold for $C = C_{k,k}$ when $n{+1}=2k+2$ for any $n{+1} \geq 8$, and for $C_{3,3} \times \R$ in $\R^9$, the latter providing a first example of a locally minimizing $(1,2)$-cluster modeled on a cone with a non-isolated singularity.
\end{theorem}

Observe that Theorem~\ref{thm:main thm2} gives a positive answer to Question~\ref{q:nonunique} in every even ambient dimension $n+1\geq 8.$
\begin{remark}\label{rmk: lawson}
    We believe that the conclusion of Theorem \ref{t:even-dim-nonuniqueness} should remain true for the quadratic cones $C_{k,l}$ with {$l=k+1$} in odd dimensions $n+1 {= k+l+2} \geq 9$, using the competitors constructed in \cite{BNNS} and the method of proof of Theorem~\ref{t:even-dim-nonuniqueness}. Since our primary goal is to demonstrate the method {that allows us to remove the unnecessary upper bound on ambient dimension in our previous argument}, we do not pursue this here.
\end{remark}

Our final main result produces an abundance of distinct minimizing $(1,2)$-clusters that are asymptotic to {area-minimizing Lawson cones}. This includes one minimizer whose exterior interface decays to $C$ at a {\it  strictly faster} rate than the rate guaranteed by \eqref{e:HS-slow-decay} in Theorem~\ref{thm:main thm}, and countably many clusters with precisely the decay \eqref{e:HS-slow-decay} with distinct leading-order coefficients. This flexibility is very much in contrast to the symmetry and rigidity properties exhibited by local minimizers both in lower dimensions for this problem, and in the classical multiple bubble problem. 

\begin{theorem}\label{thm:main thm2}
	Suppose $C_{k,l}\subset \mathbb{R}^{n+1}$ {is an area-minimizing Lawson cone} satisfying
	\begin{align*}
		\Lambda_{C_{k,l}} < \Lambda_{\rm plane}(n+1)
	\end{align*}
Then there are infinitely many distinct minimizing $(1,2)$-clusters $\X$ with $|\X(1)|=1$ that are asymptotic to ${C_{k,l}}$. 

More precisely,  there is a minimizing $(1,2)$-cluster $\X$ with $|\X|=1$ with 
\begin{equation} 
{\rm dist}_{\mathcal{H}}(\partial \X(2)\cap \partial \X(3) \cap  B_{r}^c, C_{k,l} \cap B_r^c) = 
 {\rm o}(r^{-\gamma_1^-})
\end{equation}
and a countable collection of non-zero numbers $\{a_i\}\subset\R$, accumulating to $0$, and corresponding 
local minimizers $\X_i$ with $|\X_i(1)|=1$
such that
\begin{equation}\label{eq:prescribed decay rate in theorem}
{\rm dist}_{\mathcal{H}}(\partial \X_i(2)\cap \partial \X_i(3) \cap  B_{r}^c, {C_{k,l}} \cap B_r^c) =
   a_i|x|^{-\gamma_1^-} + {\rm o}(r^{-\gamma_1^-})
\end{equation}
\end{theorem}

{In view of Theorem~\ref{t:even-dim-nonuniqueness} and our earlier work \cite{BNNS}, Theorem \ref{thm:main thm2} applies to all generalized Simons cones $C_{k,k}$ for $2k+2=n+1\geq 8$ and to all quadratic cones $C_{k,k+1}$ for $2k+3=n+1 \in \{9,\dots, 2699\}$.}

\subsection{Outline of proofs}

The starting point for the proofs of Theorems~\ref{thm:main thm} and \ref{thm:main thm2} is a refinement of the construction we introduced in \cite{BNNS}. The new idea is to use the leaves of the Hardt-Simon foliation of the prescribed cone as barriers, forcing the minimizing $(1,2)$-clusters we construct to have the prescribed asymptotic behavior. 

Let $C$ be a fixed singular area-minimizing hypercone in $\R^{n+1}$ that has isolated singularity (the cylindrical case is similar), and satisfies \eqref{e:Lambda-cone-vs-Lambda-lens}. For a fixed large scale $R>0$, we use a local variational problem to produce a $(1,2)$-cluster $\X_R$ with $|\X_R(1)|=1$ that {locally minimizes the cluster perimeter $\mathcal{P}$ (up to a small error)} among volume-constrained competitors in $B_{3R}$ and whose exterior interface coincides with $C$ outside $B_{4R}$. For a sequence of such scales $R_k \nearrow +\infty,$ we let $\X_k=\X_{R_k}$ be a corresponding sequence of such minimizers. 

One would like to consider the limit of the $\X_k$, but compactness issues arise as the finite-volume chamber $\X_k(1)$ could lose mass in the limit. A concentration-compactness argument shows that, up to a subsequence,  $\X_k$ splits into finitely many (independent of $k$) asymptotically non-interacting concentrations $\{\X_{k,i}\}_{i=1}^N$. After translating by some vector $x_{k,i}$, each concentration converges to a locally minimizing $(1,2)$-cluster $\X_{\infty,i}$ in $\R^{n+1}$ whose finite-volume chamber $\X_{\infty,i}(1)$ has a definite amount mass.

At this stage, there is in principle enormous flexibility for the geometry of the exterior interface for each limiting concentration $\X_{\infty,i}$. A blow-down of the exterior interface of $\X_{\infty, i}$ could be
a plane, e.g. if the concentration escapes along an asymptotically flat region and resembles a standard lens. It could be the prescribed cone $C$, e.g. if 
if the corresponding translations $x_{k,i}$ remain bounded. Or it could any other area-miniminzing hypercone with density at most that of $C$, e.g. 
if the translations $x_{k,i}$ tend to infinity but $|x_{k,i}|/R_k\to 0$ and the exterior interface slowly changes geometry at intermediate scales.\footnote{
In our prior work \cite{BNNS}, at this point in the  construction, we ruled out the possibility that  {\it every} concentration has a planar blow-down  by using \eqref{e:Lambda-cone-vs-Lambda-lens}, the characterization of minimizing $(1,2)$-clusters with planar area growth  \cite{BN}, and the concavity of the scaling. Hence there was at least one concentration blowing down to some singular cone. While this affirmatively answered Question~\ref{q:nonunique} in any dimension for which \eqref{e:Lambda-cone-vs-Lambda-lens} holds for some cone, little can be said about the geometry of the resulting minimizing $(1,2)$-cluster.
}

In order to prove Theorem~\ref{thm:main thm}, let $T_\lambda$ be the Hardt-Simon foliation associated to $C$; see section~\ref{s:prelim}. By construction, for each $k\in \mathbb{N}$, the finite-volume chamber $\X_k(1)$ and the exterior interface $\partial\X_k(2)\cap \partial \X_k(3)$ of $\X_k$ are trapped between two leaves $T_{\pm\lambda}$ of the foliation for sufficiently large $\lambda>0$, so let $\lambda_k$ be the smallest $\lambda>0$ for which this is true. If $\lambda_*:=\liminf \lambda_k$ is finite, this means that (up to a subsequence), the finite-volume chambers $\X_k(1)$ are trapped between the leaves $T_{\pm2\lambda_*}$ for all $k \in \mathbb{N}$. Together with uniform density estimates for $\X_k(1)$, 
this prevents loss of mass and gives compactness of the full sequence. The limiting $(1,2)$-cluster then satisfies the conclusions of Theorem~\ref{thm:main thm}.

On the other hand, if $\lambda_*=+\infty$, then the sequence is not  trapped between a fixed pair of leaves of the foliation and instead some concentration $\X_{k,i}$ has center $x_{k,i}$ drifting off to infinity. But, this means the sequence of translated nearest leaves $T_{\lambda_k}-x_{k,i}$  converges locally smoothly to a plane. The resulting limiting cluster $\X_{\infty, i}$ has exterior interface lying to one side of a plane, and hence, by the maximum principle and the characterization in \cite{BN}, is a standard lens cluster (and the mass $v_1$ of its finite-volume chamber has a definite lower bound). Going back to sequence $\X_{k}$, we excise this lens-like piece by hand to obtain competitors $\X_k'$ for a family of minimization problems as above but with $|\X_k'(v)|=1-v_1$. This reduction can be repeated. Each time $\lambda_*=+\infty,$ we extract and remove a definite amount of mass escaping as an asymptotic lens. Since the extracted volumes have a uniform positive lower bound, the process terminates after finitely many steps. The assumption \eqref{e:Lambda-cone-vs-Lambda-lens} and the concavity of the scaling ensure that some definite amount mass remains after the lens-like pieces are removed; this final sequence must have $\lambda_*<\infty$, allowing us to conclude Theorem~\ref{thm:main thm} as above.

Theorem~\ref{thm:main thm2} is based on a similar construction, but the initial family of variational problems is modeled on a fixed leaf $T_\lambda$ of the Hardt-Simon foliation of the cone $C$ rather than on $C$ itself. Then, a competitor argument and the fact that {it is too expensive to transition from the existing leaf $T_\lambda$ to any other leaf in the foliation (including the cone $C_{k,l}=T_0$ itself) forces the exterior interface to behave exactly as $T_\lambda$ to lowest order; see Lemma \ref{l:area-gap-Plateau-leaf 2}. Here we are relying heavily on the work \cite{SimonSolomon}, which allows us to characterize the asymptotic behavior of the interface $\partial\X(2)\cap\partial\X(3)$ as that of a leaf at lowest order.} 

Finally, the proofs of both parts of Theorem~\ref{t:even-dim-nonuniqueness} build on competitor constructions from our previous paper with Bronsard \cite{BNNS}. There, we constructed competitors $E_C$ for the variational problem $\Lambda_C$ defined in \eqref{e: LambdaC} for generalized Simons cones $C=C_{k,k}$ in $\R^{n+1}$ for even ambient dimensions $n+1=2k+2$ (and more generally for Lawson cones in all dimensions). In order to show \eqref{e:Lambda-cone-vs-Lambda-lens} for these generalized Simons cones $C$ (and thus the first part of Theorem~\ref{t:even-dim-nonuniqueness}, it suffices to show $\mathcal{E}(E_C)<\Lambda_{\rm plane} (n+1)$ where we let  $\mathcal{E}(E)=P(E) - \mathcal{H}^n(C \cap E^{(1)})$. Since $\mathcal{E}(E_C)$ can be expressed explicitly as sums, products, and quotients of dimensionally-dependent integrals of single variable functions, we are able to carefully estimate $\mathcal{E}(E_C)$  using Laplace's method with quantitative error terms to establish the inequality analytically. To prove the second part of Theorem~\ref{t:even-dim-nonuniqueness}, that is, that \eqref{e:Lambda-cone-vs-Lambda-lens} holds for $C=C_{3,3}\times \R,$ we construct a competitor $E_C$ for $\Lambda_C$ as a warped product of the competitor $E_{C_{3,3}}$ from our previous paper and compute $\mathcal{E}(E_C)$ using the coarea formula. 

In both parts of Theorem~\ref{t:even-dim-nonuniqueness}, we expect the methods here could be generalized to other cones, e.g. to other Lawson cones (c.f. Remark~\ref{rmk: lawson}) or cylinders over other generalized Simons cones. In view of Theorem~\ref{thm:main thm}, doing so would solve the cone realization problem, Question~\ref{q:prescribed-bd}, in these cases. We leave these generalizations as an interesting open problem.

\medskip 

\noindent{\bf Acknowledgements.} R.N. is grateful for the support of NSF CAREER grant DMS-2340195 and NSF RTG grant DMS-2342349. A.S. is grateful for the generous support of Dr.~ Max R\"ossler, the Walter Haefner Foundation and the ETH Z\"urich Foundation. This research was partially conducted during the period A.S. served as a Clay Research Fellow. {ChatGPT was used to produce initial versions of estimates in the proofs of Theorems 1.6 and Lemma 5.1. The final versions have been written by the authors and no AI-generated text appears in the article.}

\section{Notation and Preliminaries}\label{s:prelim}
We use the notation $P(E)$ to denote the perimeter of a set $E\subset \R^{n+1}$ of finite perimeter, and $P(E; U)$ to denote the relative perimeter of a set $E$ of locally finite perimeter in an open set $U \subset \R^{n+1}$. We additionally use the notation $\partial^* E$ for the reduced boundary of such a set, {and we will always assume that $E$ satisfies $\overline{\partial^* E} = \partial E$; see \cite[Proposition 12.19]{MaggiBook}.} We point the reader to \cite{MaggiBook} for all relevant background on sets of finite perimeter.

Let $C$ be an area-minimizing hypercone in $\R^{n+1}$. By this we mean that $C=\partial E$ for a conical set of locally finite perimeter $E\subset\R^{n+1}$ with $P(E;U) \leq P(F;U)$ for any open, bounded set $U\subset \R^n$ and any set $F\subset \R^{n+1}$ with $E\Delta F \Subset U$. By conical, we mean that $\iota_{0,r}(E) = E$ for the map $\iota_{0,r}(y) := \frac{y}{r}$.\footnote{Recall, thanks to identification between codimension one boundaryless integral currents and boundaries of sets of locally finite perimeter in $\R^{n+1}$, that this is equivalent to $C$ being an area-minimizing hypercone in the framework of integral currents.}

{Recall that we have monotonicity of surface area ratios $r \mapsto \frac{\Hcal^{n}(\Sigma)}{r^{n}}$ for any area-minimizing hypersurface $\Sigma$ in $\R^{n+1}$. This in particular implies that if $\Sigma$ is in addition entire, any subsequential limit\footnote{{This is made sense of weakly as currents, or equivalently, boundaries of sets of locally finite perimeter.}}
\[
    C= \lim_{j \to \infty} \frac{\Sigma\cap B_{\rho_j}}{\rho_j}\,, \rho_j \uparrow \infty
\]
is an area-minimizing hypercone satisfying
	\begin{equation}
		\label{eqn: density at infinity}
		\frac{\Hcal^n(C\cap B_\rho)}{\omega_n \rho^n} = \Theta_\Sigma := \lim_{R\uparrow \infty} \frac{\Hcal^n(\Sigma \cap B_R)}{\omega_n R^n} \qquad \forall \rho > 0\,.
	\end{equation}   
Thanks to an exterior version (see e.g. \cite{BNNS}*{Lemma A.1}) of the classical monotonicity formula for area-minimizing hypersurfaces, any blowdown
\[
    K = \lim_{j \to \infty} \frac{\X(2)\cap B_{\rho_j}}{\rho_j}\,, \qquad \rho_j \uparrow \infty
\]
of a minimizing $(1,2)$-cluster $\X$ yields a set of locally finite perimeter $K$ such that $\partial K$ is an area-minimizing hypercone.}

\subsection{{Cones with isolated singularities}}\label{subsec:iso sing prelim}

{In this subsection $C$ has an isolated singularity.} Recall that Hardt \& Simon proved the following.
\begin{theorem}[\cite{HardtSimon84}*{Theorem 2.1}]\label{t:Hardt-Simon}
	Let $C$ be an area-minimizing hypercone with $\sing(C) =\{0\}$ in $\R^{n+1}$ and let $K_+, K_-$ be connected components of $\R^{n+1}\setminus \overline{C}$, so that $K_{\pm}$ are local perimeter minimizers. There exists a unique {open} set $F_{\pm} \subset K_\pm$ such that $F_{\pm}$ is a local perimeter minimizer, $\partial F_{\pm}$ is analytic, asymptotic to $C$ at infinity, and $\dist(\partial F_\pm, \{0\})= 1$.
	
	In particular, $\{\lambda \, \partial F_\pm\}_{\lambda>0}$ yields a (unique) foliation of $K_\pm$ by analytic area-minimizing hypersurfaces.
\end{theorem}

For simplicity of notation, we will henceforth write $T_{\pm\lambda} = \lambda \partial F_\pm$ for $\lambda >0$ and $T_0 = C$, and $F_{\pm \lambda}$ for the sets of locally finite perimeter $\lambda F_\pm$. In other words, we may write $\{T_\lambda\}_{\lambda \in \R}$ for the foliation of $\R^{n+1}$ associated to $C$, with $T_0=C$, which is singular only at $\lambda=0$, and $\{F_{\lambda}\}_{\lambda\in \R}$ for the corresponding family of sets of locally finite perimeter. {Note that the subfamilies $\{F_\lambda\}_{\lambda>0}$ and $\{F_\lambda\}_{\lambda<0}$ are nested. We will henceforth refer to this foliation as the \emph{Hardt-Simon foliation} associated to $C$.}

\begin{remark}
	We will only make use of the existence and regularity part of Theorem \ref{t:Hardt-Simon}. We do not require the uniqueness, although this is the part of the conclusion that is the most difficult to generalize to area-minimizing cones without isolated singularities, and remains open in general; see for instance \cite{Simon_Liouville, Wang, EdSz}.
\end{remark}

Recall (see e.g. \cite{HardtSimon84} or \cite{Simon_entire_MSE}) that stability guarantees the existence of positive Jacobi fields (namely solutions of $\Delta_{C} f + |A_C|^2 f = 0$) given in polar coordinates by $f_\pm(\xi,r) = \phi_1(\xi) r^{-\gamma_1^\pm}$  on $C\setminus \{0\}$, for $\phi_1: C\cap\partial B_1\to (0,\infty)$ a positive first eigenfunction of $\Delta_{C\cap \partial B_1} + |A_{C\cap \partial B_1}|^2$ {with eigenvalue $\lambda_1 \geq -(\tfrac{n-2}{2})^2$} and $\gamma_1^\pm$ solving
\[
\gamma_1^\pm(n-2-\gamma_1^\pm) = -\lambda_1\,,
\]
i.e.
\[
\gamma_1^\pm = \frac{n-2}{2} \pm \sqrt{\left(\frac{n-2}{2}\right)^2 + \lambda_1}\,.
\]
{We normalize $\phi_1$ to be positive with $\|\phi_1\|_{L^2(\partial B_1)}=1$. After this normalization, note} that if $C$ is \emph{strictly stable}, i.e. $\lambda_1 > - (\tfrac{n-2}{2})^2$, then there will be two distinct positive Jacobi fields $f_\pm$, and
$\gamma_1^- < \tfrac{n-2}{2} < \gamma_1^+$. Otherwise, if $\lambda_1 = - (\tfrac{n-2}{2})^2$, there will be just one such positive Jacobi field, and $\gamma_1^\pm = \tfrac{n-2}{2}$.

We recall in addition (see, for instance, \cite{HardtSimon84}*{Theorem 3.2}) that $C$ is said to be \emph{strictly minimizing} if\footnote{In light of \cite{HardtSimon84}, this is equivalent to saying that there exists $\kappa>0$ such that for any $R>0$ and any area-minimizing hypersurface (i.e. codimension 1 integral current) $T$ with $\partial T = \partial C$ in $\overline{B_R}$ satisfying $\spt T \subset B_\rho^c$, one has $\Hcal^{n}(C\cap B_R)\leq \Hcal^{n}(T\cap B_R) - \kappa \rho^{n}$.} there exists $c_0(n,C)>0$ and $\rho_0(n,C,\lambda)>0$ such that for any $\rho_0 < R$ the leaves $T_{\pm\lambda}$ of the Hardt-Simon foliation given by Theorem \ref{t:Hardt-Simon} satisfy
\begin{equation}\label{e:strictly-min-strictly-stable}
    T_{\pm \lambda} \setminus B_{R} = c_0 \lambda \graph_C \big(\phi_1 \cdot |x|^{-\gamma_1^-} + {\rm o}\big(|x|^{-\gamma_1})\big)
\end{equation}
in the case that $C$ is strictly stable, and
\begin{equation}\label{e:strictly-min-non-strictly-stable}
    T_{\pm \lambda} \setminus B_{R} = c_0 \lambda\graph_C \Bigg(\phi_1 \cdot {|x|^{-\tfrac{n-2}{2}}} \log |x| + {\rm o}\Big({|x|^{-\tfrac{n-2}{2}}\log |x|}\Big)\Bigg)
\end{equation}
otherwise (i.e. if $\lambda_1 = -(\tfrac{n-2}{2})^2$). {Here, we are using the notation
    \begin{equation}\notag
      \graph_C u=\{x+u(x)\nu_C(x) : x\in C\setminus\{0\} \}\,.  
    \end{equation}
    for a function $u:C \setminus\{0\}\to \mathbb{R}$, where $\nu_C$ denotes a unit normal to $C\setminus \{0\}$ with fixed orientation.}

{Instead, if $C$ is not strictly minimizing, then 
\begin{equation}\label{e:non-str-min}
    T_{\pm\lambda}\setminus B_R = c_0\lambda \graph_C(\phi_1 \cdot |x|^{-\gamma_1^+})  + {\rm o}(|x|^{-\gamma_1^+})\,,
\end{equation}
see \cite[(1.9)]{HardtSimon84}. 
\subsection{{Cylindrical cones}}\label{subsec:cylindrical prelim}
{If $C=C' \times \mathbb{R}^m$ where $C'$ is an area minimizing cone with isolated singularity, we will (in a slight abuse of notation) use $F_{\pm \lambda}$ and $T_{\pm \lambda}=\partial F_{\pm \lambda}$ to refer to the cylindrical extension to $\mathbb{R}^n$ of subsets of $\mathbb{R}^m$ associated to $C'$ and defined in the previous subsection.}

\subsection{Currents and the Plateau problem} We will be considering the Plateau problem with boundary data that is graphical over $C$. {We adopt the notation of \cite{Simon_GMT}, including the mass $\Mbb(T)$ of a current $T$, the multiplicity one current $\llbracket S \rrbracket$ associated to an oriented submanifold, and the spherical radius $s$ slice $\langle T,|\cdot|,s \rangle$ of a current $T$}. To set notation, we fix a normal vector $\nu_C$ for $C$ and consider the associated current $\a{C}$. We also orient the leaves $T_\lambda$ of the foliation so that $\a{T_\lambda}/r \to \a{C}$ as $r\to \infty$. Given a function $h: C \cap \partial B_s\to \mathbb{R}$, the current $\a{\graph_{C\cap \partial B_s} h}$ is taken to have the orientation that agrees with the slice $\langle \a{C},|\cdot|,s \rangle$ when $h=0$. Graphs over $C$ are similarly oriented.

\section{Compactness revisited}\label{s: compact}
In this section, we revisit the compactness procedure laid out in our previous paper \cite{BNNS}. There, we fixed an area-minimizing hypercone $C$ and considered a sequence of minimization problems $\mu_{R_k}$ ``modeled on $C$'' for clusters with a volume constraint prescribing the first chamber to have volume $1.$ We used the resulting sequence of minimizers of $\mu_{R_k}$ to construct locally minimizing $(1,2)$-clusters whose exterior interfaces blow down to singular cones. To carry out the more refined analysis of the present paper, we need to generalize this setup slightly in two directions. First, we consider minimization problems modeled on a fixed
area-minimizing hypersurface, not necessarily a cone, and second, {given a range of volume constraints $v\in [v_0,1]$ for the $\mu_{R_k}$ problems, we obtain a lower bound depending only on $v_0>0$ for the volume of the first chamber of any limiting local minimizer. Neither of these modifications plays a decisive role in the compactness procedure. The only change is ensuring that lower volume density estimates for the first chamber (Lemma \ref{lem: lower density} below) are uniform across $v\in [v_0,1]$, which in turn is a consequence of uniform estimates in the concentration compactness-type result \cite[Lemma 29.10]{MaggiBook} under lower perimeter and volume bounds on the sets in question. Thus, we recall the main steps from \cite{BNNS} and provide a sketch of how the proofs therein adapt to the present setting.}

Throughout this section, $\Sigma$ is a fixed, possibly singular, entire area-minimizing hypersurface in $\R^{n+1}$ on which all estimates will depend.  
Let
\[
\Lambda_{\Sigma} = \inf\{P(E) - \mathcal{H}^n (\Sigma \cap E^{(1)}) \ : \ E \text{ is a bounded set of finite perimeter with}\, |E| = 1\}. 
\]
Here, by bounded, we mean that there exists $R>0$ such that $E\subset B_R.$

Let $F_\Sigma$ be one of the two sets of locally finite perimeter in $\R^{n+1}$ with $\partial F_\Sigma = \Sigma$.  For a fixed number $R>0$, define the confinement functional
\begin{equation}\label{eq:gR def}
	\cG_R(E) := \int_{E} g_R(|y|) \, dy\,   \qquad \text{where} \qquad     g_R(t) := \begin{cases}
		0 & t \in [0, \sqrt{R}),\\
		\frac{t- {\sqrt{R}}}{\sqrt{R}} & t \geq {\sqrt{R}} \,
	\end{cases}
\end{equation}
and consider the energy functional
\[
\Ecal_R(\X) = \Pcal(\Xcal; B_{4R}) + \Gcal_R(\Xcal(1)) 
\]
for a $(1,2)$-cluster $\X$.
Let $\bar{R}\geq 1$ be such that $|B_{\bar{R}}|\geq 2.$
For $v \in(0,1]$ and $R\geq \bar{R}$, consider the variational problem
\begin{equation}
	\label{eqn: tau problem}
	\tau_{R,\Sigma}(v) = \inf \left\{\Ecal_R(\X) :\ |\Xcal(1)|=v, \ \Xcal(2)\setminus B_{3R} = F_\Sigma\setminus B_{3R},
	\  \Xcal(3)\setminus B_{3R} = F_\Sigma^c\setminus B_{3R} \right\}\,.
\end{equation}
It is not difficult to check that the function $R\mapsto \tau_{\Sigma,R}(v) - \mathcal{H}^n(\Sigma\cap B_{4R})$ is nonnegative (by the minimality of $\Sigma$) and nonincreasing (thanks to nested competitor sets and the fact that $g_{R_2}-g_{R_1} \leq 0$ for $R_2\geq R_1$), and thus the asymptotic normalized energy
\begin{equation}\label{eqn: asymptotic normalized energy}
	\LT_\Sigma(v) = \lim_{R\to \infty} (\tau_{\Sigma,R}(v) - \mathcal{H}^n(\Sigma\cap B_{4R}))
\end{equation}
is well defined. Observe that $\LT_\Sigma(1)\leq \Lambda_\Sigma.$
Note that the constraints in \eqref{eqn: tau problem} ensure that any competitor $\X$ has $\X(1)\subset B_{3R}$, and the direct method guarantees the existence of a minimizer $\X_R^v$ of $\tau_{R,\Sigma}(v)$ for each $v\in(0,1]$ and $R \geq \bar{R}$. As in \cite[Section 3]{BNNS}, a basic competitor argument yields some key initial estimates for such a minimizer. Let $\X'$ be the $(1,2)$-cluster where $\X'(1)$ is a volume-$v$ ball $B^v$ centered at the origin, $\X'(2) = F_\Sigma \setminus \X'(1)$, and $\X'(3) = F_\Sigma^c\setminus\X'(1)$. Taking $\X'$ as a competitor in \eqref{eqn: tau problem}, we find
\begin{equation}
	\label{eqn: bound 1}
	\Pcal(\X_R^v;B_{4R}) \leq \Ecal_R(\X_R^v ) \leq \Ecal_R(\X')  \leq P(B^v) + P(F_\Sigma; B_{4R}).
\end{equation}
Since the local perimeter minimality of $F_\Sigma$ means that $ P(\X_R^v(h); B_{4R}) \geq P(F_\Sigma; B_{4R})$ for $h=2,3$,  we in turn find that
\begin{equation} \label{eqn: bound 3}
	P(\X_R^v(1)) +2 \Gcal_R(\X_R^v(1))  = 2 \Ecal_R(\X_R^v) - P(\X_R^v(2); B_{{4R}})- P(\X_R^v(3); B_{4R}) \leq 2 P(B^v)\,.
\end{equation}
Since $g_R(|y|) \geq s/\sqrt{R}$ for any $y \in \mathbb{R}^{n+1}\setminus B_{\sqrt{R}+s}$, the bound  \eqref{eqn: bound 3} guarantees that 
\begin{equation}\label{eqn: measure bound}
	|\X_R^v(1) \setminus B_{\sqrt{R} + s} | \leq \frac{\sqrt{R} P(B^v)}{s}
\end{equation}
for any $s>0$. Using \eqref{eqn: measure bound} and arguing as in \cite[Lemma 4.1]{BNNS} shows that for $v$ bounded uniformly away from $0$, the first chambers of minimizers of $\tau_{R,\Sigma}(v)$ satisfy the following uniform lower volume density estimates.
\begin{lemma}\label{lem: lower density}
	Fix $v_0 \in (0,1]$ and $R\geq \bar{R}$. 
	There is a constant $c_0 = c_0(v_0,n,\Sigma)>0$ such that if $\X_R^v$ is a minimizer of $\tau_{R,\Sigma}(v)$ with $v \in [v_0,1]$, then 
	\begin{align}
		\frac{|\X_R^v(1) \cap B_\rho(y)|}{\omega_{n+1} \rho^{n+1}} \geq c_0 \label{e:lower-density-1}
	\end{align} 
	for each $x \in\partial \X_R^v(1)$ and $\rho \in (0,1].$
\end{lemma}
\begin{proof}
	The proof is similar to \cite[Lemma 4.1]{BNNS}; we sketch the proof here for completeness.
	Suppose by way of contradiction that \eqref{e:lower-density-1} does not hold. Then we may find sequences $\{R_k\}$, $v_k$, and $\rho_k$ and minimizers $\X_{k} = \X_{R_k}^{v_k}$ of $\tau_{R_k,\Sigma}(v)$ such that 
	\begin{equation}
		\label{eqn: density contra}
		\rho_k^{-(n+1)}{|\X_k(1) \cap B_{\rho_k}(y)|} \to 0\,.
	\end{equation}
	Take $c_{n+1}>0$ to be the dimensional constant from the nucleation lemma \cite[Lemma 29.10]{MaggiBook} and fix $0< \e < \min\{v_0, \frac{P(B^{v_0})}{2(n+1)c_{n+1}}\}$. Let $\bar{N}=(\frac{2P(B^1)}{c_{n+1}\e})^{n+1}$. Applying the nucleation lemma to each $\X_k(1)$ together with the bound \eqref{eqn: bound 3}, we obtain points $\{x_{k,1},\dots , x_{k,N_k}\}$ with $N_k \leq v_k(\frac{P(\X_k(1))}{c_{n+1}\e})^{n+1} \leq \bar{N}$, such that $|\X_k(1) \setminus \cup_{i=1}^{N_k} B_1(x_{k,i})| <\e$ and 
	\[
	|B_1(x_{k,i})\cap \X_k(1) | \geq \Big(\frac{c_{n+1}\e}{P(\X_k(1))}\Big)^{n+1} \geq \bar{N}^{-1} \quad \text {for each}\quad i=1,\dots , N_k.
	\]
	Applying \eqref{eqn: measure bound} with $s=s_k = 2\sqrt{R_k}P(B^1) \bar{N}$ ensures that $B_1(x_{k,i})\subset B_{\sqrt{R_n} + s_k +2}$ for each $i=1\dots, N_k,$ and in particular $g_k(|x_{k,i}|)$ is bounded above independently of $i$ and  $k$.
	From here, by repeating Step 2 and Step 3 of \cite[Proof of Lemma 4.1]{BNNS} (which in turn are adaptations of arguments in the proofs of \cite[Lemma 17.2]{MaggiBook} and \cite[Lemma 4.1]{BN}), we perform volume fixing variations and use the classical differential inequality argument to show there exist $\rho_1$ and $c_1$, independent of $k$, such that $|\X_k(1) \cap B_\rho(y) | \geq c_1\rho^n$ for each $k \in \mathbb{N}$, $y \in \partial \X_k(1)$, and $\rho \in (0,\rho_1).$ Up to replacing $c_1$ by $c_1\rho_1^{n+1}$, we see that the same estimate extends to all $\rho \in (0,1]$.  This contradicts \eqref{eqn: density contra} and completes the proof.
\end{proof}

From Lemma~\ref{lem: lower density}, we can show uniform density estimates for minimizers of $\tau_{R,\Sigma}(v)$ with $v \in [v_0,1]$ and that, for a fixed $v$, along (sub)sequences $R_k \to \infty,$ the first chambers $\X_k^v(1)$ concentrate in a uniform number of well-separated balls.
\begin{lemma}\label{lem: density and concentrations}
	Fix $v_0 \in (0,1]$,  $v \in [v_0,1]$, and $R\geq \bar{R}$. 
	There are constants $c_1, C_0>0$ and $N_0 \in \mathbb{N}$ depending on $v_0$, $n$, and $\Sigma$ such that the following holds. 
	Any minimizer $\X_R^v$ of $\tau_{R,\Sigma}(v)$ has
	\begin{equation}\label{eqn: chamber 1 in ball sqrt R}
		\X_R^v(1) \subset B_{C_0\sqrt{R}}
	\end{equation}
	and,  for each  $h \in \{1,2,3\},$ $y \in \partial \X_R^v(h) \cap B_{2R}$, and $\rho \in (0,1]$, satisfies the density estimates
	\begin{align}
		c_1   \leq \frac{ |\X_R^v(h) \cap B_\rho(y)|}{\omega_{n+1}\rho^{n+1}  } \leq 1-c_1 \,.\label{e:lower-density-2-3}
	\end{align} 
	Moreover, consider a sequence $\X_k$ of minimizers to $\tau_{R_k, \Sigma}(v)$ with $R_k \to \infty$. Up to a subsequence, there exist $R_0\geq \bar{R}$,
	$N \leq N_0$ and points $\{x_{k,1},\dots ,x_{k,N}\} \subset B_{C_0\sqrt{R}_k}$ such that 
	\begin{equation}\label{eqn: contained in balls}
		\X_k(1) \subset \bigcup_{i=1}^N B_{R_0}(x_{k,i}),
	\end{equation} 
	$B_{R_0}(x_{k,i})$ intersects $\X_{k}(1)$ nontrivially, and  $\lim_k |x_{k,i} - x_{k,j}| = +\infty$. 
\end{lemma}
\begin{proof}
	Let $\X_R^v$ be a minimizer of $\tau_{R,\Sigma}(v)$.
	Using the uniform lower density estimate \eqref{e:lower-density-1} and the nucleation lemma \cite[Lemma 29.10]{MaggiBook}, we repeat the argument of \cite[Lemma 4.1, Step 4]{BNNS} essentially verbatim to obtain points $\{y_1, \dots, y_M\}$ with $M\leq N_0$ such that 
	\begin{equation}\label{eqn: contained in B4s}
		\{y_i\}_{i=1}^M \subset B_{C_1\sqrt{R}}\qquad \text{ and  }\qquad\X(1) \subset \bigcup_{i=1}^M B_4(y_i),
	\end{equation}
	where  $C_1>0$ and $N_0 \in \mathbb{N}$ depend on $n$, $v_0$, and the constant $c_0$ from \eqref{e:lower-density-1}, and thus on $n,v_0,$ and $\Sigma$. 
	
	This in particular guarantees \eqref{eqn: chamber 1 in ball sqrt R} with $C_0 = C_1 + 4\bar{R}^{-1/2}$. With \eqref{eqn: chamber 1 in ball sqrt R} in hand, the proof of \cite[Corollary 4.2]{BNNS} can be repeated verbatim to obtain the density estimates \eqref{e:lower-density-2-3}.

	Next, take a sequence $\X_k^v$ of minimizers to $\tau_{R_k, \Sigma}(v)$ with $R_k \to \infty$ and the corresponding points $\{y_{k,1}, \dots, y_{k,M_k}\}$ with $M_k\leq N_0$ satisfying \eqref{eqn: contained in B4s} with $\X_k^v$ in place $\X$, we can repeat the argument of \cite[Lemma 4.4]{BNNS} to obtain the final conclusion of the lemma.  
\end{proof}

Now, for the remainder of the section, fix $ v_0\leq v\leq 1$, a sequence $R_k\to \infty$, and a corresponding sequence of minimizers $\X_k$ of $\tau_{R_k,\Sigma}(v)$. Suppose  that we have passed to a subsequence (without relabeling) to obtain $R_0$, $N$, and $\{x_{k,1}, \dots, x_{k,N}\}$ as in Lemma~\ref{lem: density and concentrations}.
For each $k \in \mathbb{N}$ and each $i=1,\dots, N,$ define the {\it concentration} $\X_{k,i}$ to be the translated $(1,2)$-cluster
\begin{equation}\label{eqn: concentration}
	\X_{k,i} = \X_k - x_{k,i}\,. 
\end{equation}
As a direct corollary of the uniform lower density estimates of Lemma~\ref{lem: lower density} is a lower bound on the volume of the first chamber of any concentration:
\begin{lemma}\label{lem: vol lb}
	There exists $\bar{v} = \bar{v}(v_0, n, \Sigma)>0$ such that $|\X_{k,i}(1) \cap B_{R_0}| \geq \bar{v}$ for each $i= 1, \dots N$ and  $k\in \mathbb{N}$.
\end{lemma}
\begin{proof}
	Since $R_0\geq \bar{R}$ and $B_{R_0}$ intersects $\X_{k,i}(1)$ nontrivially by Lemma~\ref{lem: density and concentrations}, there is some $x \in \partial \X_{k,i}(1) \cap B_{R_0}$. By Lemma~\ref{lem: lower density} applied with $\rho =1$, we know that $|\X_{k,i}(1)\cap B_1(x)| \geq c_0\omega_{n+1}$. Since we can moreover assume (up to a subsequence) that the balls $B_{R_0+1}(x_{k,i})$ are disjoint, we know from \eqref{eqn: contained in balls} that $\X_{k,i}(1)\cap B_1(x) \subset B_{R_0}$. Hence the lemma holds with $\bar{v}= c_0\omega_{n+1}.$
\end{proof}
The concentrations $\X_{k,i}$ converge to limiting clusters that are themselves locally minimizing:
\begin{lemma}\label{lem: limit local min}
	For each $i=1,\dots, N,$ there exist partitions $\X_{\infty,i}= (\X_{\infty,i}(1), \X_{\infty, i}(2), \X_{\infty, i}(3))$ {of $\R^n$} by sets of locally finite perimeter with the following properties:
	
	\smallskip
	
	\noindent (i) setting $v_i:=|\X_{\infty,i}(1)|$, we have $v_i \geq \bar{v}$ (with $\bar{v}$ as in Lemma~\ref{lem: vol lb}) and
	\begin{align}\label{eq:vi sum to 1}
		\sum_{i=1}^N v_i = v\, ;
	\end{align}
	
	\noindent (ii) for each $h=1,2,3$, $\X_{k,i}(h)$ and $\partial \X_{k,i}(h)$ converge locally in the Hausdorff distance to $\X_{\infty,i}(h)$ {and $\partial \X_{\infty,i}(h)$ respectively};
	
	\noindent (iii) {
		if $|\X_{\infty,i}(2)|= |\X_{\infty,i}(3)|=\infty,$} then $\X_{\infty,i}$ is a locally minimizing $(1,2)$-cluster with $\X_{\infty,i}(1)\subset B_{R_0}$.   
	If $|\X_{\infty,i}(2)|$ (resp. $|\X_{\infty,i}(3)|$) is finite, then {$\X_{\infty,i}(2)$ (resp. $\X_{\infty,i}(3)$) is empty and}
	\begin{align}\label{eq:floater lower bound}
		\frac{1}{2}\sum_{j=1}^3 P(\X_{\infty,i}(j);B_{2R_0}) \geq (n+1)\omega_{n+1}^{1/(n+1)} v_i^{n/(n+1)}\,.
	\end{align}
	
\end{lemma}
\begin{proof}
	All but {the claim that $v_i \geq \bar{v}$ in (i) and the second claim of (iii)} are direct analogues of those in  \cite[Lemma 4.5]{BNNS} and the proofs are the same. The additional statement that $v_i \geq \bar{v}$ follows from Lemma~\ref{lem: vol lb}. {The additional statement that the finiteness of  $|\X_{\infty,i}(2)|$ implies that $\X_{\infty,i}(2)$ is empty can be deduced as follows. Suppose not. Arguing just as in  \cite[Lemma 4.5(iii)]{BNNS} shows that $\X_{\infty,i}(2)$ is locally perimeter minimizing in $\R^{n+1}\setminus B_{R_0}$. Thus $\X_{\infty,i}(2)$ satisfies volume density estimates in $\R^{n+1}\setminus B_{R_0+1}$ and a standard argument (as in the proof of \eqref{eqn: contained in balls}) shows that $\X_{\infty,i}(2)$ is bounded; say $\X_{\infty, i}(2)\subset B_s$ for some $s>0$. Then by the Hausdorff convergence of the chambers and their boundaries, $\X_{k,i}(2)$ has non-empty intersection with $B_{2s}(x_{k,i})$ and has empty intersection with the annulus $B_{4s}(x_{k,i})\setminus B_{2s}(x_{k,i})$. Taking the competitor $\X'_k$ in $\tau_{\Sigma,R_k}$ {for $R_k > s$} with $\X_k'=(\X_k(1), \X_k(2)\setminus B_{3s}(x_{k,i}),  \X_k(3) \cup (\X_k(2) \cap B_{3s}(x_{k,i}))$ violates the minimality of $\X_k$, a contradiction.}
\end{proof}

Applying Lemma~\ref{lem: limit local min} and repeating the arguments of \cite[Sections 5 and 6]{BNNS} produces a locally minimizing $(1,2)$-cluster modeled on {\it some} singular area minimizing cone:
\begin{proposition}\label{prop: summary from last paper}
	Suppose $\Lambda_{\Sigma} < \Lambda_{plane}$ and $v\in (0,1]$.  For some $i \in \{1,\dots, N\}$,  $\X := X_{\infty,i}$ is a locally minimizing $(1,2)$-cluster and any blowdown of $\partial^*\X(2) \cap \partial^*\X(3)$ is a singular area-minimizing cone with density at most $\Theta_\Sigma$.
\end{proposition}
Recall $\Theta_\Sigma$ was defined in \eqref{eqn: density at infinity}. In the sections that follow, we give a refined analysis that produces---for suitably chosen $\Sigma$---a locally minimizing $(1,2)$-cluster whose (unique) blowdown is precisely the blowdown of $\Sigma.$

\section{Proof of Theorem \ref{thm:main thm}}\label{sec: construction}
Throughout this section, we fix  an area-minimizing hypercone $C$ as in Theorem \ref{thm:main thm}, {so that $C$ either has an isolated singularity or is cylindrical}, and satisfies
\[
\alpha_0 := \Lambda_{\plane}(n+1) - \Lambda_C >0.
\]
Let $\{T_\lambda\}_{\lambda\in \R}$ be {the foliation associated with $C$ as defined in Sections \ref{subsec:iso sing prelim}-\ref{subsec:cylindrical prelim}, with the convention that $T_0=C$.} Observe that  there exists $\lambda_0>0$  (depending on $C$) such that for every $\lambda$ with $|\lambda| \leq \lambda_0$ we have
\begin{equation}\label{eqn: Lambda gap}
	\Lambda_{\plane}(n+1) - \Lambda_{T_\lambda} =: \alpha_\lambda \geq \frac{\alpha_0}{2} > 0\,.
\end{equation}
For the remainder of the section, we let $\lambda \in [-\lambda_0, \lambda_0]$ be fixed. Up to swapping the labels of the components $K^\pm$ of $\R^{n+1}$ with boundary $C$, we can and will assume that $\lambda\geq 0.$

We use the shorthand  $\tau_{R,\lambda}(v) =  \tau_{R,T_\lambda}(v)$ for the variational problem \eqref{eqn: tau problem} introduced in the previous section, i.e. 
\begin{equation}
	\label{eqn: tau lambda problem}
	\tau_{R,\lambda}(v) = \inf \left\{\Ecal_R(\X) :\ |\Xcal(1)|=v, \ \ \Xcal(2)\setminus B_{3R} = F_\lambda\setminus B_{3R},
	\ \  \Xcal(3)\setminus B_{3R} = F_\lambda^c\setminus B_{3R} \right\}\,.
\end{equation}

\begin{lemma}\label{lem: trapping}
	Fix $v_0 \in (0,1]$ and $\lambda_- <0\leq \lambda <\lambda_+$. There exist $\bar\rho,\hat{R}>0$ depending on $v_0, n, C, \lambda$, and $\lambda_\pm$ such that if $v \in [v_0,1]$, $R\geq \hat{R}$, and $\X$ is a minimizer of $\tau_{R,\lambda}(v)$ satisfying
	\begin{equation}\label{eqn: hp compactly contained}
		\X(1) \Subset \R^{n+1}\setminus (\overline{F}_{\lambda_+} \cup \overline{F}_{{\lambda_-}})\,,
	\end{equation}
	then there is $\overline{\rho}>0$ such that $\X(1) \subset B_{\bar\rho}{(0)}$ {if the singularity of $C$ is isolated and $\X(1)\subset B_{\overline{\rho}}(0)\times \mathbb{R}^m$ if $C$ is cylindrical} and 
	\begin{equation}\label{eqn: contain conclusion}
		F_{\lambda_+} \subset \X(2) \qquad  \text{ and }\qquad F_{\lambda_-} \subset \X(3)\,.
	\end{equation}
	In particular, $\partial\X(2) \cap \partial\X(3) \subset \R^{n+1}\setminus (F_{\lambda_+} \cup F_{{\lambda_-}}).$ 
\end{lemma} 
\begin{remark}
	\rm{
		The assumptions that $\lambda_-<0$ and $\lambda_+>0$ are simply made to streamline the statement using the notation for the sets $F_{{\lambda_+}}$. While the statement above is sufficient for our needs, it is immediate from the proof that the analogous result holds when $0\leq \lambda_- <\lambda <\lambda_+$.
	}
\end{remark} 
\begin{proof}
	To see the first conclusion, let $\bar{R}$ be the dimensional constant from Section~\ref{s: compact}, assume
	$R \geq \bar{R}$ and fix $x \in \partial \X(1).$ By Lemma~\ref{lem: lower density}, $|\X(1)\cap B_1(x)| \geq c_0\omega_{n+1}$ where $c_0= c_0(v_0, n, C, \lambda)$. On the other hand, in the case that $C$ has an isolated singularity, the assumption \eqref{eqn: hp compactly contained} guarantees that $|\X(1)\cap B_1(x) | \leq |B_1(x) \setminus (\overline{F}_{\lambda_+} \cup \overline{F}_{{\lambda_-}})|$, which tends to zero as $|x|\to \infty$. Hence $|x|<\bar\rho$ for a constant $\bar\rho$ depending on $v_0, n, C, \lambda, \lambda_+$, and $\lambda_-$. Similarly,
    if $C=C'\times \R^m$ is cylindrical, then {we still have $|\X(1)\cap B_1(x)| \leq |B_1(x)\setminus (\overline{F}_{\lambda_+} \cup \overline{F}_{{\lambda_-}})|$, and this still
    tends to zero as $|x|\to \infty$ for any $x\in \R^{n+1}$ lying outside {a tubular neighborhood of the} linear subspace $\{0\}\times \R^m$}, so $\X(1)\subset B_{\overline{\rho}}(\{0\}\times \R^m)$ where $\overline{\rho}$ has the same dependencies. 

    Next we show \eqref{eqn: contain conclusion}.
	We may assume $P(\X(3); T_{\lambda_+}) =P(\X(3); T_{\lambda_-}) = 0$; otherwise apply the following argument along a sequence $\lambda_+^j\downarrow \lambda_+$ and $\lambda_-^j \uparrow \lambda_-$ where this holds. We also assume for simplicity that $\X(1),\X(2)$ are open, which we may do up to modifying on a Lebesgue null set using minimality. Note that $\X(3)$ has trivial intersection with $F_{\lambda_+}\setminus \overline{B_{3R}}$ thanks to the constraints of the minimization problem \eqref{eqn: tau lambda problem} and the fact that $F_{\lambda_+}\subset F_{\lambda}$. We claim that $\X(3)$ has trivial intersection with $F_{\lambda_+} \cap \overline{B_{3R}}$ as well, thus establishing the first conclusion of \eqref{eqn: contain conclusion}.

Indeed, consider the $(1,2)$-cluster $\Xcal' =(\X'(1), \X'(2), \X'(3))$ with 
	\[
	\X'(1) = \Xcal(1), \quad \X'(2) =  \Xcal(2)\cup F_{\lambda_+},\quad \X'(3) = \Xcal(3) \setminus F_{\lambda_+}\,.
	\]
Taking $\X'$ as a competitor for $\tau_{R,\lambda}(v)$ and using that $\X(1)=\X'(1)$, we find
\[
\mathcal{H}^{n}(\partial\X(2)\cap \partial \X(3)) \leq \mathcal{H}^{n}(\partial\X'(2)\cap \partial \X'(3))
\]
which in particular shows that 
\begin{equation}\label{eqn: may31}
P(\X(3); F_{\lambda_{+}}\cap B_{4R}) \leq P(T_{\lambda_+}; \X(3) \cap B_{4R}).
\end{equation}
On the other hand, consider the set $G:=F_{\lambda_+}\setminus (\X(3) \cap {B}_{3R})$ Since $\X(3)\setminus B_{3R} \subset  F_{\lambda_+}^c \setminus B_{3R} $, we see that $G=F_{\lambda_+}\setminus \X(3)$. Using that $F_{\lambda_+}$ is locally perimeter minimizing, \cite[Theorem 16.3]{MaggiBook}, and \eqref{eqn: may31}, we have 
\begin{align*}
P(F_{\lambda_+}; B_{4R}) \leq P(G; B_{4R})& = P(F_{\lambda_+} ; B_{4R}\setminus \X(3)) + P(\X(3) ; F_{\lambda_+} \cap B_{4R})\\
& \leq P(F_{\lambda_+} ; B_{4R}\setminus \X(3)) +P(T_{\lambda_+}; \X(3) \cap B_{4R}) = P(F_{\lambda_+};B_{4R}) \,.
\end{align*}
So, $G$, along with $F_{\lambda_+}$, is locally perimeter minimizing in $B_{4R}$, and $G\subset F_{\lambda_+}$. This, however, violates the maximum principle. Indeed, by construction, there is a point $x \in \partial G\cap \partial F_{\lambda_+}\cap B_{4R}$. Then \cite[Lemma 37.10]{Simon_GMT} or \cite[Theorem 2.2]{Sternberg1992} (the latter of which is a consequence of \cite[Corollary 1]{simon1987strict}) guarantees $\partial G$ and $\partial F_{\lambda_+}$ coincide in a neighborhood of $x$. Since $S:=\partial F_{\lambda_+} \cap B_{4R} = T_{\lambda_+} \cap B_{4R}$ is connected, this means the set $A:=S \cap \partial G$ is a relatively open and closed subset of $S$, hence $A=T_{\lambda_+}\cap B_{4R}$. In other words, $G= F_{\lambda_+}$, i.e. $F_{\lambda_+}\cap \X(3)$ is empty, establishing the claim. 
	
	Repeating this argument with $F_{\lambda_-}$  in place of $F_{\lambda_+}$ and with the roles of $\X(2)$ and $\X(3)$ swapped shows that $\X(2)$ has trivial intersection with $F_{\lambda_-}$. From the boundary data of \eqref{eqn: tau lambda problem}, we then deduce that $F_{\lambda_-} \subset \X(3)$. This completes the proof.
\end{proof}

Now, to set up the statement of the next lemma and the constructions giving rise to Theorems~\ref{thm:main thm} and \ref{thm:main thm2}, fix $0<v_0 \leq v  \leq 1$, and choose a sequence $R_k \to \infty.$ Let $\X_k$ be corresponding minimizers of $\tau_{R_k, \lambda}(v)$. Pass to a subsequence, without relabeling, to obtain $R_0$, $N$, and $\{x_{k,i}\}_{i=1}^N$ as in Lemma~\ref{lem: density and concentrations}, so that in particular 
\[
\X_k(1) \subset \bigcup_{i=1}^N B_{R_0}(x_{k,i})
\]
by \eqref{eqn: contained in balls}.
{For $i=1,\dots,N$, let $v_i=|\X_{\infty,i}(1)|$ for the limits $\X_{\infty,i}$ of} the corresponding concentrations \eqref{eqn: concentration} as obtained in Lemma~\ref{lem: limit local min}.
For each $k \in \mathbb{N}$ and $i=1,\dots, N$, choose $\lambda_{k,i} \in \R$ so that 
\[
B_{R_0}(x_{i,k}) \subset \R^{n+1}\setminus F_{\lambda_{k,i}} \qquad \text{ and } \qquad \text{dist}(x_{i,k}, T_{\lambda_{k,i}}) = R_0+1.
\]
After passing to a subsequence and re-indexing in $i$, we may assume $|\lambda_{k,1}| = \max_{1\leq i\leq N}\{ |\lambda_{k,i}| \}$ for each $k$, and that each $\lambda_{k,1}$ has the same sign. Without loss of generality we may assume $\lambda_{k,1}\geq 0$. Define the constant $\lambda_*\in \R_+ \cup\{+\infty\}$ 
by $ \lambda_*:=\liminf_{k\to \infty}\, \lambda_{k,1}$.We pass to a further subsequence, without relabeling, so that this $\liminf$ is a limit, i.e.
\begin{equation}\label{eqn: lambda star def}
	\lambda_*:=\lim_{k\to \infty}\, \lambda_{k,1}.
\end{equation}
It follows from the first conclusion of Lemma~\ref{lem: trapping} that $\lambda_*<\infty$ if the sequence $\{x_{k,1}\}$ is bounded, while $\lambda_*=+\infty$ if the sequence is unbounded. 
We prove the following dichotomy using Lemma~\ref{lem: trapping}. The argument is similar to the proof of \cite[Theorem 1.2]{BNNS}.  Here $\bar\rho>0$ is the constant from Lemma~\ref{lem: trapping}.
\begin{lemma}   \label{lemma: dichotomy}
	The following dichotomy holds:\\
	
	$\bullet$ If $\lambda_*< +\infty$, then $\X_k(1) \subset B_{\bar\rho}{(0)}$ {if the singularity of $C$ is isolated and $\X_k(1)\subset B_{\overline{\rho}}(0)\times \mathbb{R}^m$ if $C$ is cylindrical} and 
	\begin{equation}\label{eqn: all mass stays}
		 F_{2\lambda_*} \subset \X_k(2), \qquad \text{and}\qquad F_{-2\lambda_*} \subset \X_k(3) 
	\end{equation}
	for all $k$ sufficiently large. {Thus $\X_{\infty,1}(1) \subset B_{\overline{\rho}}(0)$, $F_{2\lambda_*} \subset\X_{\infty,1}(2)$, and $F_{-2\lambda_*} \subset \X_{\infty,1}(3)$.}

	$\bullet$ If $\lambda_*=+\infty$, then 
	\begin{align}
		\LT_\Sigma(v) \geq \Lambda_{\rm plane}(n+1)v_1^{n/(n+1)} + \LT_\Sigma(v-v_1)\,.
	\end{align}
	Here $v_1 \geq \bar{v}$ is first-chamber volume of the limit $\X_{\infty,1}$ {and $\LT_\Sigma(v)$ is the asymptotic renormalized energy defined in \eqref{eqn: asymptotic normalized energy}. We adopt the convention that $\LT_{T_\lambda}(0)=0.$}
\end{lemma}

\begin{proof}
	If ${\lambda}_* <+\infty,$  then for all $k$ large enough, the clusters $\X_k$ satisfy the assumption \eqref{eqn: hp compactly contained} with $\lambda_{\pm} = \pm 2{\lambda}_*.$ The conclusion \eqref{eqn: all mass stays} then follows directly from Lemma~\ref{lem: trapping}. The remaining conclusions are then immediate consequences of Lemma~\ref{lem: limit local min}.
	
	Next suppose ${\lambda}_* = +\infty$. Each shifted leaf $\hat{T}_k := T_{\lambda_{k,1}} -x_{k,1}$ intersects $B_{R_0+2}$ nontrivially. Moreover, $\sup_{x \in T_{\pm} }|A_{T_\pm}(x)|<+\infty$, so by scaling we have $\sup_{x \in \hat{T}_k }|A_{\hat{T}_k}(x)|\to 0$ as $k \to \infty$. Here $A$ denotes the second fundamental form.
	As a consequence, $\hat{T}_{k}$ converges locally in $C^2$ to a hyperplane $\Pi$.
	
	Furthermore, each $\X_k$ satisfies \eqref{eqn: hp compactly contained} with $\lambda_{\pm} = \pm \lambda_{k,1}$, and hence Lemma~\ref{lem: trapping}  guarantees that 
	\[
	\partial \X_k(2) \cap \partial \X_k(3) \subset \R^{n+1} \setminus (F_{\lambda_{k,1}} \cup F_{-\lambda_{k,1}}).
	\]
	In particular, the concentration $\X_{k,1}$ has all interfaces lying to one side of $\hat{T}_k$, 
	and thus its limiting partition $\X_{\infty,1}= (\X_{\infty,1}(1), \X_{\infty,1}(2), \X_{\infty,1}(3))$ obtained in Lemma~\ref{lem: limit local min} has interfaces lying to one side of $\Pi$, and $|\X_{\infty,1}(1)|= v_1 \geq\bar{v}$. In view of Lemma~\ref{lem: limit local min}(iii), there are two possible scenarios: either $|\X_{\infty,1}(2)|=|\X_{\infty,1}(3)|=+\infty$, or $|\X_{\infty,1}(h)|<\infty$ for one of $h=2,3.$ 
	\\
	
	\noindent{\it Case 1:}
	First suppose $|\X_{\infty,1}(2)|=|\X_{\infty,1}(3)|=+\infty.$ By Lemma~\ref{lem: limit local min}(iii),  $\X_{\infty,1}$ is a locally minimizing $(1,2)$-cluster. By the considerations above, $\X_{\infty,1}(2)$ contains a half space bounded by $\Pi$. This means any blowdown of $\partial \X_{1,\infty}(2)\cap \partial \X_{1,\infty}(3)$ is a minimal surface contained in a half space, and hence is a hyperplane by the maximum principle. By \cite[Theorem 2.9]{BN}, this implies that $\X_{\infty,1}$ is a standard lens cluster of volume $v_1$ centered at a point $y \in \R^{n+1}$.  
	
	Let us use this information to bound  the energy of $\X_k$ from below. Let $\hat{B}_k = B_{3R_0}(x_{k,1}+y)$, and note that $B_{R_0}(x_{k,1})\subset \hat{B}_k.$ Assuming without loss of generality that $R_0$
	is large enough so that $\X_{\lens}(1)\subset B_{R_0}$, we have
	\begin{equation}
		\mathcal{P}(\X_{\infty,1}; B_{3R_0}(y))  =  (3R_0)^{n}\omega_{n} +\Lambda_{\plane}(n+1) v_1^{n/(n+1)},
	\end{equation}
	and therefore
	\begin{align}\label{eqn: energy 1}
		\mathcal{P}(\X_k ; \hat{B}_k)  = (3R_0)^n \omega_n + \Lambda_{\plane}(n+1) v_1^{n/(n+1)} + o_k(1).
	\end{align}
	
	By Lemma~\ref{lem: limit local min}(ii), each interface of the concentration $\X_{k,1}$ converges locally in the Hausdorff distance to the corresponding interface of the standard lens cluster above. In particular, the interface $\partial \X_k(2) \cap \partial \X_k(3)$ becomes arbitrarily close to a plane in the annulus $B_{4R_0}(x_{k,1}+y) \setminus \hat{B}_k$. So, we may find sets\footnote{Concretely, one can obtain (the interface between) such sets by gluing a correctly oriented plane in $B= B_{cR_0}(x_{k,1}+y)$, for a suitable $c \in [1,3/2]$, to the interface $\partial \X_k(2)\cap \partial \X_k(3)$ in $B_{4R_0}(x_{k,1}+y)\setminus B$, incurring a $o_k(1)$ error of energy along a portion of $\partial B$.  
	} $\X_k'(2)$ and $\X_k'(3)$ that coincide with $\X_k(2)$ and $\X_k(3)$ respectively in $B_{4R_k}\setminus \hat{B}_k$ and are complementary in $\hat{B}_k$ so that the cluster $\X'_k = (\X_k(1)\setminus \hat{B}_k, \X_k'(2), \X_k'(3))$ satisfies
	\begin{equation}\label{eqn: energy 3}
		\mathcal{P}( \X_k'; B_{4R_k} )
		= \mathcal{P}(\X_k ; B_{4R_k}\setminus \hat{B}_k)
		+ \omega_{n}(3R_0)^n + o_k(1)\,. 
	\end{equation}
	
	Rearranging and summing \eqref{eqn: energy 1} and \eqref{eqn: energy 3} shows that
	\begin{equation}\label{eqn: replaced cluster}
		\mathcal{P}(\X_k ; B_{4R_k}) = \mathcal{P}( \X_k'; B_{4R_k} ) + \Lambda_{\plane}(n+1) v_1^{n/(n+1)} + o_k(1).
	\end{equation}
	Let $v_{k,1} = |\X_k(1) \cap \hat{B}_k| \to v_1,$ so that the cluster $\X'_k$ above is an admissible competitor for $\tau_{\lambda, R_k}(v - v_{k,1})$. Here we adopt the convention that $\tau_{\lambda, R_k}(0) =0$, and in a slight abuse of notation consider a partition with empty $\X(1)$-chamber as an admissible competitor for this problem. Adding the term $\mathcal{G}_{R_k}(\X_k(1)) - \mathcal{H}^n(\Sigma\cap B_{4R_k})$ to both sides of \eqref{eqn: replaced cluster} and noting that $\mathcal{G}_{R_k}(\X'_k(1)) \leq \mathcal{G}_{R_k}(\X_k(1))$, we find 
	\begin{align*}
		\LT_{T_\lambda}(v) 
		&=\mathcal{E}_k(\X_k; B_{4R_k}) - \mathcal{H}^n(\Sigma\cap B_{4R_k}) +o_k(1) \\
		&\geq  \Lambda_{\plane}(n+1) v_1^{n/(n+1)} + \mathcal{E}_k(\X_k';B_{4R_k}) - \mathcal{H}^n(\Sigma\cap B_{4R_k}) + o_k(1)\\
		&\geq  \Lambda_{\plane}(n+1) v_1^{n/(n+1)} + [\tau_{R_k, \lambda}(v-v_k) - \mathcal{H}^n(\Sigma\cap B_{4R_k})] + o_k(1)\,.
	\end{align*}
	The term in brackets on the right-hand side tends to $\LT_{T_\lambda}(v-v_1)$ since $v\mapsto \tau_{R,\lambda}(v)$ is continuous in $v$, uniformly in $R$, a fact that can be readily verified through simple competitor arguments. Thus
	\begin{equation}
		\LT_{T_\lambda}(v)  \geq     \Lambda_{\plane}(n+1) v_1^{n/(n+1)} + \LT_{T_\lambda}(v-v_1),
	\end{equation}
	completing the proof in this case.\\
	
	\noindent{\it Case 2:} Now suppose $|\X_{\infty,1}(h)|<\infty$ for either $h=2$ or $3$. Without loss of generality we may assume this holds for $h=2.$ By Lemma~\ref{lem: limit local min}(iii), $\X_{\infty,1}(2)$ is empty, and in particular from the Hausdorff convergence, $\X_{k,1}(2) \cap B_{4R_0}(x_{k,1})$ is empty for $k$ sufficiently large. Notice that $\Lambda_{\plane}(n+1) < (n+1)\omega_{n+1}$ since the standard lens cluster is the unique local minimizer with planar area growth by \cite[Theorem 2.9]{BN}. In conjunction with 
	\eqref{eq:floater lower bound}, this means $\mathcal{P}(\X_{\infty,1}; B_{3R_0}) > \Lambda_{\plane}(n+1)v_1^{n/(n+1)}$, and thus, now letting $\hat{B}_k = B_{3R_0}(x_{k,1}),$
	\begin{equation}
		\mathcal{P}(\X_{k}; \hat{B}) > \Lambda_{\plane}(n+1)v_1^{n/(n+1)} + o_k(1).
	\end{equation}
	The cluster $\X'_k = (\X_k(1) \setminus \hat{B}, \X_k(2), \X_k(3) \cup \hat{B}_k)$ satisfies 
	\begin{equation}
		\mathcal{P}(\X_k'; B_{4R_k}) = \mathcal{P}(\X_k; B_{4R_k}\setminus \hat{B}_k) 
	\end{equation}
	Rearranging and summing these together yields
	\begin{equation}
		\mathcal{P}(\X_k; B_{4R_k} ) >  \Lambda_{\plane}(n+1)v_1^{n/(n+1)} +  \mathcal{P}(\X_k'; B_{4R_k}) +o_k(1).
	\end{equation}
	We now repeat the argument of case 1 starting directly after \eqref{eqn: replaced cluster} to conclude the proof in this case as well. 
\end{proof}

We are now in a position to prove Theorem \ref{thm:main thm}.

\begin{proof}[Proof of Theorem \ref{thm:main thm}]
	Fix $v_0>0$ small enough so that 
	\begin{equation}
		\label{eqn: v0 fix}
		\Lambda_{\rm plane}(n+1)\cdot (1-v_0)^{n/(n+1)}> \Lambda_{T_\lambda}>0
	\end{equation}
	for all $|\lambda|<\lambda_0$; this is possible thanks to \eqref{eqn: Lambda gap}. Fix $\lambda \in (-\lambda_0, \lambda_0)$ and take a sequence $R_k$ tending to $+\infty$. For each $k \in \mathbb{N}$, let $\X_k^1$ be a minimizer of the variational problem $\tau_{R_k,\lambda}(1)$ as defined in \eqref{eqn: tau lambda problem}.
	Let $\lambda_*$ be the constant associated to the sequence $\{\X_k^1\}$ as defined in \eqref{eqn: lambda star def}. Here we have tacitly passed to a subsequence without relabeling and will continue to do so in the remainder of the proof.
	
	Apply Lemma~\ref{lemma: dichotomy}. If $\lambda_*<+\infty$, {then the conclusions of the theorem are satisfied by the limiting cluster $\X$ of the first concentration. The decay rates \eqref{e:HS-slow-decay} and \eqref{e:HS-fast-decay} follow from the decay of the leaves $T_{\pm 2\lambda_*}$ of the foliation, see Remark~\ref{rmk: decay}. 
    } Suppose $\lambda_* = +\infty$. Then according to Lemma~\ref{lemma: dichotomy},
	\begin{equation}\label{eqn: first lower bound}
		\LT_{T_\lambda}(1) \geq \Lambda_{\plane} v_1^{n/(n+1)} + \LT_{T_\lambda}(1-v_1).
	\end{equation}
	By \eqref{eqn: v0 fix}, we see that $v_1\leq 1-v_0$, i.e. $1-v_1 \geq v_0$, hence we may run this argument again: let $\X_k^{1-v_1}$ be minimizers of $\tau_{R_k, \lambda}(1-v_1)$. Once again, take the $\lambda_*$ associated to this sequence as in \eqref{eqn: lambda star def}. By Lemma~\ref{lemma: dichotomy}, if $\lambda_*<+\infty$, the proof is complete. 
 If $\lambda_*=+\infty$, use the notation $v_2$ in place of $v_1$ to denote the volume of the first chamber of the limiting concentration $\X_{\infty,1}$ from this sequence. Then Lemma~\ref{lemma: dichotomy} together with  \eqref{eqn: first lower bound} and the concavity of $t\mapsto t^{n/(n+1)}$ shows
	\[
	\LT_{T_\lambda}(1) \geq \Lambda_{\plane} \left( v_1 + v_2\right)^{n/(n+1)} + \LT_{T_\lambda}(1-v_1-v_2).
	\]
	Again invoking \eqref{eqn: v0 fix}, we have $1-v_1-v_2 \geq 0$, allowing us to iterate the procedure again. In the $j$th iteration, the associated sequence of minimizers of $\tau_{R_k, \lambda}(1-v_1-\dots - v_{j-1})$ either satisfies $\lambda_*<+\infty$, in which case the proof is complete by Lemma~\ref{lemma: dichotomy}, or else $\LT_{T_\lambda}(1) \geq \Lambda_{\plane} \left( v_1 +\dots  v_j\right)^{n/(n+1)}$, which by \eqref{eqn: v0 fix} guarantees $v_1+\dots +v_j \leq 1-v_0.$  
	By Lemma~\ref{lem: vol lb}, we have $v_i \ge \bar{v}$ for each $i,$ hence we must have $\lambda_*<+\infty$ for the $j$th iteration of the argument for some $j \leq \lceil 1/\bar{v}\rceil$, thus allowing us apply Lemma \ref{lemma: dichotomy} and conclude {the existence of a minimizing cluster with blowdown $C$}. {Again, the decay estimates \eqref{e:HS-slow-decay} and \eqref{e:HS-fast-decay} are immediate consequences of the decay of the leaves of the foliation.}
\end{proof}

\begin{remark}\label{rmk: uniform rescaling}
    \rm{
    For each $\lambda \in [-\lambda_0,\lambda_0]$, the proof above produces a minimizing $(1,2)$-cluster $\X_\lambda$ with $|\X_\lambda(1)| =v_\lambda \geq \bar{v}$.   Up to rescaling by a factor $\rho_\lambda = |\X_\lambda(1)|^{-1/n} \leq \bar{v}^{-1/n}$, we may assume $|\X_\lambda(1)|=1.$

    A contradiction argument akin to the one in the proof of Lemma~\ref{lem: lower density} shows that the constant $c_0$ in the lower density estimate \eqref{e:lower-density-1}  can be taken to be uniform for all $T_\lambda$ with $\lambda \in [-\lambda_0,\lambda_0]$. In turn, the volume lower bound $\bar{v}>0$ for each concentration produced in this construction (coming from Lemma~\ref{lem: vol lb}) is uniform in $\lambda \in [-\lambda_0,\lambda_0].$ In particular the rescaling factor $\rho_\lambda$ has a uniform upper bound for all $\lambda \in [-\lambda_0,\lambda_0].$

    }
\end{remark}

\section{Proof of Theorem \ref{thm:main thm2}}\label{sec: leaf v cone}
In this section, we will be frequently working with the positive Jacobi fields $\phi_1(\xi) r^{-\gamma_1^\pm}$ from Section \ref{s:prelim} (or $\phi_1(\xi) r^{-\gamma_1}$ with $\gamma_1 = \tfrac{n-2}{2}$ if $C$ is strictly stable). To simplify notation, we will let
\begin{equation}\label{e:beta 2}
    \beta := n-2-2\gamma_1^- \geq 0\,,
\end{equation}
and we notice that $\gamma_1^+ + \gamma_1^- = n-2$, and thus $\gamma_1^+ - \gamma_1^- = \beta$.

For $0< s < r$, let $A_{s,r}$ denote the annulus {$B_r\setminus B_s$}. Given $u:C \cap \overline{A_{s,r}}\to \mathbb{R}$, we will write
    \begin{equation}\notag
      \graph_C^{s,r} u=\{x+u(x)\nu_C(x) : x\in C \cap A_{s,r} \}\,.  
    \end{equation}
    Note that the graph is not necessarily contained in the annulus. If $(s,r)=(0,\infty)$, we will omit the superscript and simply write $\graph_C$.

\begin{lemma}\label{l:area-gap-Plateau-leaf 2}
	Let $C$ be a strictly minimizing cone with an isolated singularity and let $\{T_\lambda\}_{\lambda\in \mathbb{R}}$ be the Hardt-Simon foliation given by Theorem \ref{t:Hardt-Simon}. Fix $\lambda \in \mathbb{R}$ and suppose there exist $R_0>0$, $q(r):(0,\infty) \to (0,\infty)$, and $a\in \mathbb{R}$ such that
\begin{equation}\label{eq:pure first mode hypothesis}
    T_\lambda \setminus B_{R_0} =\graph_C q(|x|)\phi_1(x/|x|)\,,
\end{equation}
with $q(r) = a r^{-\gamma_1^-} +\rm{o}(r^{-\gamma_1^-})$, and $\phi_1$ is as in Section \ref{subsec:iso sing prelim}. For $R_0 \leq \rho_0 < S < R/4$, suppose that we have a hypersurface $M = M_1 \sqcup M_2$ where $M_1 = \graph_C^{\rho_0,S} u$ for a function $u: C\cap A_{\rho_0,S} \to \R$ satisfying
\begin{equation}\label{eq:uk vanishes}
    \|u-\tilde{a}r^{-\gamma_1^-}\phi_1\|_{L^\infty(C \cap A_{{\rho_0},S})} + \|r(\nabla_C u-\nabla_C (\tilde{a}r^{-\gamma_1^-}\phi_1))\|_{L^\infty(C \cap A_{{\rho_0},S})}{=\mathrm{o}(S^{-\gamma_1^-})}
\end{equation}
 if $C$ is strictly stable, and
{\begin{equation}\label{eq:uk vanishes - non_str_stable}
    \|u-\tilde{a}r^{-\gamma_1^-}\log r\cdot\phi_1\|_{L^\infty(C \cap A_{{\rho_0},S})} + \|r(\nabla_C u-\nabla_C (\tilde{a}r^{-\gamma_1^-}\log r\cdot\phi_1))\|_{L^\infty(C \cap A_{{\rho_0},S})}{=\mathrm{o}(S^{-\gamma_1^-}\log S)}
\end{equation}}
 otherwise, while $\llbracket M_2 \rrbracket$ is {an area-minimizing integral current with $\partial \llbracket M_2 \rrbracket = \llbracket \graph_{C\cap \partial B_R} q\phi_1 \rrbracket + \llbracket \graph_{C\cap \partial B_S} u \rrbracket$}. If $\tilde{a} \neq a$, then
\begin{equation}\label{eq:area divergence}
    \mathcal{H}^n(M) - \mathcal{H}^n(\graph_C^{\rho_0,R}(q\phi_1))\to \infty
\end{equation}
as $S\to\infty$ and {$R/S\to \infty$}.
\end{lemma}
\begin{remark}\label{r:pure-first-mode}
    The hypothesis \eqref{eq:pure first mode hypothesis} is required in order to gain improved error estimates on the area comparison between $M$ and $T_\lambda$ that are higher order in the smaller scale $S$ rather than the larger scale $R$. This is crucial, since the dominant term in the area comparison is of order $S^\beta$ in the strictly stable case and order $\log S$ in the non-strictly stable case. Within the proof of the lemma, this can be seen in the derivation of \eqref{e:rem-bd} and in step four.

    It is worth pointing out that the hypothesis \eqref{eq:pure first mode hypothesis} is not necessary in the specific case of $n=8$ and $\tilde a = 0$, since one may show that all higher order errors are $O(R^{n-3\gamma_1^- -3})$ and $\gamma_1^- \geq \frac{n-2}{2} - \sqrt{\left(\frac{n-2}{2}\right)^2 - (n-1)}$ so in particular $n-3\gamma_1^- -3 < 0$ when $n=7$, cf. the area expansion in Appendix \ref{ap:expansion} and step one of the proof, but we omit the details here since the conclusion is significantly stronger under hypothesis \eqref{eq:pure first mode hypothesis}.
\end{remark}

\begin{proof}
    {We perform the computation in the case when $C$ is in addition strictly stable, namely \eqref{e:strictly-min-strictly-stable} holds for $T_\lambda$. The only difference in the argument for the non-strictly stable case is that, in light of the amended asymptotic behavior \eqref{e:strictly-min-non-strictly-stable}, all terms of the form $S^\beta$ and $R^\beta$ become respectively $\log S$ and $\log R$, with error terms being of the order ${\rm o}(\log S)$. We therefore omit the details of the modifications in the non-strictly stable case and simply state the final estimate obtained in that case in the concluding step of the proof below.}

    Before giving the proof, which is divided into steps, let us set some notation. We will use polar coordinates $(r,\omega)$ for $x\in \mathbb{R}^{n+1}$. We set $Z_u(x):=|\nabla_C u(x)| + |u(x)|/|x|$, and write
    \[
        Q_{s,r}(u_1,u_2) := \int_{C\cap A_{s,r}} \langle \nabla_C u_1, \nabla_C u_2 \rangle - |A_C|^2 u_1 u_2 \, d\Hcal^n\,
    \]
for the second variation of $C$ restricted to an annulus. {We let $F_{s,r}(u) = \mathcal{H}^n(\graph_C^{s,r} u) - \mathcal{H}^n(C \cap A_{s,r})$. Note also that by Lemma \ref{lemma:small data dirichlet existence}, for $R/S$ large enough, $M_2$ is graphical over $C$, and so we can extend $u$ so that the entire surface $M$ is the graph of $u$.}

\medskip

\noindent{\it Step zero (expansion of area).} We claim that there are $\varepsilon_0, C_0>0$ such that if $\|Z_u\|_{L^\infty(C \cap A_{s,r})}\leq \varepsilon_0$, then 
\begin{equation}\label{eq:remainder estimate}
    F_{s,r}( u) = \frac{1}{2}Q_{s,r}(u,u)  + \mathcal{R}_{s,r}(u)\,,
\end{equation}
where the remainder $\mathcal{R}_{s,r}(u)$ satisfies
\begin{equation}\label{e:remainder-cubic-bd}
   |\mathcal{R}_{s,r}(u)| \leq C_0 \int_{C \cap A_{s,r}}Z_u^3\,d\mathcal{H}^n\,,
\end{equation}
and for any smooth, compactly supported $\psi: C \cap A_{s,r}\to \mathbb{R}$, the first variation $\delta \mathcal{R}_{s,r}(u)$ of the remainder at $u$ satisfies
\begin{equation}\label{eq:remainder first variation estimate}
    |\delta \mathcal{R}_{s,r}(u)[\psi]|\leq C_0\int_{C \cap A_{s,r}}Z_u^2 Z_\psi\,d\mathcal{H}^n\,.
\end{equation}
{Lastly, up to decreasing $\varepsilon_0$ and increasing $C_0$, we claim that
\begin{equation}\label{eq:closeness of second variations}
    \left|\delta^2 F_{s,r}(u)[v,w] - Q_{s,r}(v,w) \right| \leq C_0 \|Z_u\|_{L^\infty}\left(\int_{C \cap A_{s,r}}|\nabla_C v|^2 + \frac{v^2}{r^2} \right)^{1/2}\left(\int_{C \cap A_{s,r}}|\nabla_C w|^2 + \frac{w^2}{r^2} \right)^{1/2}\,.
\end{equation}
The proofs of these claims can be found in Appendix \ref{ap:expansion}.}

\medskip

\noindent{\it Step one (area difference on $A_{{\rho_0},S}$)}. Here we claim that 
\begin{equation}\label{eq:area difference on first annulus}
\mathcal{H}^n(M\cap A_{\rho_0,S}) -\mathcal{H}^n(T_{\lambda }\cap A_{\rho_0, S})    = \frac{1}{2}[a^2 - \tilde{a}^2]\gamma_1^-S^\beta + \mathrm{o}(S^\beta)\,.
\end{equation}
Let $L_C:= \Delta_C + |A_C|^2$ denote the Jacobi operator on $C$. Thanks to step zero combined with an integration by parts,  we may write
\begin{align*}
 \mathcal{H}^n(M\cap A_{\rho_0,S}) -\mathcal{H}^n(T_{\lambda }\cap A_{\rho_0, S}) 
    &= \frac{1}{2} Q_{{\rho_0},S}(u,u) - \frac{1}{2} Q_{{\rho_0},S}(v,v) + \Rcal_{{\rho_0},S}(u) - \Rcal_{{\rho_0},S_k}(v) \\
    &= \frac{1}{2} \int_{C\cap A_{{\rho_0},S}} \big(v L_C v - u L_C u\big)\, d\Hcal^n \\
    &\qquad+ \frac{1}{2} \int_{C\cap \partial B_{S}}( u \partial_r u  -  v \partial_r v)\, d\Hcal^{n-1} + C_1({\rho_0}) \\
    &\qquad+ \Rcal_{{\rho_0},S}(u) - \Rcal_{{\rho_0},S}(v)\,.
\end{align*}
Now, \eqref{eq:uk vanishes} and the definition of $v$ guarantee that
\begin{equation}\label{e:step1-error1}
    \Rcal_{{\rho_0},S}(u) = O(S^{n-3\gamma_1^- -3}) = o(S^\beta) \qquad \text{and} \qquad \Rcal_{{\rho_0},S}(v) = O(S^{n-3\gamma_1^- -3}) = o(S^\beta) \,.
\end{equation}
Moreover, by elliptic estimates we have
\begin{equation}\label{e:step1-error2}
     L_C (v) = o(S^{-\gamma_1^- -2})\qquad \text{and} \qquad L_C u = o(S^{-\gamma_1^- -2})\,,
\end{equation}
and
\begin{equation}\label{e:step1-error3}
    \partial_rv = -a\gamma_1^- S^{-\gamma_1^- -1} \phi_1 + o(S^{-\gamma_1^- -1})\qquad \text{and} \qquad \partial_r u = - \tilde{a} \gamma_1^- S^{-\gamma_1^- -1}\phi_1 + o(S^{-\gamma_1^- -1})\,.
\end{equation}
Inserting \eqref{e:step1-error1}-\eqref{e:step1-error3} into the calculation above and applying the coarea formula yields \eqref{eq:area difference on first annulus} as desired.

\medskip

\noindent{\it Step two (reduction of area estimate on $A_{S,R}$ to case with Jacobi field data on $C \cap \partial B_S$)}. 
Consider the Plateau problem {(again in the sense of integral currents)} with boundary data given by the graph of $\tilde{a}S^{-\gamma_1^-}\phi_1$ on $C \cap \partial B_S$ and given by $T_\lambda$ on ${C \cap \partial B_{R}}$. By Lemma~\ref{lemma:small data dirichlet existence}, the solution is a normal graph of a function $\tilde{u}$ on $C\cap A_{S,R}$. Similarly, let $\tilde{v}$ be the function whose graph solves the Plateau problem with boundary data given by the graph of $aS^{-\gamma_1^-}\phi_1$ on $C \cap \partial B_S$ and given by $T_\lambda$ on ${C \cap \partial B_{R}}$. Note that 
\begin{equation}\label{eq:uk tilde uk estimate}
    Z_{\tilde{u}} + Z_{\tilde{v}} = {\rm O}(|x|^{-\gamma_1^- - 1})\,.
\end{equation}
by Lemma~\ref{lemma:small data dirichlet existence}. In this step we show that
\begin{align}\label{eq:reduction to inner jacobi data tilde}
    {F}_{S,R}(u) = {F}_{S,R}(\tilde{u}) + \mathrm{o}(S^\beta)\qquad \mbox{and}\qquad
    {F}_{S,R}(v) &= {F}_{S,R}(\tilde{v}) + \mathrm{o}(S^\beta)\,.
\end{align}
We prove the first estimate in \eqref{eq:reduction to inner jacobi data tilde}; the second is the same. We first show that 
\begin{equation}\label{eq:inner jacobi upper bound}
    {F}_{S,R}(u) \leq {F}_{S,R}(\tilde{u}) + \mathrm{o}(S^\beta)
\end{equation}
 Let $H= (u-\tilde{u})|_{C\cap \partial B_S}$, which has $\|H\|_{C^1(C\cap \partial B_S))} = o(S^{-\gamma_1^-})$ by \eqref{eq:uk vanishes}. By an elementary extension argument, we may extend $H$ to some $W \in W^{1,\infty}(C;\mathbb{R})$ such that
\begin{equation}\label{eq:wk estimate}
    \spt\, W \subset C \cap B_{2S}\,, \quad \|W\|_{L^\infty(C)}=\mathrm{o}(S^{-\gamma_1^-})\,,\quad \mbox{and}\quad \|\nabla_C W\|_{L^\infty(C \cap B_{2S}\setminus B_{S})}=\mathrm{o}(S^{-\gamma_1^- -1})\,.
\end{equation}
Since $W$ agrees with $H$ on $\partial B_{S}$, ${\rm graph}_C^{S,R}(\tilde{u} + W)$ is an admissible competitor for the Plateau problem corresponding to $u$. Thus, Taylor expanding and using \eqref{eq:closeness of second variations} and the fact that the second fundamental form $A_C$ of $C$ satisfies $|A_C| \simeq |x|^{-1}$ on $C \setminus \{0\}$, we obtain
\begin{align}\
   {F}_{S,R}(u)&\leq F_{S,R}(\tilde{u}+W)\\ 
    &\leq F_{S,R}(\tilde{u})+ \delta F_{S,R}(\tilde{u})[W] + C_0 \|Z_{v}\|_{L^\infty}{Q}_{S,R}(W,W)   \label{eq:fkvk expansion}
\end{align}
Define $G:C \cap \partial B_{S}\to \mathbb{R}^{n+1}$ by $G(x) = x+ \tilde{u}(x)\nu_C(x)$. Since $N:=\graph_C^{S,R} \tilde{u}$ is a minimal surface with boundary and $\spt\,W\subset B_{2S} \cc B_{R}$, we have
\begin{equation}\notag
   \delta F_{S,R}(\tilde{u})[W]= \int_{G(C \cap \partial B_{S})}(W\circ G^{-1}) (\nu_C \circ G^{-1}) \cdot \nu_{N}^{\rm co}\,d\mathcal{H}^{n-1}\,,
\end{equation}
{where $\nu_{N}^{\co}$ denotes the unit conormal to $G(C \cap \partial B_{S})$ in $G(C \cap B_{S})$.} Towards estimating the dot product $(\nu_C\circ G^{-1}) \cdot \nu_{N}^{\rm co}$, we recall that  since  $|\nabla_C \tilde{u}| \leq Z_{\tilde{u}}=\mathrm{O}(S^{-\gamma_1^--1})$ on $C \cap \partial B_{S}$. {Again using that $|A_C| \simeq |x|^{-1}$ on $C \setminus \{0\}$}, a direct computation of $|(\nu_C\circ G^{-1}) \cdot \nu_{N}^{\rm co}|$ shows that it is bounded by $(1+{\rm o}(1))|\nabla_C \tilde{u}|$, yielding
\begin{equation}\notag
    (\nu_C\circ G^{-1}) \cdot \nu_{N}^{\rm co} = {\rm O}(S^{-\gamma_1^--1})\qquad \mbox{on $G(C \cap \partial B_{S})$}\,.
\end{equation}
Using this together with the fact that $\mathcal{H}^{n-1}(G(C \cap \partial B_{S}))={\rm O}(S^{n-1})$ and the $L^\infty$ bound for $W$ from \eqref{eq:wk estimate}, we obtain
\begin{equation}\label{eq:delta fsr estimate}
    \delta F_{S,R}(\tilde{u})[W]={\rm O}(S^{n-1-\gamma_1^--1}){\rm o}(S^{-\gamma_1^-})={\rm o}(S^{n-2-2\gamma_1^-})={\rm o}(S^\beta)\,.
\end{equation}
Similarly, using the estimate \eqref{eq:uk tilde uk estimate} for $Z_{\tilde{u}}$, the estimate \eqref{eq:wk estimate} for $W$ and $\nabla_C W$, and the fact that $\spt\, W\subset C \cap B_{2S}$, we have
\begin{equation}\label{eq:delta squared fsr estimate}
    \|Z_{\tilde{u}}\|_{L^\infty}{Q}_{S,R}(W,W) = {\rm O}(S^{-\gamma_1^--1 + n - 2\gamma_1^- - 2})={\rm o}(S^\beta)\,.
\end{equation}
Inserting \eqref{eq:delta fsr estimate} and \eqref{eq:delta squared fsr estimate}
into \eqref{eq:fkvk expansion} yields \eqref{eq:inner jacobi upper bound}.

To finish the proof of \eqref{eq:reduction to inner jacobi data tilde}, it remains to prove the reverse inequality 
\begin{equation}\label{eq:inner jacobi lower bound}
   F(\tilde{u})-{F}(u) \leq \mathrm{o}(S^\beta)\,.
\end{equation}
The proof is nearly identical to the proof of \eqref{eq:inner jacobi upper bound}, so we provide a summary instead of complete details. First, using the area minimality of the graph of $\tilde{u}$, we have the analogue 
\begin{align*}\notag
   F(\tilde{u})&\leq F_{S,R}( u-W)\leq {F}_{S,R}(u)- \delta F_{S,R}(u)[W] + C_0 \|Z_{u}\|_{L^\infty}{Q}_{S,R}(W,W)
\end{align*}
of \eqref{eq:fkvk expansion}. The first and second variation estimates then follow exactly as in \eqref{eq:delta fsr estimate} and \eqref{eq:delta squared fsr estimate}, since $u$ satisfies the same decay estimates as $\tilde{u}$. This concludes the proof of \eqref{eq:inner jacobi lower bound} and thus the proof of \eqref{eq:reduction to inner jacobi data tilde}.

\medskip

\noindent{\it Step three (equivalence of area estimate on $A_{S,R}$ with inner Jacobi field data to estimate of second variations)}.
Next, let $U$ and $V$ solve the Jacobi equation $L_C \ \cdot =0$ on $C\cap A_{S,R}$ with boundary data equal to $\tilde{u}$ and $\tilde{v}$ respectively on $C\cap (\partial B_S\cup \partial B_R)$. In this step we show
\begin{align}
F_{S,R}(\tilde{u})&-{F}_{S,R}(\tilde{v})  \label{eq:equivalence to quadratic energies}
    =  {\frac{1}{2}}{Q}_{S,R}(U,U)-{\frac{1}{2}}{Q}_{S,R}(V,V)+ \mathrm{o}(S^\beta)\,.
\end{align}
By symmetry, since $a$ and $\tilde{a}$ are arbitrary, to show \eqref{eq:equivalence to quadratic energies} it is enough to show the inequality 
\begin{align}
 F_{S,R}(\tilde{u})&-{F}_{S,R}(\tilde{v})   \label{eq:equivalence to quadratic energies 2}
    \leq{\frac{1}{2}}{Q}_{S,R}(U,U)-{\frac{1}{2}}{Q}_{S,R}(V,V)+ \mathrm{o}(S^\beta)\,;
\end{align}
the opposite inequality is the exact same. 

Set $\Psi= U - V$,
which by linearity solves the Jacobi equation
\begin{equation}\label{eq:pde for psi}
L_C\Psi= 0 \text{ on }C \cap A_{S,R}, \qquad
  \restr{\Psi}{C \cap \partial B_{S}} = (\tilde{a}-a)S^{-\gamma_1^-}\phi_1\,, \qquad \restr{\Psi}{C \cap \partial B_{R}} = 0\,.
\end{equation}
Since $\tilde{v}+\Psi = \tilde{u}$  on $C \cap (\partial B_{S}\cup \partial B_{R})$,
we can test the area minimality of the graph of $\tilde{u}$ against the graph of $\tilde{v}+ \Psi$, so that,
using the expansion of the area functional over $C$ from step zero, we obtain
\begin{align}\notag
F_{S,R}(\tilde{u}) - F_{S,R}(\tilde{v}) &\leq F_{S,R}(\tilde{v}+ \Psi)- F_{S,R}(\tilde{v})\\ \label{eq:tested tilde uk}
&= \frac{1}{2} Q_{S,R}(\tilde{v} + \Psi, \tilde{v} + \Psi) -\frac{1}{2}Q_{S,R}(\tilde{v},\tilde{v})+ \mathcal{R}_{S,R}(\tilde{v}+ \Psi,\tilde{v}+ \Psi) - \mathcal{R}_{S,R}(\tilde{v},\tilde{v})  \,.
\end{align}
Next, we claim that
\begin{equation}\label{eq:q fiddling}
    Q_{S,R}(\tilde{v} + \Psi, \tilde{v} + \Psi) -Q_{S,R}(\tilde{v},\tilde{v}) = Q_{S,R}(V+ \Psi,V+ \Psi) -Q_{S,R}(V,V)\,.
\end{equation}
Indeed, we expand the square and integrate by parts, recalling that $\Psi$ solves the Jacobi equation $L_C\Psi=0$ and vanishes on $C\cap \partial B_R$ to find \begin{align}\notag
Q_{S,R}(\tilde{v}+ \Psi, \tilde{v}+\Psi) -Q_{S,R}(\tilde{v}, \tilde{v}) &= Q_{S,R}(\Psi,\Psi) + 2Q_{S,R}({\tilde v},\Psi)\\ \notag
&= -\int_{\partial B_S\cap C} \Psi \partial_r\Psi + 2 \Big(\int_{\partial B_R\cap C} \tilde{v} \partial_r \Psi - \int_{\partial B_{S}\cap C} \tilde{v} \partial_r\Psi \Big)
\end{align}
Since $\tilde{v}=V$ on $C\cap (\partial B_S\cup\partial B_R)$, the expansion with $V$ in place of $\tilde{v}$ yields the identical right-hand side.
As a consequence, inserting \eqref{eq:q fiddling} into \eqref{eq:tested tilde uk} and recalling that $\Psi +V = U$ gives 
\begin{align}
F_{S,R}(\tilde{u}) - F_{S,R}(\tilde{v}) &\leq \frac{1}{2}Q_{S,R}(U, U) -\frac{1}{2}Q_{S,R}(V,V)+ \mathcal{R}_{S,R}(\tilde{v}+ \Psi) - \mathcal{R}_{S,R}(\tilde{v})\label{eq:equality up to remainders}
\end{align}
So, to finish proving \eqref{eq:equivalence to quadratic energies 2}, it remains to show that
\begin{equation}\label{e:rem-bd}
    \mathcal{R}_{S,R}(\tilde{v} + \Psi) - \mathcal{R}_{S,R}(\tilde{v}) = {\rm o}(S^\beta)\,.
\end{equation}
With this aim in mind, we begin by computing $\Psi$ using \eqref{eq:pde for psi}. Since the boundary data for $\Psi$ involves only the $\phi_1$ mode, $\Psi$ is a linear combination of the Jacobi fields corresponding to $\gamma_1^-$ and $\gamma_1^+$. A direct calculation which we omit that uses the boundary conditions to solve for these coefficients yields the following expression for $\Psi$ in polar coordinates:
\begin{equation}\notag
  \Psi(r,\omega) = \frac{S^\beta}{1-S^\beta/R^\beta}(\tilde{a}-a)(r^{-\gamma_1^+}- R^{-\beta} r^{-\gamma_1^-})\phi_1(\omega)\,.  
\end{equation}
Since $2S \leq R$, {recalling also that $-\gamma_1^- - \beta = -\gamma_1^+$}, this implies
\begin{equation}\notag
  |\Psi| \leq C S^\beta|\tilde{a}-a|\left(|r^{-\gamma_1^+}| + r^{-\beta}|r^{-\gamma_1^-}|\right)\|\phi_1\|_{L^\infty} \leq C S^\beta {r^{-\gamma_1^+}} \,.
\end{equation}
Similarly, differentiating $\Psi$ leads to the bound $ |\nabla_C \Psi| \leq C S^\beta r^{-{\gamma_1^+}-1}\,$
which, combined with the previous inequality shows
\begin{equation}\label{eq:zpsi bound}
    Z_{\Psi}\leq C S^\beta r^{-{\gamma_1^+}- 1}\,.
\end{equation}
Using the fundamental theorem of calculus, \eqref{eq:remainder estimate}, the triangle inequality, and the inequality $Z_{t\Psi} \leq Z_\Psi$ for $t\in [0,1]$, we estimate
\begin{align}\notag
  \left|\mathcal{R}_{S,R}(\tilde{v}+ \Psi) - \mathcal{R}_{S,R}(\tilde{v}) \right|&\leq \int_0^1 |\delta \mathcal{R}_{S,R}(\tilde{v} + t \Psi)[\Psi]|\, dt \leq \int_0^1 \int_{C \cap A_{S,R}}(Z_{\tilde{v} + t \Psi})^2 Z_\Psi\, d\Hcal^n \,dt \\ \notag
  &\leq \int_{C \cap A_{S,R}} (Z_{\tilde{v}} + Z_\Psi)^2 Z_\Psi\, d\Hcal^n = \int_{C \cap A_{S,R}} Z_{\tilde{v}}^2 Z_\Psi + 2Z_{\tilde{v}}Z_\Psi^2+ Z_\Psi^3\, d\Hcal^n  \,.
\end{align}
Recalling the bounds $Z_{\tilde{v}} = {\rm O}(r^{-\gamma_1^--1})$ (from \eqref{eq:uk tilde uk estimate}) and $Z_\Psi \leq C S^\beta r^{-\gamma_1^+- 1}$ (from \eqref{eq:zpsi bound}), {and the identities $\gamma_1^+ + \gamma_1^- = n-2$ and $\gamma_1^+ - \gamma_1^- = \beta$}, we estimate
\begin{align}\notag
    \left|\mathcal{R}_{S,R}(\tilde{v} + \Psi) - \mathcal{R}_{S,R}(\tilde{v}) \right|&\lesssim  \int_{S}^{R} {S^\beta} r^{n-1}[r^{-2\gamma_1^- - 2 - \gamma_1^+- 1} + S^\beta r^{-\gamma_1^- - 1 - 2\gamma_1^+ - 2}+ {S^{2\beta}} r^{-3\gamma_1^+ - 3}] \, dr \\ \notag
    &\leq S^\beta \int_{S}^{R} {r^{-2 - \gamma_1^-}}\,dr = {\rm o}(S^\beta)\,,
\end{align}
This concludes the remainder bound and thus the proofs of \eqref{eq:equivalence to quadratic energies 2} and \eqref{eq:equivalence to quadratic energies}.

\medskip

\noindent{\it Step four (computation of second variation on $A_{S,R}$ with inner Jacobi data)}. In this step we prove that
\begin{align}
   Q_{S,R}(U,U)&-Q_{S,R}(V,V) \label{eq:computation of outer jacobi difference}
    =\left[\gamma_1^-(\tilde{a}^2 - a^2) + \frac{\beta (a-\tilde{a})^2}{1-(S/R)^{\beta}}\right]S^\beta + \mathrm{o}(S^\beta)\,.
\end{align}
Integrating by parts and using the fact that $U$ and $V$ solve the Jacobi equation in $A_{S,R}$ and agree on $\partial B_{R}$, we first notice that
\begin{align*}
   Q_{S,R}(U,U)-Q_{S,R}(V,V)  
    &=\int_{C \cap \partial B_{R}} U (\partial_{r} U - \partial_{\nu} V) - \int_{C \cap \partial B_{S}}U\partial_{r} U\, d\Hcal^{n-1}  + \int_{C \cap \partial B_{S}} {V} \partial_{r} {V}\, d\Hcal^{n-1} \\
    &=\int_{C \cap \partial B_{R}} U (\partial_{r} U - \partial_{\nu} V)\, d\Hcal^{n-1} - \tilde a\int_{C \cap \partial B_{S}} S^{-\gamma_1^-} \phi_1 \partial_{r} U\, d\Hcal^{n-1} \\
    &\qquad + a \int_{C \cap \partial B_{S}} S^{-\gamma_1^-}\phi_1 \partial_{r} {V}\, d\Hcal^{n-1}
\end{align*}
Now, we may expand $V$ and $U$ in polar coordinates as
\begin{equation}
	    \label{eqn: g expansion}
	V(r,\omega) = \sum_{j =1}^\infty (a_j \phi_j(\omega) r^{-\gamma_j^-} + b_j \phi_j(\xi) r^{-\gamma_j^+}) \,,\qquad U(r,\omega) = \sum_{j=1}^\infty (\tilde a_j \phi_j(\omega) r^{-\gamma_j^-} + \tilde b_j \phi_j(\xi) r^{-\gamma_j^+})
		\end{equation}
	where $\{\phi_j\}_j$ are a basis of eigenfunctions (with $\phi_1$ as before) of $\Delta_{C\cap\partial B_1} + |A_{C\cap\partial B_1}|^2$ with  $\int_{C\cap \partial B_1} \phi_j^2 = 1$ and eigenvalues $\lambda_j \to \infty$ as $j\to \infty$, and $\gamma_j^\pm = \frac{n-2}{2} \pm \sqrt{\left(\frac{n-2}{2}\right)^2 + \lambda_j}$ (cf. Section \ref{s:prelim}), and since the boundary data of $V$ and $U$ lie in the span of $\phi_1$ alone, the coefficients $a_j, b_j, \tilde a_j, \tilde b_j$ satisfy 
    \[
        a_j = b_j = \tilde a_j =\tilde b_j = 0 \qquad \text{for $j \geq 2$\,,}
	\]
    while
    \[
	\begin{cases}
		a_1 R^{-\gamma_1^-} + b_1 R^{-\gamma_1^+} = {a} R^{-\gamma_1^-} {+ \sigma} \\
		a_1 S^{-\gamma_1^-} + b_1 S^{-\gamma_1^+} = a S^{-\gamma_1^-} \,.
	\end{cases}
    \]
    and
    \[
	\begin{cases}
		\tilde a_1 R^{-\gamma_1^-} + \tilde b_1 R^{-\gamma_1^+} = {a} R^{-\gamma_1^-} {+ \sigma}\\
		\tilde a_1 S^{-\gamma_1^-} + \tilde b_1 S^{-\gamma_1^+} = \tilde a S^{-\gamma_1^-} \,.
	\end{cases}
	\]
    for a constant $\sigma$ which is $o(R^{-\gamma_1^-})$. This yields 
    \begin{align*}
        a_1 &= a + \sigma \frac{R^{\gamma_1^+}}{R^\beta - S^\beta}\,, \qquad b_1 = -\sigma\frac{R^{\gamma_1^-}R^\beta S^\beta}{R^\beta - S^\beta}\\
        \tilde a_1 &= \frac{a R^\beta - \tilde a S^\beta}{R^\beta - S^\beta} + \sigma \frac{R^{\gamma_1^+}}{R^\beta - S^\beta}\,, \qquad \tilde b_1 = \frac{(\tilde a-a -\sigma R^{\gamma_1^-})R^\beta S^\beta}{R^\beta - S^\beta}\,.
	\end{align*}
    In other words,
    \begin{align*}
        V(r,\omega) &= \left[\left(a + \sigma \frac{R^{\gamma_1^+}}{R^\beta - S^\beta}\right) r^{-\gamma_1^-} - \sigma\frac{R^{\gamma_1^-}R^\beta S^\beta}{R^\beta - S^\beta} r^{-\gamma_1^+}\right] \phi_1(\omega)\\
        U(r,\omega) &= \left[\left(\frac{a R^\beta - \tilde a S^\beta}{R^\beta - S^\beta} + \sigma \frac{R^{\gamma_1^+}}{R^\beta - S^\beta}\right) r^{-\gamma_1^-} + \frac{(\tilde a- a -\sigma R^{\gamma_1^-})R^\beta S^\beta}{R^\beta - S^\beta} r^{-\gamma_1^+}\right] \phi_1(\omega) \,.
    \end{align*}
    Let us now insert this into the above computation for the difference between $Q_{S,R}(U,U)$ and $Q_{S,R}(U,U)$. Noticing that $U$ and $V$ share the same $\sigma$ terms, we observe that the corresponding terms in $\partial_{r} U- \partial_{r} V$ cancel, and so we are left with
    \begin{align*}
        \int_{C \cap \partial B_{R}} &V (\partial_{r} U - \partial_{r} V)\, d\Hcal^{n-1} \\
        &= (a R^{-\gamma_1^-} + \sigma) \left(-\gamma_1^-\frac{a R^\beta - \tilde a S^\beta}{R^\beta - S^\beta} R^{\beta+\gamma_1^-} - \gamma_1^+(\tilde a - a) \frac{R^\beta S^\beta}{R^\beta - S^\beta} R^{\gamma_1^-} + \gamma_1^- a R^{\beta+\gamma_1^-}\right) \\
        &= \beta (a R^{-\gamma_1^-} + \sigma)\frac{(a- \tilde a) S^\beta R^{\beta+\gamma_1^-}}{R^\beta - S^\beta}\,.
    \end{align*}
    Similarly,
    \begin{align*}
        a \int_{C \cap \partial B_{S}} S^{-\gamma_1^-} \phi_1 \partial_r V &= a \left[-\gamma_1^- \left(a + \sigma \frac{R^{\gamma_1^+}}{R^\beta - S^\beta}\right) S^{\beta} + \sigma \gamma_1^+ \frac{R^{\gamma_1^-}R^\beta S^\beta}{R^\beta - S^\beta}\right]
    \end{align*}
    and
    \begin{align*}
        -\tilde a \int_{C \cap \partial B_{S}} S^{-\gamma_1^-} \phi_1 \partial_r U &= \tilde a\left[\gamma_1^-\left(\frac{a R^\beta - \tilde a S^\beta}{R^\beta - S^\beta} + \sigma \frac{R^{\gamma_1^+}}{R^\beta - S^\beta}\right) S^{\beta} +\gamma_1^+ \frac{(\tilde a-a -\sigma R^{\gamma_1^-})R^\beta S^\beta}{R^\beta - S^\beta}\right] 
    \end{align*}
    We therefore arrive at
    \begin{align*}
        &  Q_{S,R}(U,U)-Q_{S,R}(V,V)  \\ \notag
        &= \left\{\beta (a R^{-\gamma_1^-} + \sigma)\frac{(a- \tilde a) S^\beta R^{\beta+\gamma_1^-}}{R^\beta - S^\beta} + a \left[-\gamma_1^- \left(a + \sigma \frac{R^{\gamma_1^+}}{R^\beta - S^\beta}\right) S^{\beta} + \sigma \gamma_1^+ \frac{R^{\gamma_1^-}R^\beta S^\beta}{R^\beta - S^\beta}\right]\right.\\ \notag
        &\qquad+\left. \tilde a\left[\gamma_1^-\left(\frac{a R^\beta - \tilde a S^\beta}{R^\beta - S^\beta} + \sigma \frac{R^{\gamma_1^+}}{R^\beta - S^\beta}\right) S^{\beta} +\gamma_1^+ \frac{(\tilde a-a -\sigma R^{\gamma_1^-})R^\beta S^\beta}{R^\beta - S^\beta}\right]\right\}\\
        &= \gamma_1^- (\tilde a^2 - a^2)S^\beta + \beta (a-\tilde a)^2  \frac{S^\beta}{1 - (S/R)^\beta} + 2\beta \sigma \frac{R^{\gamma_1^-}S^\beta}{1 - (S/R)^\beta} (a-\tilde a)
    \end{align*}
   Keeping in mind that $\sigma = o(R^{-\gamma_1^-})$, this exactly simplifies to \eqref{eq:computation of outer jacobi difference}.
\medskip

\noindent{\it Conclusion}. Combining  \eqref{eq:area difference on first annulus}, \eqref{eq:reduction to inner jacobi data tilde}, \eqref{eq:equivalence to quadratic energies}, \eqref{eq:computation of outer jacobi difference}, and the fact that $u_k=q\phi_1$ on $C \cap \partial B_{R_k}$, we compute
\begin{align}\notag
    \mathcal{H}^n(M) - \mathcal{H}^n({\graph_C^{\rho_0,R}(q\phi_1)}) &=\mathcal{H}^n({M_1}) - \mathcal{H}^n({\graph_C^{\rho_0,S}(q\phi_1)}) + {F}_{S,R}(u) - {F}_{S,R}(v)\\ \notag
    &= \frac{1}{2}(a^2 - \tilde{a}^2)\gamma_1^-S^\beta +  {F}_{S,R}(\tilde{u} ) -{F}_{S,R}(\tilde{v}) + \mathrm{o}(S^\beta)\\ \notag
    &= \frac{1}{2}(a^2 - \tilde{a}^2)\gamma_1^-S^\beta + \frac{1}{2}(Q_{S,R}(U,U) - Q_{S,R}(V,V) ) +{\rm o}(S_k^\beta)\\ \notag
    &= \frac{1}{2}(a^2 - \tilde{a}^2)\gamma_1^-S^\beta+\frac{1}{2}\left[\gamma_1^-(\tilde{a}^2 - a^2) + \frac{\beta (a-\tilde{a})^2}{1-(S/R)^{\beta}}\right]S^\beta+{\rm o}(S^\beta)\\ \notag
    &= \frac{1}{2}\frac{\beta (a-\tilde{a})^2}{1-(S/R)^{\beta}}S^\beta+{\rm o}(S^\beta)\to \infty
\end{align}
as $S\to \infty$ since $a\neq \tilde{a}$.

In the non-strictly stable case, as explained at the beginning of the proof, by the same reasoning we instead obtain
\[
    \mathcal{H}^n(M) - \mathcal{H}^n(\graph_C^{\rho_0,R}(q\phi_1)) = \frac{1}{2}\frac{\beta (a-\tilde{a})^2}{(1-\log S/\log R)} \log S +{\rm o}(\log S)\to \infty\,.
\]
\end{proof}

\begin{remark}
    Note that if $C$ is not strictly minimizing then the asymptotic behavior of $T_\lambda$ is instead given by \eqref{e:non-str-min}, so our final lower bound on the difference of areas would have $\gamma_1^+$ in place of $\gamma_1^-$, resulting in a lower bound on the difference of areas that will instead decay to zero as $S \to \infty$.
\end{remark}

\begin{proof}[Proof of Theorem \ref{thm:main thm2}]
The proof is divided into two steps. In the first, we construct locally minimizing $(1,2)$-clusters decaying to the cone at the same rate as $T_\lambda$ for every $|\lambda|$ sufficiently small, and with volumes uniformly bounded away from zero. In the second step, we argue that rescaling each of these so that the volume of the first chamber is one yields infinitely many distinct local minimizers (up to rigid motions).

\medskip

\proofstep{Step one: existence of infinite (not necessarily distinct) continuum of local minimizers with prescribed decay.}
Let $\lambda_0$ be as in section~\ref{sec: construction} and $\lambda \in [-\lambda_0, \lambda_0].$ {Let $q=ar^{-\gamma_1^-} + \rm{o}(r^{-\gamma_1^-})$ and $R_0>0$ be such that $T_\lambda$ satisfies \eqref{eq:pure first mode hypothesis}. We claim that for any such $\lambda$, there exists a locally minimizing $(1,2)$-cluster $\X_\lambda$ as in Theorem \ref{thm:main thm} with $\partial\X_\lambda(2)\cap \partial\X_\lambda(3)\cap B_\rho^c\subset \graph_C v$ where $v = q\phi_1 + o(r^{-\gamma_1^-})$. Suppose for a contradiction that this is not the case.} Repeating the argument of the proof of Theorem~\ref{thm:main thm}, we excise at most $\lceil 1/\bar{v}\rceil-1$ lens-like concentrations to arrive at the family of variational problems $\tau_{R_k, \lambda}(v)$ with $v\geq \bar{v}$ for which $\lambda_*<\infty$ (as defined in \eqref{eqn: lambda star def}) and whose  corresponding minimizers $\X_{\lambda,k}^v$ converge to the claimed  minimizing $(1,2)$-cluster $\X_\lambda$ of Theorem~\ref{thm:main thm}. There is a radius $\rho_0$ such that $\X_{\lambda,k}(1), \X_{\lambda}(1)\subset B_{\rho_0}.$

The exterior interfaces ${M}_k:= \partial \X_{\lambda,k}(2)\cap \partial \X_{\lambda,k}(3)$, and hence additionally the limiting exterior interface $M_\infty = \partial \X_\lambda(2) \cap \partial \X_\lambda(3)$, are trapped between the leaves $T_{\pm 2\lambda_*}$ of the foliation. In particular, ${M}_\infty \subset \graph_C u_\infty$ on $B_{\rho_0}^c$. Now, the decay $|u_\infty|\lesssim r^{-\gamma_1^-}$ together with \cite[Lemma 2.10 and Claim 3.3]{SimonSolomon}, both of which hold for exterior minimal surfaces with no modifications to the proof,
guarantee that there is a constant $\tilde{a} \in \R$ with $\tilde{a} \neq a$ such that 
\begin{equation}\label{eqn: u infinity asymptotics}
    u_\infty(r, \omega) = \tilde{a} \phi_1(\omega) r^{-\gamma_1^-} + {\rm o}(\gamma_1^-).
\end{equation}
Here we are using that $C_{k,l}$ is an area-minimizing quadratic cone in order to apply the results of \cite{SimonSolomon}, together with the contradiction assumption to say that $\tilde a \neq a$.

Therefore, by the local convergence of ${M}_k$ to $\graph_C u_\infty$ on $B_{\rho_0}^c$ (which can be improved to smooth convergence by {De Giorgi's/}Allard's {epsilon} regularity), there are $S_k\to \infty$ such that $S_k/R_k\to 0$  and {so that $M_k \subset \graph_C^{\rho_0,S_k} u_k$ and \eqref{eq:uk vanishes} holds for $u=u_k$ with $S=S_k$ for $k$ sufficiently large.}
Note that any area-minimizing quadratic cone is strictly minimizing and strictly stable, thanks to {\cite[Theorem 3.2, Remark 3.3]{HardtSimon84} and} \cite{SimonSolomon}*{Proposition 2.7} (see also \cite{EdSp}*{Section 2.5}). 
Moreover, the hypothesis \eqref{eq:pure first mode hypothesis} holds for all area-minimizing quadratic cones, since the works \cite{BDGG,DPP, Lawson72, Davini04} guarantee that the leaves of the foliation inherit the $SO(\Sbb^{k+1})\times SO(\Sbb^{l+1})$ symmetry of the cone asymptotically, which in particular implies the hypothesis, since $\phi_1$ is constant. 
Now, for each such $k$ let $M_{k,1}$ to be $\graph_C^{\rho_0,S_k}$ and let $M_{k,2}$ be an area-minimizing integral current with $\partial \llbracket M_{k,2} \rrbracket = \llbracket \graph_{C\cap \partial B_{7R_k/2}} q\phi_1 \rrbracket + \llbracket \graph_{C\cap \partial B_{S_k}} u_k \rrbracket$. In light of the above discussion, we may apply Lemma \ref{l:area-gap-Plateau-leaf 2} for $k$ sufficiently large to the surface $\tilde{M}_k := M_{k,1} \sqcup M_{k,2}$ with $R= 7R_k/2$, to obtain
 
	\begin{align*}\label{eq:SK estimate}
\LT_{T_\lambda}(v) &= \lim_{k\to \infty}	(\mathcal{E}_{R_k}(\X_{\lambda,k}) -\mathcal{H}^n(T_\lambda \cap B_{4R_k})) \\
 & \geq
 \lim_{k\to \infty}	(\mathcal{H}^n(\tilde{M}_k) -\mathcal{H}^n( \graph_C^{\rho_0,7R_k/2}(q\phi_1) ))- c_{\rho_0} \\
 & =+\infty
	\end{align*}
for some $\rho_0$-dependent constant $c_0$ (independent of $k$), where we have also used that 
\[
    (M_{k,2}\setminus \graph_C^{\rho_0,7R_k/2}(q\phi_1)) \cap B_{4R_k} = (T_\lambda\setminus \graph_C^{\rho_0,7R_k/2}(q\phi_1)) \cap B_{4R_k}\,.
\]
This gives the desired contradiction.
\medskip

\proofstep{Step two: existence of infinitely many distinct local minimizers with unit volume first chambers.} For $|\lambda|<\lambda_0$, let $\X_{\lambda}$ denote the locally minimizing $(1,2)$-cluster with exterior interface decaying to $C$ at the same leading order rate as $T_\lambda$. We define an equivalence relation $\sim$ on these clusters by saying $\X_{\lambda} \sim \X_{\lambda'}$ if there is $r>0$ such that, up to rigid motions, $\X_{\lambda}=r\X_{\lambda'}$. To finish the proof of the theorem, it suffices to show that there are infinitely many equivalence classes. Suppose for contradiction that this were not the case. Then there must exist non-zero $\lambda_j\to 0$, $r_j>0$, and $\lambda$ such that $r_j\X_{\lambda_j} = \X_{\lambda}$ for each $j$. By scaling,
\begin{equation*}
   r_j = \left(|\X_{\lambda}(1)|/|\X_{\lambda_j}(1)|\right)^{1/(n+1)} \,.
\end{equation*}
Since $|\X_{\lambda}(1)|\in [\overline{v},1]$ for all $|\lambda|<\lambda_0$ by Remark \ref{rmk: uniform rescaling} with $\overline{v}>0$, we have
\begin{equation}\label{eq:liminf rj bound}
   0<\limsup_{j\to \infty} r_j\leq  \limsup_{j\to \infty} r_j <\infty\,.
\end{equation}
On the other hand, let $a_j$ and $a$ denote the coefficients in front of $\phi_1 r^{-\gamma_1^-}$ in the leading order term of the graphical expansions of $\partial \X_j(2) \cap \partial \X_j(3)$ and $\partial \X(2) \cap \partial \X(3)$ respectively. Since the leading order decay of $\X_{\lambda_j}$ must match the leading order decay of $\X_{\lambda_i}$, we have
\begin{equation}\label{eq:aj equation}
    r_j a_j = a\,.
\end{equation}
If $a=0$, then \eqref{eq:aj equation} implies that $r_j=0$, contradicting \eqref{eq:liminf rj bound}, and if $a>0$, then \eqref{eq:aj equation} implies that $\limsup_{j\to \infty}r_j = \infty$, again contradicting \eqref{eq:liminf rj bound}. Thus our original assumption that there are finitely many equivalence classes cannot hold, and the proof is complete. 
\end{proof}

\section{Proof of Theorem \ref{t:even-dim-nonuniqueness}}
\subsection{$\Lambda_{C_{k,k}}$ estimate in all even dimensions $n+1 = 2k+2$}
To simplify notation slightly, we will henceforth set $m:=n+1\geq 8$, with $m$ an even integer. {Note that in our preceding work \cite{BNNS}, joint with Lia Bronsard, the ambient dimension is $n$ in place of $n+1$.} We recall from \cite{BNNS} that 
\begin{equation}\label{e:Lambda_plane}
\Lambda_{\text{plane}}(m) = \frac{2(m-1)\omega_{m-1} \int_{1/2}^1 \big(1-t^2\big)^{\frac{m-3}{2}}\,dt - \omega_{m-1} \left(\frac{\sqrt{3}}{2}\right)^{m-1}}{\left(2\omega_{m-1} \int_{1/2}^1 \big(1-t^2\big)^{\frac{m-1}{2}} \,dt \right)^{\frac{m-1}{m}}}
\end{equation}
With the substitution $t=\cos \theta$, we may rewrite the right-hand side of \eqref{e:Lambda_plane} as $\omega_{m-1}^{1/m}
\frac{(m-1)P_{m-2}-(\frac{\sqrt3}{2})^{m-1}}{P_m^{(m-1)/m}}$
where here and in the sequel we let
\begin{equation}\label{eqn: P def}
P_m:=2\int_0^{\pi/3}\sin^m\theta\,d\theta.
\end{equation}
Moreover, integrating by parts shows that we have the recursion $mP_m = (m-1)P_{m-2} - (\frac{\sqrt{3}}{2})^{m-1}$ and therefore 
\begin{equation}\label{eqn: lambda plane}
 \Lambda_{\text{plane}}(m)=m\,\omega_{m-1}^{1/m}P_m^{1/m}.
\end{equation}

Next, in \cite{BNNS} (see also the proof of the validity of \eqref{e:Lambda-cone-vs-Lambda-lens} for $C_{3,3}\times \R$ from Theorem \ref{thm:main thm} below) we constructed a competitor whose normalized energy (i.e. $D_m = M(k,k)$ for $2k+2=m$ with $M(k,l)$ as in \eqref{e:M(k,l)}) is  
\[
D_m :=
\omega_{m/2}^{2/m}\frac{\frac{m^2}{2}r
\int_1^{r-h}u^{(m-2)/2}\bigl(r^2-(u+h)^2\bigr)^{(m-4)/4}\,du-\frac{m^2\sqrt2}{4(m-1)}}
{\Big(1+m\int_1^{r-h}u^{(m-2)/2}\bigl(r^2-(u+h)^2\bigr)^{m/4}\,du\Big)^{(m-1)/m}}\,.
\]
Here \[
h:=1+\sqrt3,\qquad r:=\sqrt6+\sqrt2=\sqrt2\,h.
\]
We simplify this expression with the change of variables $u=hx$ to obtain
\[
D_m
=
\om_{m/2}^{2/m}\Big(\frac{m^2\sqrt2}{2}\Big)
\frac{
h^{m-1}\int_a^b x^{(m-2)/2}Q(x)^{(m-4)/4}\,dx-\frac{1}{2(m-1)}
}{
\Big(
1+m h^m
\int_a^b x^{(m-2)/2}Q(x)^{m/4}\,dx\Big)^{(m-1)/m}
}
\]
where we have introduced the shorthand notation
\begin{equation}
    \label{eqn: abQ}
a:=\frac1h=\frac{\sqrt3-1}{2},\qquad b:=\sqrt2-1,\qquad Q(x):=1-2x-x^2.
\end{equation}
As we did for $\Lambda_{\text{plane}}(m)$, we rewrite the numerator as the integral of a positive function. To this end, define the auxiliary functions
\begin{equation}\label{eqn: W and G}
W_m(x):=x^{(m-2)/2}Q(x)^{(m-4)/4}, \qquad
G_m(x):=1-\frac{mh}{4(m-1)}(3x+2x^2-1).
\end{equation}
Direct computation verifies that $(x^{m/2}Q(x)^{m/4})'
=
\frac{m}{2} W_m(x)(1-3x-2x^2),$ and $Q(a)=a^2$, $Q(b)=0$, so that
\[
a^m
=
\frac m2\int_a^b W_m(x)(3x+2x^2-1)\,dx.
\]
Recalling that $a =1/h$, the numerator of $D_m$  becomes
\begin{equation}
    \label{eqn: M def}
M_m:= h^{m-1}
\int_a^b W_m(x)G_m(x)\,dx,
\end{equation}
Finally, we introduce the shorthand 
\begin{equation}
    \label{eqn: L def}
L_m:=h^{m}\displaystyle\int_a^b x^{(m-2)/2}Q(x)^{m/4}\,dx
\end{equation}
for the integral in the denominator of $D_m$.
With this notation, $D_m$ takes the consolidated form
\begin{equation}\label{eqn: B}
D_m=
\om_{m/2}^{2/m}\left( \frac{m^2\sqrt2}{2}\right)
\frac{M_m}{(1+mL_m)^{(m-1)/m}}.
\end{equation}

Our goal is to show that $D_m <\Lambda_{\text{plane}}(m)$ in every even dimension $m$.  Recall that in \cite{BNNS}, we verified this computationally up to $m=n+1=2700$. Here we prove:

\begin{theorem}\label{thm: asymptotic inequality}
For every even $m\ge {2000}$, 
$\Lambda_{\text{plane}}(m)>D_m.$
\end{theorem}
{Since $2001 <2700$, the validity of $\Lambda_{C_{k,k}} < \Lambda_{\plane}(n+1)$ for all even $n+1 = 2k+2$ follows immediately from Theorem \ref{thm: asymptotic inequality} and the main results of \cite{BNNS}.}

The desired inequality is equivalent to  $\log(\Lambda_{\text{plane}}(m)/D_m) >0$. We begin with formal asymptotics of this quantity as $m\to \infty$. 
From \eqref{eqn: lambda plane} and \eqref{eqn: B}, we have
\begin{equation}
    \label{eqn: log ratio}
\log\frac{\Lambda_{\text{plane}}(m)}{D_m}=
-\log m
 +\frac{1}{m}\log\frac{\om_{m-1}}{\om_{m/2}^2}+\frac1m\log P_m
-\log M_m
+
\frac{m-1}{m}\log(1+mL_m)-\log\frac{\sqrt2}{2}.
\end{equation}
Recall that $P_m, M_m,$ and $L_m$ were defined in \eqref{eqn: P def}, \eqref{eqn: M def}, and \eqref{eqn: L def} respectively. 
 Stirling's formula shows that the first term on the right-hand side of \eqref{eqn: log ratio} behaves as
\begin{equation}\label{e:asymptotic1}
\frac1m\log\frac{\om_{m-1}}{\om_{m/2}^2}=-\log \sqrt{2} + \frac{\log m}{m} -\frac{1}{m} \log (2\sqrt 2) + O(m^{-2}).
\end{equation}

Moreover, using Laplace's method (in the endpoint case; see Chapter 4.3 of Asymptotic Expansions of Integrals by de Bruijn), one shows that the functions $P_m, M_m,$ and $L_m$ have asymptotics
\begin{align}
     P_m & =  \frac{2\sqrt3}{m}\left(\frac{\sqrt{3}}{2}\right)^m(1+O(m^{-1})), \label{e:asymptotic2}\\
 M_m& =\frac{3}{2m}
\left[ 1-\frac{7-\frac{4\sqrt3}{3}}{m}+O(m^{-2})\right],\label{e:asymptotic3}\\   
1+mL_m &= \sqrt3\left[
1-\frac{6-\frac4{\sqrt3}}{m}+O(m^{-2})
\right].\label{e:asymptotic4}
\end{align}
Below, we will verify these estimates and obtain quantitative errors on the remainders, in order to obtain the explicit threshold on $m$ given by Theorem \ref{thm: asymptotic inequality} above which we have $\Lambda_{\plane}(m)>D_m$. Let us first see that the asymptotics \eqref{e:asymptotic1}-\eqref{e:asymptotic4} yield $\Lambda_{\plane}(m)>D_m$ for a sufficiently large (not quantified) dimension $m$. Substituting the asymptotics into the right-hand side of \eqref{eqn: log ratio}, we see that the terms with coefficient $\log m$, the constant terms, and the terms with coefficient $\frac{\log m}{m}$ cancel (see the proof of Theorem~\ref{thm: asymptotic inequality} below for more details) and the leading contribution is of order $\frac{1}{m}$:
\[
\log\frac{\Lambda_{\text{plane}}(m)}{D_m}
=\frac{1-\log \sqrt{2}}{m}+O(m^{-2})\,.
\]
As this coefficient is positive, the desired inequality holds asymptotically as $m\to \infty$. The asymptotics above guide the estimates proven in the following lemma, which will allow us to show the estimate holds for $m$ above an explicit threshold.
\begin{lemma}\label{lemma: ests}
    The following estimates hold:
    \begin{align}
    \label{eqn: omega bound}
\frac1m\log\frac{\omega_{m-1}}{\omega_{m/2}^2}
\ge
-\log \sqrt{2}+  \frac{\log m}{m} -  \frac{\log ( 2\sqrt{2})}{m}  & \qquad\text{ for all } m\geq 1,\\
\label{eqn: P bound}
P_m\ge \bigg(\frac{\sqrt{3}}{2}\bigg)^m\frac{2\sqrt3}{m}\left(1-\frac{13}{m}\right)  & \qquad\text{ for all } m\geq 1,\\
\label{eqn: M bound}
M_m\le \frac{3}{2m}\left(1-\frac{4.1}{m}\right) & \qquad\text{ for all } m\geq 2000 ,\\
\label{eqn: L bound}
1+mL_m \ge \sqrt3\left[ 1-\frac{6.03-\frac4{\sqrt3}}{m}
\right] & \qquad\text{ for all } m\geq 2000\, .
    \end{align}
\end{lemma}

Before proving the lemma, let us see how it yields Theorem~\ref{thm: asymptotic inequality}.
\begin{proof}[Proof of Theorem~\ref{thm: asymptotic inequality}]
{Let $m \geq 2000$.} Starting from \eqref{eqn: log ratio} and substituting \eqref{eqn: omega bound}, we find
\begin{equation}
    \label{eqn: log ratio 2}
\log\frac{\Lambda_{\text{plane}}(m)}{D_m}\geq -\log m +
  \frac{\log m}{m} -\frac{\log ( 2\sqrt{2})}{m} 
+
\frac1m\log P_m
-\log M_m
+
\frac{m-1}{m}\log(1+mL_m).
\end{equation}
Now, we take the logarithm of \eqref{eqn: P bound}. Noting that for $m \geq 26$ (i.e. $13/m\le 1/2$) we have $\log\left(1-\frac{13}{m}\right) \geq -26/m$ since  $\log(1-u)\ge -2u$ for $0\le u\le 1/2$, we arrive at the bound
\[
\frac1m\log P_m
\ge
\log \frac{\sqrt{3}}{2}-\frac{\log m}{m}+\frac{\log(2\sqrt3)}{m}
-\frac{26}{m^2}\,.
\]
Similarly, we take the logarithm of \eqref{eqn: M bound} and use concavity of the logarithm, which in particular gives $\log(1-x) \leq -x$, to obtain
\[
-\log M_m \geq 
\log m-\log\frac32+\frac{4.1}{m}.
\]
Summing these terms and combining with \eqref{eqn: log ratio 2} yields
\begin{equation}
    \label{eqn: log ratio 3}
\log\frac{\Lambda_{\text{plane}}(m)}{D_m}\geq -\log \sqrt{3} +\frac{4.1 +\log ( \sqrt{3/2})}{m} +
\frac{m-1}{m}\log(1+mL_m) - \frac{26}{m^2}.
\end{equation}
Finally, taking the logarithm of \eqref{eqn: L bound} and noticing that we have the bounds
\[
    \begin{cases}
        \log(1-t)\geq -t -t^2 & \text{for $t\in (0,1/2)$};\\
        \frac{6.03-\tfrac{4}{\sqrt3}}{m} \leq \frac{1}{2} & \text{for $m\geq 8$}; \\
        (6.03-\tfrac{4}{\sqrt3})^2 < 14 \,,
    \end{cases}
\]

we obtain
\begin{align*}
\frac{m-1}{m}\log(1+mL_m)
& \geq \log\sqrt{3} - \frac{1}{m}\log\sqrt{3} +\log\left(1- \frac{6.03-\frac4{\sqrt3}}{m}\right)\\
& \geq
\log\sqrt{3} + \frac{1}{m}\left(-6.03 + \frac4{\sqrt3}-\log\sqrt3\right)
-\frac{14}{m^2}\,.
\end{align*}
Combining this with \eqref{eqn: log ratio 3} yields
\begin{equation}
    \label{eqn: log ratio 4}
\log\frac{\Lambda_{\text{plane}}(m)}{D_m}\geq \frac{1}{m}\left[4.1 -\log ( \sqrt{2}) +\frac4{\sqrt3}-6.03\right] - \frac{40}{m^2} \,.
\end{equation}
The coefficient $C_0:= 4.1 -\log ( \sqrt{2}) +\frac4{\sqrt3}-6.03$ of \(1/m\) is positive, indeed an explicit calculation yields
\[
4.1-6.03+\frac4{\sqrt3}-\log \sqrt{2}>0.032.
\]
Thus
\[
\log\frac{\Lambda_{\text{plane}}(m)}{D_m}
\ge
\frac{C_0}{m}-\frac{40}{m^2} > \frac{1}{m} \left(0.032 - \frac{40}{m}\right)\,.
\]
When $m\ge 1251$, we have $0.032 - \frac{40}{m} > 0$, thus completing the proof.
\end{proof}

Now we prove the lemma.
\begin{proof}[Proof of Lemma~\ref{lemma: ests}]
{\it Proof of \eqref{eqn: omega bound}:}
Recall that $\omega_m=\frac{\pi^{m/2}}{\Gamma(1+m/2)}$. Using this and the identity $\Gamma(2z) = 2^{2z-1}\pi^{-1/2}\Gamma(z)\Gamma(z+\frac{1}{2})$ 
with $z=(m+1)/4$, we have
\[
\frac{\omega_{m-1}}{\omega_{m/2}^2}
=
\frac{1}{2^{(m-1)/2}}
\frac{\Gamma\left(\frac m4+1\right)^2}
{\Gamma\left(\frac m4+\frac14\right)\Gamma\left(\frac m4+\frac34\right)}.
\]
Next, applying Gautschi's inequality, which states that $\frac{\Gamma(x+1)}{\Gamma(x+s)}\ge x^{1-s}$ for all $x>0$ and $0<s<1$, with the choices $x=m/4$ and $s=1/4, 3/4$ yields
\[
\frac{\Gamma\left(\frac m4+1\right)^2}
{\Gamma\left(\frac m4+\frac14\right)\Gamma\left(\frac m4+\frac34\right)} \geq \frac{m}{4}.
\]
Therefore $\frac{\omega_{m-1}}{\omega_{m/2}^2}
\ge
\frac{m}{2^{(m+3)/2}}.$
Taking logarithms and dividing by $m$ gives \eqref{eqn: omega bound}.
\\

{\it Proof of \eqref{eqn: P bound}:} Since the integrand of $P_m$ is maximized at the right endpoint $\pi/3$ where $\sin(\pi/3) = \sqrt{3}/2$, it is convenient to change variables $x=\pi/3-\theta$ and factor out this contribution, expressing $P_m$ as  
\begin{equation}
    \label{eqn: Pn new}
P_m=2\bigg(\frac{\sqrt{3}}{2}\bigg)^m\int_0^{\pi/3}\bigg(\frac{2}{\sqrt{3}}\sin\Big(\frac{\pi}{3}-x\Big)\bigg)^m\,dx = 2\bigg(\frac{\sqrt{3}}{2}\bigg)^m\int_0^{\pi/3} \exp\{m\phi(x)\}\,dx
\end{equation}
where
\[
\phi(x):=\log\frac{2\sin(\pi/3-x)}{\sqrt{3}}.
\]
With the aim of bounding the integral below, we note that $
\phi'(0)=-\frac1{\sqrt3}$ and $\phi''(x)=-\csc^2(\pi/3-x),$ so in particular on the interval $[0,\pi/6]$ we have $\phi''(x)\ge -4$ and thus $\phi(x)\ge -\frac{x}{\sqrt3}-2x^2$ by Taylor's Theorem.
Thus
\[
\int_0^{\pi/3} \exp\{m\phi(x)\}\,dx
\ge
\int_0^{\pi/6}\exp\Big\{-\frac{mx}{\sqrt3}-2mx^2\Big\}\,dx.
\]
Now we make the change of variables $y=mx$ and further truncate the interval of integration: noting that $m\pi/6 \geq \sqrt{3}\log m$ for all $m\ge 6$, and recalling that $e^{-z}\ge 1-z$, we have
\begin{align*}
\int_0^{\pi/6}\exp\Big\{-\frac{mx}{\sqrt3}-2mx^2\Big\}\,dx
&\geq
\frac1m\int_0^{\sqrt3\log m} \exp\Big\{-\frac{y}{\sqrt3} -\frac{2y^2}{m}\Big\}\,dy\\
&\geq 
\frac1m \int_0^{\sqrt3\log m} \exp\Big\{-\frac{y}{\sqrt3}\Big\}\left(1-\frac{2y^2}{m}\right)dy\\
&\geq 
\frac1m \int_0^{\sqrt3\log m} \exp\Big\{-\frac{y}{\sqrt3}\Big\}dy
-\frac{2}{m^2} \int_0^\infty \exp\Big\{-\frac{y}{\sqrt3}\Big\}y^2 dy \\
&=\sqrt3\left(\frac{1}{m}-\frac{1}{m^{2}}\right) - \frac{12\sqrt{3}}{m^2}\,.
\end{align*}
In conclusion, we obtain
\begin{equation}\label{e:exp-lb}
\int_0^{\pi/3} \exp\{m\phi(x)\}\,dx
\ge
\frac{\sqrt3}{m}\left(1-\frac{13}{m}\right)
\end{equation}
for \(m\ge 6\). On the other hand, for \(1\le m\le 5\), $1-\frac{13}{m}$ is negative, so \eqref{e:exp-lb} holds for all integers $m \geq 1$. Together with \eqref{eqn: Pn new}, this completes the proof.\\

{\it Proof of \eqref{eqn: M bound}:} We proceed in several steps.\\

{\it Step 1:}
Define the auxiliary functions
\begin{equation}
    \label{eqn: Phi def}
    H_m(x):=xG_m(x), \qquad
\Phi(x):=\frac12\log x+\frac14\log Q(x),
\end{equation}
with $G_m$ as in \eqref{eqn: W and G} and $Q$ as in \eqref{eqn: abQ}.
Then, observing that $W_m(x)=x\,e^{(m-4)\Phi(x)}$, we can write
\[
M_m=h^{m-1}\int_a^b H_m(x)e^{(m-4)\Phi(x)}\,dx,
\]
where we recall that $a=\frac{\sqrt3-1}{2} = \tfrac{1}{h}$ and $b=\sqrt2-1$, as in \eqref{eqn: abQ}. Since we additionally have $
e^{\Phi(a)}=\frac1h$ and $-\Phi'(a)=\frac{h^2}{2}=2+\sqrt3$, we change variables $x=a+t$ and write
\[
M_m
=
h^3\int_0^{b-a}H_m(a+t)e^{(m-4)\Phi'(a) t}e^{(m-4)R(t)}\,dt,
\]
where
\[
R(t):=\Phi(a+t)-\Phi(a) - \Phi'(a)t
\]
denotes the remainder of $\Phi(a+t)$ from its linear approximation at $\Phi(a)$.
We record for later use that 
\begin{align}
    H_m(a) &= a\,\frac{3m-4}{4(m-1)},
    \qquad H_m'(a) =
-\frac{1+am}{m-1}\,, \qquad m \geq 2\,, \label{eqn: H values}\\
\Phi''(a) &= -8 \left(4 \sqrt{3}+7\right) \in [-112,-111]\,.\label{eqn: Phi''(a)}
\end{align}
Letting $c:= - \Phi'(a)$ (which is positive), our main claim is that
\begin{equation}
    \label{eqn: main m est}
        M_m \leq h^3\left[\frac{H_m(a)}{(m-4)c}+
\frac1{(m-4)^2}\left(
\frac{H_m'(a)}{c^2}+\frac{\Phi''(a) H_m(a)}{c^3}
\right)
\right]
+\frac{2h+655}{(m-4)^3}
\end{equation}
provided that $m\geq 2000.$
\\

{\it Step 2:} We split $M_m$ into pieces and bound each piece separately. To begin, let
\[
Y:=\frac{2\log (m-4)}{c}.
\]
and split $M_m = M_m^{\rm loc} + M_m^{\rm tail}$ where
\[
M_m^{\rm tail}:= h^3 \int_{Y/(m-4)}^{b-a} H_m(a+t)e^{-(m-4)ct}e^{(m-4)R(t)}\,dt\,.
\]
Note that the lower endpoint $\frac{Y}{m-4}$ of the integral is less than $b-a$ provided that $m \geq 46$. Otherwise, we simply set $M_m^{\rm tail}=0$.

  Now, observe that $\Phi$ is concave so $R(t) \leq 0$, and $|H_m|\leq 1$ on $[a,b]$. 
So, since $c= \tfrac{h^2}{2}$,
\begin{equation}\label{eqn: tail bound}
M_m^{\rm tail}
\leq h^3\int_{Y/(m-4)}^\infty e^{-(m-4)ct}\,dt
=\frac{h^3}{(m-4)c}e^{-cY}=\frac{h^3}{(m-4)c}(m-4)^{-2}= \frac{2h}{(m-4)^3}.
\end{equation}

{\it Step 3:} Toward estimating $M_m^{\rm loc}$, first note that  $
H_m(x)=x-\frac{mh}{4(m-1)}(3x^2+2x^3-x)$  is concave on $[a,b]$ since $H_m''(x)=-\frac{mh}{4(m-1)}(6+12x)$, and hence 
\[
H_m(a+t) \leq H_m(a)+H_m'(a)t\,.
\]

As a consequence, letting $s=(m-4)t$,
\begin{equation}\label{eqn: Mloc 1}
\begin{split}
      M_m^{\rm loc} &\leq h^3\int_0^{Y/(m-4)}\{H_m(a) +H_m'(a)t\} e^{-(m-4)ct}e^{(m-4)R(t)}\,dt\\
    & = \frac{h^3}{m-4}\int_0^{Y}\Big\{H_m(a)+H_m'(a)\frac{s}{m-4}\Big\} e^{-cs}e^{(m-4)R(s/(m-4))}\,ds 
\end{split}    
\end{equation}
We examine the error term $R(t)$.
First of all, for $m\geq 707,$ we have $Y/(m-4) \leq 0.005$.
Direct computation verifies that 
\[
    {\Phi'''(x) = \frac{1}{x^3} - \frac{(x+1)(x^2+2x+7)}{(1 - 2x - x^2)^3}}
\]
is a negative and decreasing function on $[a,a+0.005]$ and thus $|\Phi'''|\leq |\Phi'''(a+0.005)| \leq 6200.$
As a consequence, since $6200/6\leq 1100$, Taylor's Theorem gives
\[
R(t)=\frac{\Phi''(a)}{2}t^2+\rho_\Phi(t),
\qquad
|\rho_\Phi(t)|\le \max_{s \in [0, 0.005]} |\Phi'''(s)|\, \frac{t^3}{6} \leq 1100t^3
\] 
for $t \in [0, 0.005]$. 
In particular, recalling the bound \eqref{eqn: Phi''(a)}, for $s\in[0,Y]$  we have
\[
|(m-4)R(s/(m-4))| \leq \frac{|\Phi''(a)|}{2}\frac{(2\log(m-4)/c)^2}{m-4}+ 1100\frac{(2\log(m-4)/c)^3}{(m-4)^2} \leq \frac{1}{2}\,,
\]
when $m\geq 2000.$
Consequently, since another application of Taylor's Theorem yields $e^t \leq 1+ t+t^2$ for $|t|<1/2$, we have 
\[
e^{(m-4)R(s/(m-4))} \leq 1 + \frac{\Phi''(a)}{2}\frac{s^2}{m-4}+1100\frac{s^3}{(m-2)^2} + \left(\frac{|\Phi''(a)|}{2}\frac{s^2}{m-4}+1100\frac{s^3}{(m-4)^2}\right)^2.
\]
Inserting this bound into \eqref{eqn: Mloc 1} and using that $H_m(a) +H_m'(a) t>0$ on $[a,a+0.005]$ for all {$m\geq 6$}, this yields
\begin{align*}
    M_m^{\rm loc} \leq  M_m^{\rm main}+   M_m^{\rm err}\,,
\end{align*}
where 
\begin{equation}
    \begin{split}
 M_m^{\rm main} 
&=\frac{h^3}{m-4}\Big\{H_m(a) \int_0^Y e^{-cs} \,ds 
+\frac{H_m'(a)}{m-4}  \int_0^Y e^{-cs} s \,ds +\frac{H_m(a)\Phi''(a)}{2(m-4)}  \int_0^Y e^{-cs} s^2 \,ds\Big\}       
    \end{split}
\end{equation}
and
\begin{equation}\label{eqn: M error}
    \begin{split}
  M_m^{\rm err} &= \frac{h^3}{(m-4)^3} \frac{| H_m'(a)|| \Phi''(a)|}{2} \int_0^{\infty} e^{-cs}{s^3}\,ds \\
  &\qquad+   \frac{h^3}{m-4} H_m(a)\int_0^{\infty} e^{-cs}\left(\frac12|\Phi''(a)|\frac{s^2}{m-4}+1100\frac{s^3}{(m-4)^2}\right)^2\, ds   .    
    \end{split}
\end{equation}
{Here, we have used in addition that $H_m'(a) < 0$, so in particular
\[
    \frac{h^3}{(m-4)^2} H_m'(a)\int_0^{\infty} e^{-cs} s\left(\frac12|\Phi''(a)|\frac{s^2}{m-4}+1100\frac{s^3}{(m-4)^2}\right)^2\, ds < 0\,.
\]}

{\it Step 4:}
Let us first bound \eqref{eqn: M error}. 
Using
\begin{equation}\label{e:Gaussian-moments}
\int_0^\infty y^k e^{-cy}\,dy=\frac{k!}{c^{k+1}},
\end{equation}
and expanding the square $(\frac12|\Phi''(a)|\frac{s^2}{m-4}+1100\frac{s^3}{(m-4)^2})^2 = \frac{|\Phi''(a)|^2}{4(m-4)^2} s^4 +\frac{(1100)^2}{(m-4)^4}s^6 + 1100\tfrac{|\Phi''(a)|}{(m-4)^3} s^5$, we get
\begin{align*}
M_m^{\rm err}
&\le
\frac{h^3}{m-4}\left[ \frac{| H_m'(a)|| \Phi''(a)|}{2(m-4)^2} \frac{6}{c^4}+ H_m(a)\Big(\frac{|\Phi''(a)|^2}{4(m-4)^2}\frac{24}{c^5} + \frac{(1100)^2}{(m-4)^4}\frac{720}{c^7} + \frac{1100|\Phi''(a)|}{(m-4)^3} \frac{120}{c^6}\Big)\right]\,.
\end{align*}
{Recalling \eqref{eqn: H values} and \eqref{eqn: Phi''(a)}, in particular for $m\geq 4$ we have the crude estimates}
\begin{equation}
    \label{eqn: rough bounds}
0<H_m(a)  \leq \frac{3a}{4},\qquad |H_m'(a)|\le 1,\qquad -\Phi''(a)\leq 112,
\end{equation}
which in turn give 
\begin{equation}
    \label{eqn: M err bound}
M_m^{\rm err}
 \le
\frac{h^3}{m-4}\Big\{ \frac{112}{{2}(m-4)^2} \frac{6}{c^4}+ \frac{3a}{4}\Big\{\frac{(112)^2}{4(m-4)^2}\frac{24}{c^5} + \frac{(1100)^2}{(m-4)^4}\frac{720}{c^7} + 112\cdot\frac{1100}{(m-4)^3} \frac{120}{c^6}\Big\} \Big\}\leq \frac{650}{(m-4)^3}
\end{equation}
for $m \geq 2000.$

{\it Step 5:}
We may next evaluate $M_m^{\rm main}$ directly. Keeping in mind that 
\begin{align}
\label{eqn: integrals}
        \int_0^Y e^{-cy}\,dy & =\frac{1-e^{-cY}}{c} \leq \frac{1}{c},\\
\nonumber \int_0^Y ye^{-cy}\,dy &=\frac1{c^2}-e^{-cY}\left(\frac{Y}{c}+\frac1{c^2}\right),\\
\nonumber \int_0^Y y^2e^{-cy}\,dy &=
\frac{2}{c^3}-e^{-cY}\left(
\frac{Y^2}{c}+\frac{2Y}{c^2}+\frac{2}{c^3}
\right),
\end{align}
we arrive at
\begin{align}
\nonumber M_m^{\rm main}
&\leq
h^3\left[
\frac{H_m(a)}{(m-4)c}+\frac1{(m-4)^2}\left(
\frac{H_m'(a)}{c^2}
+
\frac{H_m(a) \Phi''(a)}{c^3}
\right)
\right] \\
&\quad
\label{eqn: error} +
\frac{h^3e^{-cY}}{(m-4)^2}
\left[
-H_m'(a)\left(\frac{Y}{c}+\frac1{c^2}\right)
-\frac12 H_m(a) \Phi''(a)
\left(
\frac{Y^2}{c}+\frac{2Y}{c^2}+\frac{2}{c^3}
\right)
\right]\,.
\end{align}
Note that {$H_m(a)>0$, while $H_m'(a)<0$} and $\Phi''(a)<0$, so each term in \eqref{eqn: error} is positive.  Recall that $Y=\tfrac{2\log (m-4)}{c}$, so $e^{-cY}=(m-4)^{-2}$. Again using the crude bounds in \eqref{eqn: rough bounds},
the term in \eqref{eqn: error} is bounded above by  $C_0\cdot(m-2)^{-4}$ for 
\[
C_0:= h^3
\left[\left(\frac{Y}{c}+\frac1{c^2}\right)
+
{112\cdot \frac{3a}{8}}
\left(
\frac{Y^2}{c}+\frac{2Y}{c^2}+\frac{2}{c^3}
\right)
\right]
\le 5(m-4)
\]
with the final inequality holding
for $m-4 \geq {156}$. 

Hence, for {$m\geq 200$, recalling that $c=-\Phi'(a)=2+\sqrt3$,} we obtain
\begin{equation}
    \label{eqn: M 2}
M_m^{\rm main} \le
h^3\left[\frac{H_m(a)}{(m-4)c}+
\frac1{(m-4)^2}\left(
\frac{H_m'(a)}{c^2}+\frac{H_m(a)\Phi''(a)}{c^3}
\right)
\right]
+\frac{5}{(m-4)^3}.
\end{equation}
Combining this with \eqref{eqn: tail bound} and \eqref{eqn: M err bound} gives \eqref{eqn: main m est}.\\

{\it Step 6:}  Finally, let us see how the estimate \eqref{eqn: main m est} implies the desired bound \eqref{eqn: M bound} on $M_m$. Note that {for any $m\geq 5$}
\begin{align*}
A:=h^3\left[\frac{H_m(a)}{(m-4)c}+
\frac1{(m-4)^2}\left(
\frac{H_m'(a)}{c^2}+\frac{H_m(a)\Phi''(a)}{c^3}
\right)
\right] &= \frac{H_m(a) h^3}{(m-4)c}\left[1 - \frac{1}{m-4}\left(\frac{4({m}a+1)}{{a(3{m}-4)}c} + \frac{-\Phi''(a)}{c^2}\right) \right] \\ 
&\leq  \frac{H_m(a) h^3}{(m-4)c}\left[1 - \frac{1}{m}\left(\frac{4(ma+1)}{(3m-4)c} + \frac{-\Phi''(a)}{c^2}\right) \right] \,.
\end{align*}
{In the final inequality we used $a\leq 1$.}
{Recalling that $c= \frac{h^2}{2}$,} we may then directly compute
\begin{align*}
    \frac{H_m(a) h^3}{(m-4)c}& = \frac{2h H_m(a)}{m}\left(\frac{m}{m-4}\right)
    = \frac{3}{2m}\left(1 +\frac{4}{m-4}\right)\left(1-\frac{1}{3(m-1)}\right) \leq \frac{3}{2m}\left(1 + \frac{4}{m-4}- \frac{1}{3(m-1)}\right)\,,
\end{align*}
so that, after first using this to crudely bound
\[
    -\frac{H_m(a) h^3}{(m-4)c} \cdot \frac{1}{m}\left(\frac{4(ma+1)}{(3m-4)c} - \frac{\Phi''(a)}{c^2}\right) \leq - \frac{3}{2m}\cdot \frac{1}{m}\left(\frac{1}{c}\frac{4(ma+1)}{3m-4} - \frac{\Phi''(a)}{c^2}\right)\,,
\]
which always holds for $m\geq 5$, we obtain
\begin{align*}
A & \leq \frac{3}{2m}\left[ 1 - \frac{1}{m}\left(\frac{4(ma+1)}{(3m-4)c} -\frac{\Phi''(a)}{c^2}\right) + \frac{4}{m-4} -\frac{1}{3(m-1)}\right]\\
&=\frac{3}{2m}\left[ 1 - \frac{1}{m}\left(\frac{4(ma+1)}{(3m-4)c} - \frac{\Phi''(a)}{c^2} - 4\frac{m}{m-4} +\frac{m}{3(m-1)}\right)\right]\,.
\end{align*}
{Recalling the additional error term from \eqref{eqn: main m est}, we observe that in order to conclude \eqref{eqn: M bound}, it remains to check that
\[
    \left(\frac{4(ma+1)}{(3m-4)c} - \frac{\Phi''(a)}{c^2} - 4\frac{m}{m-4}+\frac{m}{3(m-1)}\right) - \frac{2m^2}{3}\cdot \frac{2h+655}{(m-4)^3} \geq 4.1 
\]
for large enough $m$. This indeed holds for $m \geq 1300$.} \\

{\it Proof of \eqref{eqn: L bound}:}

We use the same notation of the proof of \eqref{eqn: M bound}. Letting $\Phi$ be as in \eqref{eqn: Phi def}, we note that 
\[
L_{m}=h^m\int_a^b e^{m\Phi(x)}x^{-1}\,dx,
\]
Recall that $-\Phi'(a)=\frac{h^2}{2}=2+\sqrt3$, so using the fact that $e^{\Phi(a)}=\frac{1}{h}$, the change of variables $x=a+t$ gives
\[
L_{m}=\int_0^{b-a} e^{m\Phi'(a)t}e^{{m}R(t)}(a+t)^{-1}\,dt,\]
where
\[
R(t):=\Phi(a+t)-\Phi(a)-\Phi'(a)t.
\]
As before, let $c:=-\Phi'(a)$, and let
\[
Z:=\frac{2\log m}{c}.
\]
Since $e^t$ and $(a+t)^{-1}$ are convex and $a^{-1} -a^{-2}t \geq0 $ for $t\in[0,Z/m]$ and any $m\geq 1$, we have 
\[
L_m \geq \int_0^{Z/m} e^{-mct}(1+mR(t))[a^{-1} -a^{-2}t]\,dt\,.
\]
As in Step 3, for $m\ge 703$, we have \(Z/m\le 0.005\), and for $t \in[0,0.005]$ we have
\[
R(t)\geq\frac12\Phi''(a)t^2- 1100t^3.
\]
Combining this with the change of variables $t=\tfrac{s}{m}$, and recalling that $a^{-1}=h$ and $\Phi''(a)<0$, we obtain
\begin{align*}
L_m &\geq \frac{1}{m}\int_0^Z e^{-cs}(1+ mR(s/m))\left(h -h^{2}\frac{s}{m}\right)\,ds\\
& {\geq}\frac{1}{m}\int_0^Z e^{-cs}\left(1+\frac{\Phi''(a)}{2}\frac{s^2}{m} - \frac{1100s^3}{m^2}\right)\left(h -h^2\frac{s}{m}\right)\,ds\\
& \geq \frac{h}{m}\int_0^Z e^{-cs}\,ds
+\frac{h}{m^2}\frac{\Phi''(a)}{2}\int_0^Z e^{-cs}s^2\,ds
-\frac{h^2}{m^2}\int_0^Z e^{-cs}s\,ds
-\frac{1100h}{m^3}\int_0^Z e^{-cs}s^3\,ds\,\\
&\geq \frac{h}{m}\int_0^Z e^{-cs}\,ds
+\frac{h}{m^2}\frac{\Phi''(a)}{2}\int_0^\infty e^{-cs}s^2\,ds
-\frac{h^2}{m^2}\int_0^\infty e^{-cs}s\,ds
-\frac{1100h}{m^3}\int_0^\infty e^{-cs}s^3\,ds\,.
\end{align*}
We can evaluate each of these four integrals explicitly. Keeping in mind \eqref{eqn: integrals} and recalling that the Gaussian moments satisfy \eqref{e:Gaussian-moments},
we have
\begin{align*}
L_m
 & \geq \frac{h}{mc} - \frac{he^{-cZ}}{mc}
+\frac1{m^2}\left(\frac{\Phi''(a) h}{c^3}
-\frac{h^2}{c^2}\right)
-\frac{6600h}{c^4m^3}\,.
\end{align*}
Since  $\frac{h}{c}=\frac2h=\sqrt3-1$ and $e^{-cZ} \leq 1$, multiplying by $m$ and adding $1$ yields
\begin{align*}
1+mL_m &\geq \sqrt{3} 
+\frac1{m}\left(\frac{\Phi''(a) h}{c^3}
-\frac{h^2}{c^2}\right)-\frac{1}{m^2}\left(
\frac{6600h}{c^4}{+ \frac{h}{c}}\right)\,\\
&\geq \sqrt{3} \left(1-\frac{1}{m}\left[\frac{1}{\sqrt{3}}\left(\frac{(-\Phi''(a))h}{c^3} +\frac{h^2}{c^2}\right)
+ \frac{1}{\sqrt{3}m}\left(
\frac{6600h}{c^4}+ \frac{h}{c}\right)\right]\right)
\,.
\end{align*}
Recalling that $\Phi''(a)=-56-32\sqrt3$, as well as the values of $c$, $a$ and $h$, a direct calculation gives
\[
    \left[\frac{1}{\sqrt{3}}\left(\frac{(-\Phi''(a))h}{c^3} +\frac{h^2}{c^2}\right)
+ \frac{1}{\sqrt{3}m}\left(
\frac{6600h}{c^4}+ \frac{h}{c}\right)\right] \leq 6.03-\frac4{\sqrt3}
\]
provided $m\geq 2000.$
This completes the proof.
\end{proof}

\subsection{Proof that $C_{3,3}\times \R$ satisfies \eqref{e:Lambda-cone-vs-Lambda-lens}}
We will proceed to construct a competitor produced from that constructed in \cite{BNNS} for $C_{3,3}$, extended suitably in the transverse direction to $C_{3,3}$ (i.e. the direction of the spine of the cone). In order to do this, let us first recall the procedure used to construct the $(SO(4)\times SO(4))$-invariant competitor $E$ for the set $K = \{(x,y) \in\R^{4} \times \R^{4} : |x|^2 = |y|^2\}$ whose boundary is $C_{3,3}$ (see \cite{BNNS}*{Section 7.2}).

    In this case, we define the ``slice" $\hat{E}$ of $E$ in a two-dimensional quadrant $Q =\{ (u,v) : u>0, v>0\}$, with $u=|(x_1,\dots,x_{4})|$ and $v= |(x_{5},\dots,x_{8})|$. Then
\[
    \partial K \cap Q = \{(u,v): u>0\,, \ v>0\,, \ u = v\}\,,
\] 
In $Q$, define $\hat E \cap \{u< v\}$ to be the circular arc of radius $r>0$ and center $(0, -h)$ forming angle $\pi/3$ with the line $\{u= v\}$ and angle $\pi/2$ with the $v$-axis $\{u=0\}$. Such an arc can be parameterized as the graph {over the $v$-axis} of the function $f: [1,-h+r] \to [0,1]$ given by $f(v) = \sqrt{r^2-(v+h)^2}$. Let $\theta = \tfrac{\pi}{4}$ denote the angle that the line $\{u= v\}$ makes with the $v$-axis $\{u=0\}$. 
Since $f(1) = 1 = \sqrt{r^2 -(1+h)^2}$ and $f'(1)= -\tan(2\pi/3-\pi/4) = -(1+h)$, this yields
\begin{equation*}
    h= \tan(2\pi/3-\pi/4) -1, \qquad 
    r = \sec(2\pi/3-\pi/4)\,.
\end{equation*}

The set $\hat E \cap \{u> v\}$ is defined analogously. Namely, we can also parameterize the arc of a circle of radius $\rho>0$ and center {$(-h,0)$} that forms angle $\pi/3$ with the line $\{u=v\}$ and angle $\pi/2$ with the $u$-axis $\{v=0\}$, via the function $u \mapsto f(u)$ over $[1,-h+r]$.

By the coarea formula, we then obtain
    \begin{align*}
        \nonumber   |E|  &= 16 \omega_{4}^2 \int_{\hat E} u^3 v^3 \, du\, dv 
        = \omega_{4}^2 \left(1+ 8\int_1^{\rho -d} v^3 \big(r^2 - (v+h)^2 \big)^{2} \,dv\right),\\
        \mathcal{H}^{7}(\partial E) &= 16 \omega_{4}^2\int_{\partial \hat E} u^3 v^3 \, d\Hcal^1 
        = 32\omega_{4}^2 r \int_{1}^{r-h} v^3 \big(r^2 -(v+h)^2 \big) \,dv\,.
    \end{align*}
    Letting $\X$ be the $(1,2)$-cluster with $\X(1) = E /|E|^{1/8}$, so that it has the same volume as $\X_{\rm lens}(1)$, with $\X(2) = K \setminus \X(1)$ and $\X(3) = (\R^8\setminus K)\setminus \X(1)$. For this cluster, we then consider the perimeter cost $M(3,3) = P(\X(1)) - \mathcal{H}^{7}(C_{3,3}\cap \X(1))$ of gluing $\X(1)$ into $C_{3,3}$, which can be expressed as
\begin{equation}\label{e:M(k,l)}
    M(3,3) = \frac{1}{|E|^{\frac{7}{8}}} \bigg(\Hcal^{7}(\partial E)-\frac{16\omega_{4}^2 \sqrt{2} }{7} \bigg)\,.
\end{equation}
We refer the reader to \cite{BNNS} for more details, noting that here we have $\lambda =1$, since $k=l=3$.

Let us now repeat this for the cone $C_{3,3}\times \R$, via an extension of the above competitor $E$ slice-wise in the transverse direction of the spine. More precisely, define a set $F \subset \R^{9}$ via
\[
    F = \{ (w,t) \in \R^{8}\times \R : w \in H(t) E\}\,,
\]
where $H:[-L,L]\to (0,\infty)$ is a Lipschitz function, to be determined. By the coarea formula, this yields
\begin{equation}\label{e:volume-coarea}
    |F| = |E| \int_{-L}^L H(t)^{8}\, dt\,.
\end{equation}
Meanwhile, letting $\Psi: \R^{9} \to \R^{9}$ denote the map $\Psi(x,t) := (H(t)x,t)$, we observe that $F = \Psi(E\times [-L,L])$, and $\partial F = \Psi(\partial E \times (-L,L))$, and 
\begin{align*}
    \Hcal^{8}(\partial F) &= \int\int_{\partial E \times (-L,L)} \Jbf^{\partial E \times (-L,L)}\Psi\, d\Hcal^7 \, dt \\
\end{align*}
where $\Jbf^{\partial E \times (-L,L)} \Psi$ is the tangential Jacobian of $\Psi$ over $\partial E \times (-L,L)$ (well-defined $\Hcal^{8}$-a.e.).

Let us first compute the differential of $\Psi$ at any point $(x,t)$ in the $\Hcal^{8}$-full measure set where the tangential Jacobian is well-defined. Since we may decompose $T_{(x,t)} (\partial E \times (-L,L))$ as $T_x \partial E \otimes T_t (-L,L) \cong T_x \partial E \otimes \R$, for any $\tau \in T_{x} \partial E$, we have
\[
    d\Psi_{x,t}(\tau,0) = (H(t)\tau,0) \qquad \text{and} \qquad d\Psi_{x,t}(0,1) = (H'(t)x,1)\,.
\]
To simplify computations, we observe that one may compute $\Jbf^{\partial E \times (-L,L)} \Psi$ as the volume of the $8$-dimensional parallelepiped spanned by the vectors $\{d\Psi_{x,t}(\tau_i,0)\}_{i=1}^{n}$, $d\Psi_{x,t}(0,1)$, for an orthonormal frame $\{\tau_i\}_{i=1}^n$ of $T_{x,t} \partial E$. This is simply the product of the area of the $7$-dimensional parallelepiped spanned by $\{H(t)\tau_i\}_{i=1}^7$ and the length of the projection of the vector $(H'(t)x,1)$ onto the orthogonal complement of the former span, namely,
\[
    H(t)^7 \cdot |\mathbf{p}_{d\Psi_{x,t}(\cdot,0)}^\perp(H'(t)x,1)|\,,
\]
where $\mathbf{p}_{d\Psi_{x,t}(\cdot,0)}^\perp$ is the projection onto the orthogonal complement of the span of $\{H(t)\tau_i\}_{i=1}^7$. Notice that the latter orthogonal complement is spanned by the orthonormal vectors $(\nu_{\partial E}(x),0)$ and $(0,1)$, where $\nu_{\partial E}(x)$ is the outward unit normal to $\partial^* E$ (defined $\Hcal^7$-a.e.) we thus have
\[
    |\mathbf{p}_{d\Psi_{x,t}(\cdot,0)}^\perp(H'(t)x,1)| = \sqrt{(H'(t))^2 (x\cdot \nu_{\partial E}(x))^2 + 1}\,.
\]
We thus arrive at
\[
    \Hcal^{8}(\partial F) = \int_{-L}^L H(t)^7 \int_{\partial E} \sqrt{(H'(t))^2 (x\cdot \nu_{\partial E}(x))^2 + 1} \, d\Hcal^7(x) \, dt\,.
\]
The coarea formula then yields
\begin{align*}
    \Hcal^{8}(\partial F) &= 16\omega_{4}^2 \int_{-L}^L H(t)^7 \int_{\partial \hat{E}} u^3 v^3 \sqrt{(H'(t))^2 ((u,v)\cdot \nu_{\partial \hat E}(u,v))^2 + 1} \, d\Hcal^1_{(u,v)} \, dt\,.
\end{align*}
Now, recalling the parameterization of $\partial\hat E$ via the two circle-arcs, we have
\[
    \nu_{\partial \hat E} = \begin{cases}
        \frac{(-f'(v),1)}{\sqrt{1+(f'(v))^2}} & \text{when $u=f(v)$, and} \\
        \frac{(-f'(u),1)}{\sqrt{1+(f'(u))^2}} & \text{when $v=f(u)$}\,,
    \end{cases}
\]
with
\[
    f'(v) = \frac{-(v+h)}{\sqrt{r^2-(v+h)^2}} \qquad \text{and} \qquad \sqrt{1+(f'(v))^2} = \frac{r}{\sqrt{r^2-(v+h)^2}}\,.
\]
Thus, when $u=f(v)$ we have
\begin{align*}
    (u,v)\cdot \nu_{\partial \hat E}(u,v) =  (f(v),v)\cdot \nu_{\partial \hat E}(f(v),v) = \frac{v(v+h)}{r} + \frac{f(v) \sqrt{r^2-(v+h)^2}}{r} =\frac{r^2-h(v+h)}{r} \,.
\end{align*}
Likewise,
\begin{align*}
    (u,f(u))\cdot \nu_{\partial \hat E}(u,f(u)) &=\frac{r^2-h(u+h)}{r} \,.
\end{align*}

In conclusion,
\begin{align*}
    \Hcal^{8}(\partial F) 
    = 32\omega_{4}^2 \int_{-L}^L H(t)^7 r\int_{1}^{r-h} v^3 (r^2-(v+h)^2) \sqrt{(H'(t))^2 \tfrac{(r^2-h(v+h))^2}{r^2}+1} \, dv\, dt\,.
\end{align*}
Just like we did for $C_{3,3}$, we may then let $\tilde\X$ be the $(1,2)$-cluster with $\tilde\X(1) = F /|F|^{1/9}$, with $\tilde\X(2) = (K\times \R) \setminus \tilde\X(1)$ and $\X(3) = (\R^9\setminus (K\times \R))\setminus \tilde\X(1)$. The normalized cost of gluing this competitor into $C_{3,3}\times \R$ is then
\begin{align*}
    \tilde{M}(3,3):= \frac{1}{|F|^{\frac{8}{9}}} \bigg(\Hcal^{8}(\partial F)-\frac{16 \omega_{4}^2 \sqrt{2} }{7} \int_{-L}^L H(t)^7\, dt\bigg)\,.
\end{align*}
Setting
\[
    H(t) = \left(1- \frac{|t|}{3}\right)_+
\]
with $L=3$, via an explicit calculation using Mathematica we obtain
\[
    \tilde{M}(3,3) \approx 7.79468\,.
\]
Recalling from \cite{BNNS} that 
\[
    \Lambda_{\plane} (9) \approx 7.93735360\,,
\]
we conclude that \eqref{e:Lambda-cone-vs-Lambda-lens} indeed holds for $C=C_{3,3}\times \R$. \qed

\begin{remark}
    We note that there is some flexibility with the choice of length scale $L=3$ for the above choice of $H$. Moreover, note that one can perform an analogous computation to the one above for general cylindrical cones over quadratic cones, i.e. $C_{k,l}\times \R^m$. However, we do not pursue this any further, since we believe that merely having one example of a cone with a non-isolated singularity for which \eqref{e:Lambda-cone-vs-Lambda-lens} holds is of interest, and there does not appear to be a single choice of $H$ that universally works in higher dimensions, nor a systematic way to look for an admissible choice of $H$ in higher dimensions. 
\end{remark}

\appendix
\section{Expansion of the area and graphicality for the Plateau problem}
\subsection{Expansion of the area functional}\label{ap:expansion}
We begin by using the area formula to write
    \[
        \Hcal^n(\graph_C^{s,r} u) = \int_{A_{s,r}} J_C u \, d\Hcal^n\,,
    \]
    where $J_C u$ is the tangential Jacobian of $h$ on $C$, given by 
    \[
        J_C u = \sqrt{\det[dH^t dH]}\,,
    \]
    where $H(x) = x + u(x)\nu_C(x)$. Letting $\{e_i\}_{i=1}^n$ denote an orthonormal frame for $TC$, we have
    \[
        dH(e_i) = e_i + \partial_i u\, \nu_C + u \nabla_{e_i} \nu_C = e_i + \partial_i u\, \nu_C - u A(e_i)\,,
    \]
    where $A(e_i) := -\sum_{j=1}^{n} (e_j \cdot \nabla_{e_i} \nu_C) e_j$, so that $A(e_i)e_j \equiv A(e_i,e_j) = - (e_j \cdot \nabla_{e_i} \nu_C)$ is the $ij$-coordinate of the second fundamental form $A$ of $C$ (previously denoted by $A_C$ but here we drop the dependence on $C$ to simplify notation) with respect to the bases $\{e_i\}$, where we identify $(TC)^\perp$ with $\R$. This yields
    \[
        dH^t dH = (I-uA)^2 + \nabla_C u\otimes \nabla_C u\,, 
    \]
    where we identify $A$ with its matrix representation. Thus, factoring out $I-uA$, which is possible for $\eps_0$ sufficiently small, we obtain
    \[
        \sqrt{\det[dH^t dH]} = \det(I-uA) \sqrt{1+ |(I- uA)^{-1} \nabla_C u|^2} \,.
    \]
    Recalling that
    \[
        \det(I - B) = 1 - \tr(B) + \tfrac{1}{2} (\tr(B)^2 - \tr(B^2)) - \tfrac{1}{6} (\tr(B)^3 - 3 \tr(B) \tr(B)^2 + 2 \tr(B^3)) + O(|B|^4)\,,
    \]
    and exploiting the minimality of $C$ together with a Taylor expansion of the square root and the observation that
    \[
        |(I- uA)^{-1} \nabla_C u|^2 = |\nabla_C u|^2 + 2u A(\nabla_C u, \nabla_C u) + O(u^2 |A|^2 |\nabla_C u|^2)\,,
    \]
    we thus have
    \begin{align*}
        \sqrt{\det[dH^t dH]} &= (1 - \tfrac{1}{2} |A|^2 u^2 - \tfrac{1}{3} \tr(A^3)u^3 + O(|A|^4 h^4)) \sqrt{1+ |(I- uA)^{-1} \nabla_C u|^2} \\
        &= \big(1 - \tfrac{1}{2} |A|^2 u^2 - \tfrac{1}{3} \tr(A^3)u^3 + O(|A|^4 u^4)\big)\big(1 + \tfrac{1}{2}|\nabla_C u|^2 + u A(\nabla_C u, \nabla_C u) + O(u^2 |A|^2 |\nabla_C u|^2)\big)\\
        &=: 1 + \tfrac{1}{2}(|\nabla_C u|^2 - |A|^2 u^2) + \Escr(u,\nabla_C u, A)\,.
    \end{align*}
    In particular, for $\eps_0$ sufficiently small 
    this yields \eqref{eq:remainder estimate} with
    \[
        |\Rcal_{s,r}(u)| \leq \int_{C\cap A_{s,r}} |\Escr(u,\nabla_C u, A)|\, d\Hcal^n \leq C_0 \int_{C\cap A_{s,r}} (|A|^3 u^3 + u |A||\nabla_C u|^2)\, d\Hcal^n\,.
    \]
    Recalling that $|A| \lesssim |x|^{-1}$ on $C\setminus \{0\}$, the claimed bound \eqref{e:remainder-cubic-bd} on $\Rcal_{s,r}$ follows.

    To see \eqref{eq:remainder first variation estimate} and \eqref{eq:closeness of second variations}, we simply observe that
    \begin{align*}
        \delta \Rcal_{s,r}(u)[\psi] &= \frac{d}{dt}\Big|_{t=0} \int_{C\cap A_{s,r}} \Escr(u+t\psi, \nabla_C u + t\nabla_C \psi, A)\, d\Hcal^n \\
        &= \int_{C\cap A_{s,r}} D_M \Escr(u,M,A)\big|_{M=\nabla_C u}\cdot \nabla_C \psi + D_v\Escr(v, \nabla_C u, A)\big|_{v=u}\cdot \psi\,d\Hcal^n\,,
    \end{align*}
    and thus, letting
\[
    \Gscr_1(u, \nabla_C u,A) := D_M \Escr(u,M,A)\big|_{M=\nabla_C u}\,, \qquad \Gscr_2(u,\nabla_C u, A):= D_v\Escr(v, \nabla_C u, A)\big|_{v=u}\,,
\]
we further have
\begin{align*}
        \delta^2 &\Rcal_{s,r}(u)[\psi_1,\psi_2] \\
        &= \frac{d}{dt}\Big|_{t=0} \int_{C\cap A_{s,r}} \Gscr_1(u+ t\psi_2, \nabla_C u + t\nabla_C \psi_2,A)\cdot \nabla_C \psi_1 + \Gscr_2(u+ t\psi_2, \nabla_C u + t\nabla_C \psi_2,A)\cdot \psi_2 \,d\Hcal^n \\
        &= \int_{C\cap A_{s,r}} D^2_M \Escr(u,M,A)\big|_{M=\nabla_C u} [\nabla_C\psi_1,\nabla_C \psi_2]\,d\Hcal^n \\
        &\qquad + \int_{C\cap A_{s,r}}  + D_v D_M \Escr(v,M,A)\big|_{v=u,M=\nabla_C u}\cdot (\psi_2 \nabla_C \psi_1 + \psi_1 \nabla_C \psi_2) + D^2_v \Escr(v,\nabla_C u,A)\big|_{v=u}\cdot  \psi_1\psi_2 \, d\Hcal^n \,.
    \end{align*}
Combining these with the fact that $|A| \lesssim |x|^{-1}$ on $C\setminus \{0\}$ and applying Cauchy-Schwarz in the case of \eqref{eq:closeness of second variations} yields the desired claim.

\subsection{Graphicality over $C$ for the Plateau problem}\label{ap:graphicality} In this section we consider the Plateau problem with data that is perturbative over $C$.

\begin{lemma}\label{lemma:small data dirichlet existence}
    Let $C \subset \mathbb{R}^{n+1}$ be an area minimizing hypercone with an isolated singularity and let $\lambda>0$. Let $G_\lambda$ denote the open set containing $C$ with $\partial G_\lambda = T_\lambda \cup T_{-\lambda}$. There exists $\delta>0$, $0<C_1<1$, $C_2>0$, and $R_0>0$ with the following property: if $C_1r\geq s\geq R_0$, $h\in C^\infty(C \cap \partial B_s)$ satisfies
\begin{align}\label{eq:invariant estimates}
  \||h|/s+|\nabla_{C\cap \partial B_s} h|+s|\nabla^2_{C\cap \partial B_s} h| \|_{L^\infty} &\leq \delta s^{-\gamma_1^--1} \qquad \mbox{and}\\ \label{eq:trapped by leaves}
  \graph_C h &\cc G_\lambda\,,
\end{align}   
$|\lambda'|<\lambda$, $v:C\to \mathbb{R}$ is such that $T_{\lambda'}=\graph_C v$ on $B_{R_0}^c$, and $S$ is an area-minimizing integral current with boundary data $\a{\graph_{C\cap \partial B_s} h}+\a{\graph_{C \cap \partial B_r} v}$, then there is $u:C\cap A_{s,r} \to \mathbb{R}$ such that $\spt \,S= \a{\graph_C^{s,r}u}$ and
\begin{equation}\label{eq:graph estimates}
    |Z_u(x)|\leq C_2 |x|^{-\gamma_1^--1}\qquad \forall x\in C \cap A_{s,r}\,,
\end{equation}
where we recall that $Z_u(x) = |\nabla_C u(x)| + |u(x)|/|x|$ as in the proof of Lemma \ref{l:area-gap-Plateau-leaf 2}.
\end{lemma}

\begin{proof}
Fix $C$ and $\lambda>0$.

We show that there are $\delta>0$, $R_0>0$, $C_1>0$, and $C_2>0$ such that given any boundary data as above and solution to the Plateau problem $S$ in the class of integral currents, $\spt\, S$ is a graph on the desired annulus satisfying the estimate \eqref{eq:graph estimates}. To prove this, it suffices to show that for any set of sequences $\delta_j\to 0$, $C_1^j\to 0$, $R_0^j\to \infty$, $r_j$ and $s_j$ with $C_1^j r_j \geq s_j \geq R_0^j$, $\lambda'_j<\lambda$ with $v_j$ such that $T_{\lambda'_j} = \graph_C v_j$ on $B_{R_0}^c$, $h_j$ satisfying \eqref{eq:invariant estimates} with $\delta_j\to 0$, and area-minimizing integral currents $S_j$ with boundary data $\a{\graph_{C\cap \partial B_{s_j}} h_j}+\a{\graph_{C \cap \partial B_{r_j}} v_j}$, there are $u_j$ such that the desired graphicality and estimate holds for large $j$. Note that by our assumption on the boundary data, the same barrier argument as in Lemma \ref{lem: trapping} implies that $\spt \, S_j \subset G_\lambda$.

First, by a straightforward area comparison between the leaf and a competitor constructed from $S_j$, we find $\Mbb(S_j)/r_j^n \leq \mathcal{H}^n(T_{\lambda'} \cap B_{r_j})/r_j^n + {\rm o}(1)$. Since $\mathcal{H}^n(T_{\lambda'} \cap B_{r_j})/r_j^n\to \mathcal{H}^n(C \cap B_1)$, $\partial S_j/r_j\to \partial (\a{C}\res B_1)$, and $\spt\, S_j/r_j \subset G_\lambda/r_j$, it follows that $S_j/r_j$ converges to $\a{C}\res B_1$ as currents. By combining this convergence with Allard's boundary regularity theorem \cite{allardbdry}, we deduce the existence of $\varepsilon_0$ such that $S_j/r_j$ is regular (with multiplicity one) on $B_{\varepsilon_0}(x)$ for every $x\in (T_{\lambda'}/r_j) \cap \partial B_1$ and large $j$. Moreover, the convergence of $S_j/r_j$ to $C$ implies that 
that the graphicality of $S_j/r_j$ over boundary tangent planes entailed by Allard's boundary regularity may be converted into graphicality over $C$ for large enough $j$. Rescaling by $r_j$, we conclude that $S_j \res B_{(1-\varepsilon_0/2)r_j}^c$ is a multiplicity one graph over $C$ for large $j$.  

Second, we claim that for sufficiently large $j$ and each $x\in C \cap B_{r_j-3} \setminus B_{s_j+3}$, $S_j \,\cap B_1(x)$ is a multiplicity one graph over $C$ that vanishes as $j\to \infty$. We begin by decomposing $S_j \res{B_2(x)}=\sum_{i=1}^{I^{j,x}} \partial \a{E_i^{j,x}}$ for each such $x$, where $E_i^{j,x}$ is a perimeter minimizer on $B_2(x)$. By the fact that $\spt\, S_j \subset G_\lambda$ and $G_\lambda$ decays to $C$ at infinity (and so decays uniformly to a disk on $B_2(y)$ as $|y|\to \infty$ along $C$), we may apply De Giorgi's $\eps$-regularity theorem for perimeter minimizers to conclude that $\partial E_i^{j,x} \cap B_1(x)$ is a graph (of a function with $L^\infty$ norm as small as one desires for $j$ large enough) over $C$ for all $j$ large and $x\in C \cap B_{r_j-3} \setminus B_{s_j+3}$. Next, note that for $j$ large enough so that this holds for all such $x$, the number $I^{j,x}$ of sheets on $B_1(x)$ is continuous in $x$. Now by the previous paragraph $I^{j,x}=1$ near $\partial B_{r_j} \cap C$. Also, $C \cap B_{r_j-3} \setminus B_{s_j+3}$ is connected due to the connectedness of $C\cap \partial B_1$ (which is a consequence of the maximum principle). Since a continuous, discrete-valued function on a connected set is constant, we deduce that $I^{j,x}=1$ for all $j$ large and $x\in C \cap B_{r_j-3} \setminus B_{s_j+3}$. This concludes the proof of the claim in this second step. 

To finish proving the graphicality, it remains to show that $\spt \, S_j$ is graphical near $\partial B_{s_j}$. Note that we cannot deduce it directly from the compactness along scales $r_j$, since it could be that $s_j/r_j \to 0$. However, we may instead consider the rescalings $S_j/s_j$ and notice that thanks to the argument in the previous paragraph, $S_j/s_j$ converges to $\a{C}\res B_1^c$ as currents. An application of Allard's boundary regularity theorem using this convergence as we did near $\partial B_{r_j}$ finishes the proof of graphicality. 

It remains to estimate $Z_{u_j}(x)$. The estimate for $u_j$ follows immediately from $\spt\, S_j \subset G_\lambda$. Also, note that a consequence of the graphicality argument is that $\|u_j\|_{C_1(C \cap A_{s_j,r_j})} \to 0$ by the convergence of the rescaled surfaces to $C$. Since $u_j$ and $\nabla_C u_j$ vanish as $j\to \infty$, the minimal surface operator on $C$ applied to $u_j$ is uniformly elliptic, and the desired estimates follow from flattening coordinates and invoking interior and boundary Schauder estimates. \end{proof}

\bibliographystyle{alpha}

\begin{bibdiv}
	\begin{biblist}
		
		\bib{AlaBroVri23}{article}{
			author={Alama, Stan},
			author={Bronsard, Lia},
			author={Vriend, Silas},
			title={The standard lens cluster in {$\mathbb{R}^2$} uniquely minimizes
				relative perimeter},
			date={2025},
			ISSN={2330-0000},
			journal={Trans. Amer. Math. Soc. Ser. B},
			volume={12},
			pages={516\ndash 535},
			url={https://doi.org/10.1090/btran/176},
			review={\MR{4888588}},
		}
		
		\bib{AlaBroLuWan22}{article}{
			author={Alama, Stanley},
			author={Bronsard, Lia},
			author={Lu, Xinyang},
			author={Wang, Chong},
			title={Core shells and double bubbles in a weighted nonlocal
				isoperimetric problem},
			date={2024},
			ISSN={0036-1410,1095-7154},
			journal={SIAM J. Math. Anal.},
			volume={56},
			number={2},
			pages={2357\ndash 2394},
			url={https://doi.org/10.1137/22M1538545},
			review={\MR{4717765}},
		}
		
		\bib{AlaBroLuWan25}{misc}{
			author={Alama, Stanley},
			author={Bronsard, Lia},
			author={Lu, Xinyang},
			author={Wang, Chong},
			title={Decorated phases in triblock copolymers: zeroth- and first-order
				analysis},
			date={2025},
			url={https://arxiv.org/abs/2503.21684},
		}
		
		\bib{allardbdry}{article}{
			author={Allard, William~K.},
			title={On the first variation of a varifold: boundary behavior},
			date={1975},
			ISSN={0003-486X},
			journal={Ann. of Math. (2)},
			volume={101},
			pages={418\ndash 446},
			url={https://doi.org/10.2307/1970934},
			review={\MR{397520}},
		}
		
		\bib{Alm76}{article}{
			author={Almgren, F.~J., Jr.},
			title={Existence and regularity almost everywhere of solutions to
				elliptic variational problems with constraints},
			date={1976},
			ISSN={0065-9266},
			journal={Mem. Amer. Math. Soc.},
			volume={4},
			number={165},
			pages={viii+199},
			url={https://doi-org.cmu.idm.oclc.org/10.1090/memo/0165},
			review={\MR{420406}},
		}
		
		\bib{BDGG}{article}{
			author={Bombieri, E},
			author={De~Giorgi, E},
			author={Giusti, E},
			title={Minimal cones and the bernstein problem},
			date={1969},
			journal={Inventiones mathematicae},
			volume={7},
			pages={243\ndash 268},
		}
		
		\bib{BonCriTop25}{article}{
			author={Bonacini, Marco},
			author={Cristoferi, Riccardo},
			author={Topaloglu, Ihsan},
			title={A stability inequality for planar lens partition},
			date={2025},
			journal={Proceedings of the Royal Society of Edinburgh: Section A
				Mathematics},
			pages={1–34},
		}
		
		\bib{BNNS}{misc}{
			author={Bronsard, Lia},
			author={Neumayer, Robin},
			author={Novack, Michael},
			author={Skorobogatova, Anna},
			title={On the non-uniqueness of locally minimizing clusters via singular
				cones},
			date={2025},
			url={https://arxiv.org/abs/2507.13995},
		}
		
		\bib{BN}{article}{
			author={Bronsard, Lia},
			author={Novack, Michael},
			title={An infinite double bubble theorem},
			date={2025},
			journal={Ann. Inst. H. Poincar\'{e} C Anal. Non Lin\'{e}aire},
		}
		
		\bib{ColEdeSpo22}{article}{
			author={Colombo, Maria},
			author={Edelen, Nick},
			author={Spolaor, Luca},
			title={The singular set of minimal surfaces near polyhedral cones},
			date={2022},
			ISSN={0022-040X},
			journal={J. Differential Geom.},
			volume={120},
			number={3},
			pages={411\ndash 503},
			url={https://doi-org.cmu.idm.oclc.org/10.4310/jdg/1649953512},
			review={\MR{4408288}},
		}
		
		\bib{Davini04}{article}{
			author={Davini, Andrea},
			title={On calibrations for {L}awson's cones},
			date={2004},
			ISSN={0041-8994},
			journal={Rend. Sem. Mat. Univ. Padova},
			volume={111},
			pages={55\ndash 70},
			review={\MR{2076732}},
		}
		
		\bib{DPP}{article}{
			author={De~Philippis, G.},
			author={Paolini, E.},
			title={A short proof of the minimality of {S}imons cone},
			date={2009},
			ISSN={0041-8994,2240-2926},
			journal={Rend. Semin. Mat. Univ. Padova},
			volume={121},
			pages={233\ndash 241},
			url={https://doi.org/10.4171/RSMUP/121-14},
			review={\MR{2542144}},
		}
		
		\bib{EdSp}{article}{
			author={Edelen, Nick},
			author={Spolaor, Luca},
			title={Regularity of minimal surfaces near quadratic cones},
			date={2023},
			ISSN={0003-486X,1939-8980},
			journal={Ann. of Math. (2)},
			volume={198},
			number={3},
			pages={1013\ndash 1046},
			url={https://doi.org/10.4007/annals.2023.198.3.2},
			review={\MR{4660135}},
		}
		
		\bib{EdSz}{article}{
			author={Edelen, Nick},
			author={Sz\'ekelyhidi, G\'abor},
			title={A {L}iouville-type theorem for cylindrical cones},
			date={2024},
			ISSN={0010-3640,1097-0312},
			journal={Comm. Pure Appl. Math.},
			volume={77},
			number={8},
			pages={3557\ndash 3580},
			url={https://doi.org/10.1002/cpa.22192},
			review={\MR{4764748}},
		}
		
		\bib{FoiAlfBroHodZim93}{article}{
			author={Foisy, Joel},
			author={Alfaro, Manuel},
			author={Brock, Jeffrey},
			author={Hodges, Nickelous},
			author={Zimba, Jason},
			title={The standard double soap bubble in {${\bf R}^2$} uniquely
				minimizes perimeter},
			date={1993},
			ISSN={0030-8730,1945-5844},
			journal={Pacific J. Math.},
			volume={159},
			number={1},
			pages={47\ndash 59},
			url={http://projecteuclid.org/euclid.pjm/1102634378},
			review={\MR{1211384}},
		}
		
		\bib{HardtSimon84}{article}{
			author={Hardt, Robert},
			author={Simon, Leon},
			title={Area minimizing hypersurfaces with isolated singularities},
			date={1985},
			ISSN={0075-4102,1435-5345},
			journal={J. Reine Angew. Math.},
			volume={362},
			pages={102\ndash 129},
			url={https://doi.org/10.1515/crll.1985.362.102},
			review={\MR{809969}},
		}
		
		\bib{HutMorRitRos02}{article}{
			author={Hutchings, Michael},
			author={Morgan, Frank},
			author={Ritor\'{e}, Manuel},
			author={Ros, Antonio},
			title={Proof of the double bubble conjecture},
			date={2002},
			ISSN={0003-486X,1939-8980},
			journal={Ann. of Math. (2)},
			volume={155},
			number={2},
			pages={459\ndash 489},
			url={https://doi.org/10.2307/3062123},
			review={\MR{1906593}},
		}
		
		\bib{Lawson72}{article}{
			author={Lawson, H.~Blaine, Jr.},
			title={The equivariant {P}lateau problem and interior regularity},
			date={1972},
			ISSN={0002-9947,1088-6850},
			journal={Trans. Amer. Math. Soc.},
			volume={173},
			pages={231\ndash 249},
			url={https://doi.org/10.2307/1996271},
			review={\MR{308905}},
		}
		
		\bib{MaggiBook}{book}{
			author={Maggi, Francesco},
			title={Sets of finite perimeter and geometric variational problems},
			series={Cambridge Studies in Advanced Mathematics},
			publisher={Cambridge University Press, Cambridge},
			date={2012},
			volume={135},
			ISBN={978-1-107-02103-7},
			url={https://doi.org/10.1017/CBO9781139108133},
			note={An introduction to geometric measure theory},
			review={\MR{2976521}},
		}
		
		\bib{MilNee23}{misc}{
			author={Milman, Emanuel},
			author={Neeman, Joe},
			title={Plateau bubbles and the quintuple bubble theorem on
				$\mathbb{S}^n$},
			date={2023},
		}
		
		\bib{MilNee22}{article}{
			author={Milman, Emanuel},
			author={Neeman, Joe},
			title={The structure of isoperimetric bubbles on {$\mathbb{R}^n$} and
				{$\mathbb{S}^n$}},
			date={2025},
			ISSN={0001-5962,1871-2509},
			journal={Acta Math.},
			volume={234},
			number={1},
			pages={71\ndash 188},
			url={https://doi.org/10.4310/acta.2025.v234.n1.a2},
			review={\MR{4877375}},
		}
		
		\bib{MilXu25}{misc}{
			author={Milman, Emanuel},
			author={Xu, Botong},
			title={Standard bubbles (and other {M}\"obius-flat partitions) on model
				spaces are stable},
			date={2025},
			url={https://arxiv.org/abs/2504.11185},
		}
		
		\bib{NovPaoTor23}{article}{
			author={Novaga, M.},
			author={Paolini, E.},
			author={Tortorelli, V.~M.},
			title={Locally isoperimetric partitions},
			date={2025},
			ISSN={0002-9947,1088-6850},
			journal={Trans. Amer. Math. Soc.},
			volume={378},
			number={4},
			pages={2517\ndash 2548},
			url={https://doi.org/10.1090/tran/9339},
			review={\MR{4880454}},
		}
		
		\bib{NovPaoTor25}{misc}{
			author={Novaga, M.},
			author={Paolini, E.},
			author={Tortorelli, V.~M.},
			title={Existence of a non-standard isoperimetric triple partition},
			date={In preparation},
		}
		
		\bib{PaTo20}{article}{
			author={Paolini, E.},
			author={Tortorelli, V.~M.},
			title={The quadruple planar bubble enclosing equal areas is symmetric},
			date={2020},
			ISSN={0944-2669,1432-0835},
			journal={Calc. Var. Partial Differential Equations},
			volume={59},
			number={1},
			pages={Paper No. 20, 9},
			url={https://doi.org/10.1007/s00526-019-1687-9},
			review={\MR{4048329}},
		}
		
		\bib{Rei08}{article}{
			author={Reichardt, Ben~W.},
			title={Proof of the double bubble conjecture in {$\mathbb{R}^n$}},
			date={2008},
			ISSN={1050-6926,1559-002X},
			journal={J. Geom. Anal.},
			volume={18},
			number={1},
			pages={172\ndash 191},
			url={https://doi.org/10.1007/s12220-007-9002-y},
			review={\MR{2365672}},
		}
		
		\bib{ReiHeiLaiSpi03}{article}{
			author={Reichardt, Ben~W.},
			author={Heilmann, Cory},
			author={Lai, Yuan~Y.},
			author={Spielman, Anita},
			title={Proof of the double bubble conjecture in {${\bf R}^4$} and
				certain higher dimensional cases},
			date={2003},
			ISSN={0030-8730,1945-5844},
			journal={Pacific J. Math.},
			volume={208},
			number={2},
			pages={347\ndash 366},
			url={https://doi.org/10.2140/pjm.2003.208.347},
			review={\MR{1971669}},
		}
		
		\bib{DRTi23}{article}{
			author={Rosa, Antonio~De},
			author={Tione, Riccardo},
			title={The double and triple bubble problem for stationary varifolds:
				the convex case},
			date={2025feb},
			journal={Trans. Amer. Math. Soc.},
			volume={378},
			number={5},
			pages={3393\ndash 3444},
			url={https://www.ams.org/tran/0000-000-00/S0002-9947-2025-09400-5/},
		}
		
		\bib{Simon_GMT}{book}{
			author={Simon, Leon},
			title={Lectures on geometric measure theory},
			series={Proceedings of the Centre for Mathematical Analysis, Australian
				National University},
			publisher={Australian National University, Centre for Mathematical Analysis,
				Canberra},
			date={1983},
			volume={3},
			ISBN={0-86784-429-9},
		}
		
		\bib{simon1987strict}{article}{
			author={Simon, Leon},
			title={A strict maximum principle for area minimizing hypersurfaces},
			date={1987},
			journal={Journal of differential geometry},
			volume={26},
			number={2},
			pages={327\ndash 335},
		}
		
		\bib{Simon_entire_MSE}{article}{
			author={Simon, Leon},
			title={Entire solutions of the minimal surface equation},
			date={1989},
			ISSN={0022-040X,1945-743X},
			journal={J. Differential Geom.},
			volume={30},
			number={3},
			pages={643\ndash 688},
			url={http://projecteuclid.org/euclid.jdg/1214443827},
			review={\MR{1021370}},
		}
		
		\bib{Simon_Liouville}{article}{
			author={Simon, Leon},
			title={A {L}iouville-type theorem for stable minimal hypersurfaces},
			date={2021},
			ISSN={2769-8505},
			journal={Ars Inven. Anal.},
			pages={Paper No. 5, 35},
			review={\MR{4462473}},
		}
		
		\bib{SimonSolomon}{article}{
			author={Simon, Leon},
			author={Solomon, Bruce},
			title={Minimal hypersurfaces asymptotic to quadratic cones in {${\bf
						R}^{n+1}$}},
			date={1986},
			ISSN={0020-9910,1432-1297},
			journal={Invent. Math.},
			volume={86},
			number={3},
			pages={535\ndash 551},
			url={https://doi.org/10.1007/BF01389267},
			review={\MR{860681}},
		}
		
		\bib{Sternberg1992}{article}{
			author={Sternberg, P.},
			author={Williams, G.},
			author={Ziemer, W.},
			title={Existence, uniqueness, and regularity for functions of least
				gradient.},
			date={1992},
			journal={Journal für die reine und angewandte Mathematik},
			volume={430},
			pages={35\ndash 60},
			url={http://eudml.org/doc/153442},
		}
		
		\bib{JTaylor76}{article}{
			author={Taylor, Jean~E.},
			title={Regularity of the singular sets of two-dimensional
				area-minimizing flat chains modulo {$3$} in {$R\sp{3}$}},
			date={1973},
			ISSN={0020-9910,1432-1297},
			journal={Invent. Math.},
			volume={22},
			pages={119\ndash 159},
			url={https://doi.org/10.1007/BF01392299},
			review={\MR{333903}},
		}
		
		\bib{Wang}{article}{
			author={Wang, Zhihan},
			title={Mean convex smoothing of mean convex cones},
			date={2024/02/01},
			journal={Geometric and Functional Analysis},
			volume={34},
			number={1},
			pages={263\ndash 301},
			url={https://doi.org/10.1007/s00039-024-00666-x},
		}
		
		\bib{Wichiramala}{article}{
			author={Wichiramala, Wacharin},
			title={Proof of the planar triple bubble conjecture},
			date={2004},
			ISSN={0075-4102,1435-5345},
			journal={J. Reine Angew. Math.},
			volume={567},
			pages={1\ndash 49},
			url={https://doi.org/10.1515/crll.2004.011},
			review={\MR{2038304}},
		}
		
	\end{biblist}
\end{bibdiv}

\end{document}